\documentclass{publmathdeb}

\usepackage{amsmath,amsfonts,amssymb, amsthm, vruler,ifthen, graphicx, algorithm, algorithmic, enumerate}
\theoremstyle{plain} 
\newtheorem{thm}{Theorem}
\newtheorem{lem}[thm]{Lemma} 
\newtheorem{conjecture}[thm]{Conjecture}
\theoremstyle{remark}
\newtheorem*{remark}{Remark}

\def\F{\mathcal{F}}

\def\bth{\begin{thm}}
\def\eth{\begin{thm}}
\def\bc{\begin{corollary}}
\def\ec{\end{corollary}}
\def\bcj{\begin{conjecture}}
\def\ecj{\end{conjecture}}

\newcommand{\bizveg}{{\hfill$\Box$}}

\author{Andr\'as Csernenszky} 

\address{Department of Computer Science\\
University of Szeged\\}

\email{csernenszkya@gmail.com}

\title[The Chooser-Picker 7-in-a-row-game]
      {The Chooser-Picker 7-in-a-row-game}



\submitted{October 05, 2008}

\begin{document}

\begin{abstract}
One of the main objective of this paper is to relate Beck's conjecture for
$k$-in-a-row games. The conjecture states that playing on the same board 
Picker is better off in a Chooser-Picker game than the second player in the 
Maker-Breaker version. It was shown that the 8-in-a-row game is a blocking 
draw that is a Breaker win. To give the outcome of 7-, or 6-in-a-row-games 
is hopeless, but these games are widely believed to be Breaker's win.
If both conjectures hold, Picker must win the Chooser-Picker version of the
$7$-in-a-row game, and that is what we prove.

\end{abstract}

\maketitle

\section{Introduction}

The well-known games of Tic-Tac-Toe, Hex or the 5-in-a-row suggest the 
generalizations as follows, see more in \cite{B-C-G, BB}. A {\em hypergraph}
is a pair $(V, \F)$, where $V$ is a set and $\F \subset 2^V$.
Finally, there are two players in the game that we call {\em first} and 
{\em second.}

The first and second players take elements of $V$ in turns. The player, who 
takes all elements of an edge $A \in \F$ first wins the game. This version 
is sometimes called {\em Maker-Maker} version.  

The so-called {\em Maker-Breaker} version of a Positional Game or 
{\em weak game} on a hypergraph $(V, \F)$ was also investigated from 
the very beginning. 
Here the moves are defined just as before, but winning conditions are 
different. Maker wins by taking all elements of an $A \in \F$, while 
Breaker wins otherwise.
In several cases Maker-Breaker games are more tractable than Maker-Maker 
games. On the other hand, these versions are closely related, since if 
Breaker wins as a second player then the original game is a draw \cite{beckcpc05}, while 
if the first player wins the original game then Maker also wins the 
Maker-Breaker version. This connection gives rise to very useful 
applications, see in \cite{Beck5, beckcpc05, BB, Zetters, H-J}.

In order to understand the very hard clique games Beck introduced the 
Chooser-Picker (and the Picker-Chooser) version of the Maker-Breaker 
games in \cite{Beck5}.  
In these versions Picker selects two vertices of the hypergraph then Chooser 
takes one of them while the other vertex goes back to Picker. Similar to the 
Maker-Breaker games the Chooser wins in the Chooser-Picker game 
if he occupies a whole winning set, and the Picker wins if he can prevent Chooser's win. 
When $|V|$ is odd, the last element goes to Chooser. 
In Picker-Chooser games the winning conditions 
are swapped, and when $|V|$ is odd, the last element goes to Picker.

The study of these games gives invaluable insight to the Maker-Breaker version. 
For some hypergraphs the outcome of the Maker-Breaker and Chooser-Picker
versions is the same \cite{Beck5, CMP}. In all cases it seems that Picker's
position is at least as good as Breaker's. It was formalized in the
following conjecture.

\begin{conjecture} [Beck] \label{Bc}
If Maker (as the second player) wins the 
Maker-Breaker game, then Picker wins the corresponding Picker-Chooser game. 
If Breaker (as the second player) wins the Maker-Breaker game, then also Picker 
wins the Chooser-Picker game.\cite{CMP}
\end{conjecture}

One can ask what is the use of such a conjecture? Usually it is easier to 
analyze a Chooser-Picker game than the corresponding Maker-Breaker game. 
So if we think that Maker wins a weak game, then to confirm it we first
check the Picker-Chooser version, and we must see that Picker wins. Again, 
if Breaker's win is expected, then the Chooser-Picker version should be a 
Picker's win.

Here our main goal is to prove that Picker wins a specific game, the Chooser-Picker $7$-in-a-row
that will be defined in Section~\ref{amobaresz}. This result has two possible
interpretations. It strengthens both Conjecture~\ref{Bc} and the general 
belief that Breaker wins the Maker-Breaker version of the $7$-in-a-row game
(and therefore the Maker-Maker is a draw). 

The rest of the paper is organized as follows. In Section~\ref{tools} we 
recall some computationally useful facts about general Chooser-Picker games.
It is also necessary to extend the Chooser-Picker games to infinite hypergraphs.
In Section~\ref{amobaresz} we define the $k$-in-a-row games and list of the
basic results. The Section~\ref{cut} contains the plan how to prove Picker's
win in the game $7$-in-a-row, or similar games. Finally we carry out the details
in the Appendix.

\section{On Chooser-Picker games} \label{tools}

Here we list some simple facts from \cite{CMP} that are very useful in analyzing 
concrete games. For the sake of completeness we give the proofs, too.

\begin{lem} \cite{CMP} \label{parositas} 
If in the course of the (Chooser- Picker) game (or just already at the beginning) there is a two 
element winning set $\lbrace x, y\rbrace$ then Picker has an optimal strategy 
starting with $\lbrace x, y\rbrace$ .
\end{lem}

\begin{proof}
It is enough to see that if Picker has a winning strategy $p$, 
then there exists a starting with $\lbrace x, y\rbrace$  -- call it $p^{*}$ which is 
also Picker win.

If the strategy $p$ asks later $\{x,y\}$:
Assume that playing $p^{*}$ Chooser can win on given distribution of $\lbrace x, y\rbrace$, 
than Chooser could pretend that this distribution already happened before. In this 
way playing strategy $p$ also could use the same strategy. 

If in strategy $p$ Picker compelled to ask not at once $x$ and $y$:

Then Chooser could both ask $x$ and $y$ when they are separately turns up with 
other elements (or one of these is the remaining one for Chooser) in strategy $p$.  
And therefore Chooser wins.
\end{proof}

It looks very desirable to extend such a successful heuristic to games
played on infinite hypergraphs. 
However, one has to be careful since in that case Picker might offer a set of
vertices $A \subset V$ such that every edge contain at most one element from $A$,
which is a trivial winning strategy for Picker.

A possible remedy is add a step at the beginning: Chooser selects a finite set 
$X \in V$, and they play on the {\em induced sub-hypergraph} that is keep only
those edges $A \in \F$ for which $A \subset X$. 
More formally, given the hypergraph $(V, \F)$ let $(V \setminus X, \F(X))$ 
denote the hypergraph 
where $\F(X)=\{A \in \F, A \cap X =\emptyset\}$.

\begin{lem}  \cite{CMP} \label{monotonitas} 
If Picker wins the Chooser-Picker game on $(V, \F)$, then Picker also wins 
it on $(V \setminus X, \F(X))$.
\end{lem}

\begin{proof} By induction it is enough to prove the statement 
for $X=\{x\}$, i.\ e., $|X|=1$.
Assume that $p$ is a winning strategy for Picker in the game on $(V, \F)$.
That is in a certain position of the game the value of the function $p$ is 
a pair of unselected elements that Picker is to give to Chooser. We can 
modify $p$ in order to get a winning strategy $p^*$ for the Chooser-Picker 
game on $(V \setminus \{x\}, \F(\{x\}))$. 

Let us follow $p$ while it does not give a pair $\{x, y\}$. Getting a pair
$\{x, y\}$, we ignore it, and pretend we are playing the game on $(V, \F)$,
where Chooser has taken $y$ and has returned $x$ to us. If $|V|$ is odd, there
is a $z \in V$ at the end of the game that would go to Chooser. Here Picker's
last move is the pair $\{y, z\}$. Picker wins, since Chooser could not win from
this position even getting the whole pair $\{y, z\}$.
If $|V|$ is even, $p^*$ leads to a position 
in which $y$ is the last element, and it goes to Chooser. But the outcome is
then the same as the outcome of the game on $(V, \F)$, that is a Picker's win.
\end{proof}

\section{The k-in-a-row game} \label{amobaresz}

The k-in-a-row game is that hypergraph game, where the vertices of the graphs 
are the fields of an infinite graph paper ($\mathbb{Z}^{2}$), and the winning sets are the consecutive 
cells (horizontal, vertical or diagonal) of length $k$.   
If one of the players gets a length $k$ line, then he wins otherwise the game is 
draw. Note the assuming perfect play, the winner is always the first player, or it
is a draw, by the strategy stealing argument of John Nash, \cite{B-C-G}. 
More details about $k$-in-a-row games in \cite{PA4,PA2}.

The board of the classical (Maker-Maker) 5-in-a-row game is a graph paper or 
the $19 \times 19$ Go board\footnote{The Go-Moku rules differs from the 5-in-a-row,
see e.\ g. \cite{B-C-G}.}, and the players' goal is to get five
squares in a row vertically, horizontally or diagonally first.

It is easy to see that the first player wins if $k \leq 4$, and a delicate case 
study by Allis \cite {Allis} shows that the first player wins for $k=5$ on 
the $19 \times 19$ or even in the $15 \times 15$ board. While the case $k=5$ is 
still open on the infinite board, Allis' result implies that Maker wins for $k=5$ 
in the Maker-Breaker version.

The game is a {\em blocking draw}, i.\ e. Breaker wins the Maker-Breaker version
a if $k \geq 9$, proved first by Shannon and Pollak, and later even a pairing 
strategy was given, \cite{beckcpc05, B-C-G}. Finally a Breaker's win was published
by A. Brouwer under the pseudo name T.\ G. L. Zetters for $k=8$, \cite{Zetters}. 

Both the Maker-Maker and the Maker-Breaker versions of the $k$-in-a-row for 
$k=6, 7$ are open. These are wisely believed to be draws (Breaker's win) but, 
despite of the efforts spent on those, not much progress has been achieved.

\section{The Chooser-Picker 7-in-a-row game} \label{cut}

\begin{thm}\label{CP7}  
Picker wins the Chooser-Picker 7-in-a-row game on every A subset of $\mathbb{Z}^{2}$.
\end{thm}

Let us start with the strategy of the proof. 
By applying the remedy mentioned before Lemma~\ref{monotonitas}
at first Chooser determines the finite board $S$. We will consider a tiling 
of the entire plane, and play an auxiliary game on each tile (sub-hypergraph). 
It is easy to see, if Picker wins all of the sub-games, then Picker wins the 
game played on any $K$ board which is the union of disjoint tiles. 

Let $K$ be the union of those tiles which meet $S$. Since $S \subset K$, 
Lemma~\ref{monotonitas} gives that Picker also wins the game on $S$, too.

Now we need to show a suitable tiling and to define and analyze the auxiliary 
games. The tiling guarantees that if Picker wins on in each sub-games then Chooser 
cannot occupy any seven consecutive squares on $K$.

Each tile is a $4 \times 8$ sized rectangle and the winning sets, for the sake 
of better understanding, are drawn on the following four board:

\begin{figure}[ht]
\centering
\includegraphics[bb=0mm 0mm 208mm 296mm, width=100mm, height=40mm, viewport=3mm 4mm 200mm 290mm]{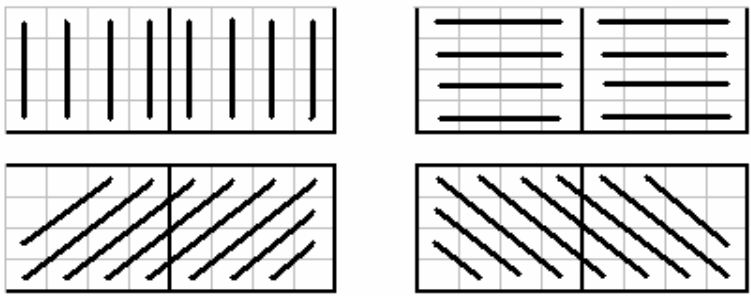}
\caption{These are the winning-sets of the $4 \times 8$ rectangle. Easy to see, that there is exactly 
one symmetry (along the double line). Later we will make use of it.}
\label{cp7_1}
\end{figure}

\begin{figure}[ht]
\centering
\includegraphics[bb=0mm 0mm 208mm 296mm, width=100mm, height=50mm, viewport=3mm 4mm 205mm 292mm]{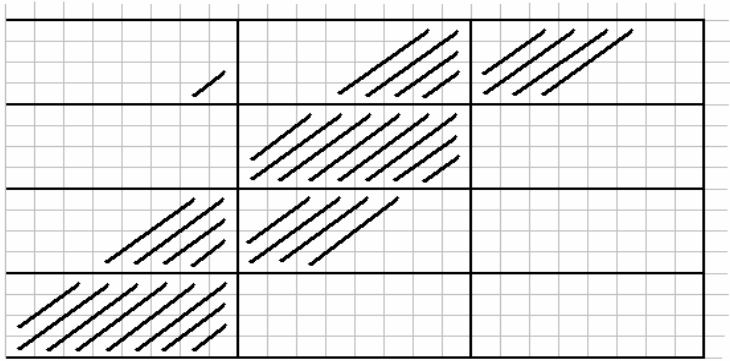}
\caption{We can see, how to draw from playing on simple tile, the game played on the infinite 
chessboard: neither vertically, nor horizontally, nor diagonally  (there is only one diagonal 
direction detailed) there are no seven consecutive squares without containing one winning set of a sub-game.}
\label{cp7_2} 
\end{figure}

The key lemma for our proof is the following.

\begin{lem}\label{4x8} 
Picker wins the auxiliary game defined on the $4 \times 8$ rectangle.
\end{lem}

Before starting the proof of Lemma~\ref{4x8}, let us estimate the actual complexity 
of the Maker-Breaker and Chooser-Picker versions of the auxiliary ``$4 \times 8$-game".
 
In the Maker-Breaker version Maker has $32$ possible moves, then Breaker has $31$,
so clearly the (unpruned) game-tree has size $32!$. Even worse, it may be hard to write
down convincing evidence of the outcome after searching this tree. In the Chooser-Picker
case, provided that Picker win, there is always a much shorter proof of the outcome.
Picker exhibits two squares and depending on Chooser move, only two smaller games 
have to be searched. This leads to a game-tree of size $2^{16}$, which is reasonable
to search. (Note that if Chooser wins a C-P game, the verification can be as hard
as the one in the M-B version.)

With some consideration the length of the case-study of the C-P version can be reduced.
One tool of this is a classification on the partially filled tables.
Let us denote the squares of a board $T$ taken by Chooser or Picker by $T_C$ and $T_P$,
respectively. From Picker's point of view the table $T$ is more dangerous than the table 
$T'$ $(T > T')$ if $T'_C \subset T_C $ and $T_P \subset T'_P$. 
Thus if Picker has a $p$ winning strategy on $T$, as a consequence of 
Lemma~\ref{monotonitas} playing the modified $p$ Picker also wins on $T'$.
See the application in \ref{sub}. 

An other gain is that Picker can ask an appearing two length winning set immediately by 
Lemma~\ref{parositas}. (In the defined $4 \times 8$ auxiliary game there are two such pairs at the beginning 
already, and some appears later.) 

Finally, we do not always have to go down to the leaves to the game-tree, since an
appropriate pairing strategy may prove Picker's win in an inner vertex of that tree. \\

\textbf{\underline{Our plan is for proving the  key lemma is:}}

\begin{enumerate}[I]
\item Sepatereing cases: A) and B) type cases.
\item Filling up one side of the auxiliary table useing breath first search.
\item After a case classification filling up the other side.
\end{enumerate}

\subsection{The proof of the key lemma}

The course of the proofe is: We take the 2 piece of 2 length winning set.
Picker picks them at the beginning (Picker can do this without any disadvantage 
thanks to the Lemma~\ref{monotonitas}). 
Depending on Chooser selection, there are two cases:

A) Chooser gets the upper square at least one side.

B) Both side Chooser gets the lower ones.

\begin{figure}[ht]
\centering
\includegraphics[bb=0mm 0mm 208mm 296mm, width=115.9mm, height=27.8mm, viewport=3mm 4mm 205mm 292mm]{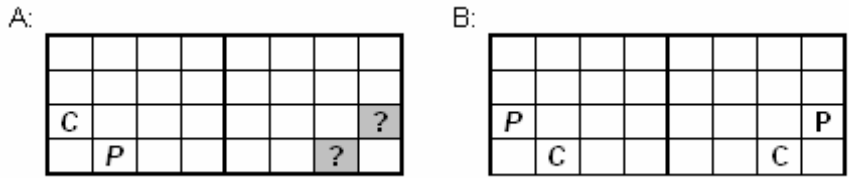}
\caption{The two cases: case \textbf{ A)}, and case \textbf{ B)}.}
\label{cp7_3} 
\end{figure}

From the characteristic of the game, the tree which describes the game is binary-tree,
we should walk breadth first on the cases; thus the cases with the same parent are beside 
each other. (After Picker's move, these are the two possible choices of Chooser). 
According to this, the positions are indexed lexicographically:
\textbf{\textit{A, B, a, b, i, ii, I, II}}\ldots

Somehow remarkable, that we the use breath first search to write down the proof, but to find the 
value of the game we always use deep first search algorithms (because we have to know the outcome 
of the game, before we are looking to the next branch of the tree).

\subsubsection{case A) }\label{sub}  

Without loss of generality, we may assume Chooser occupies the upper square on the left side 
(there might be the same on the right side).
Now Picker's strategy is to fill up the left side and leave the least possible crossing winning 
sets are left alive (see more detailed at \textbf{Index\_}\textbf{A}). On the pictures in the 
Appendix the special marks like =, *, +, etc. are the pairs to be asked by Picker. At that stage 
it makes no difference which squares are chosen by Chooser. And those marks also indicates the 
ending of a branch of the game-tree.

It is convenient to introduce to a new notation: It helps Chooser's game if we change one of 
Picker's square to a free square, and it is also advantageous for Chooser if he/she gets one 
of the free squares (P $\ll$ FREE $\ll$ C). It means that it is not necessary to prove a case 
if there exists a more dangerous one.   

From both Picker and Chooser point of view a square (be occupied or not) is 
\textit{uninteresting}\textit{,} if each winning set which contains it is ``dead''. It does not 
change the outcome of the game if we give these squares to Picker. 

Then after filling up the left side we can create equivalence classes using the relations and 
arrangements above. In this way one have to consider seven cases only:

\begin{figure}[ht]
\centering
\includegraphics[bb=0mm 0mm 208mm 296mm, width=116.7mm, height=108.7mm, viewport=3mm 4mm 205mm 292mm]{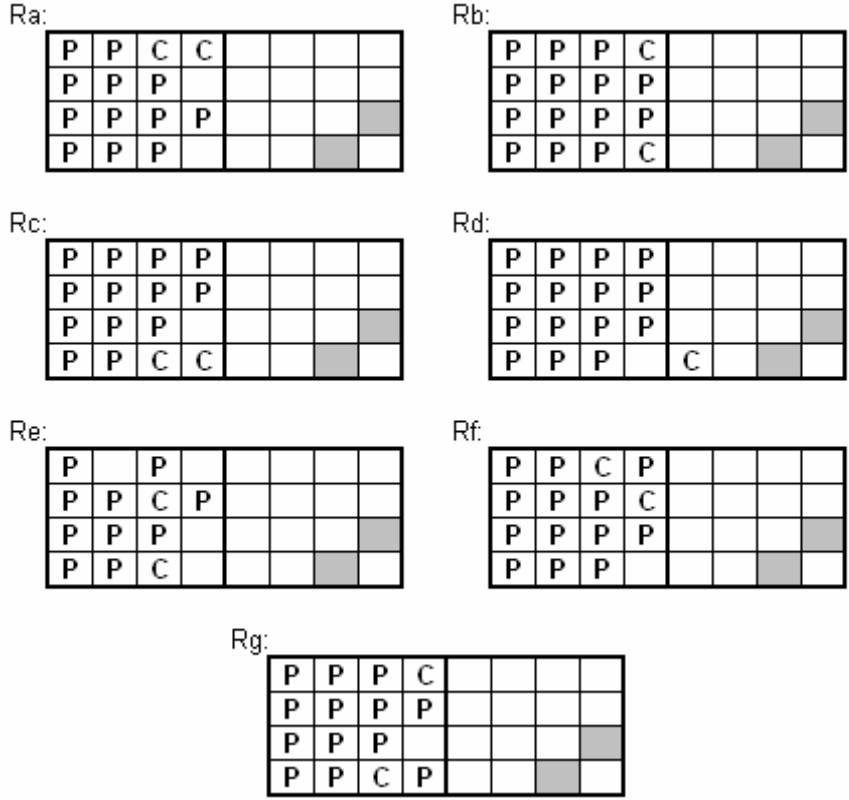}
\caption{If on the ``left side'' matching Chooser occupies the ``upper'' square, than Picker can 
achieve one of this stages (or an equivalent or less perilous position, using arrangement above).}
\label{cp7_4} 
\end{figure}

The finish of these positions above (=filling up the right side) can be found detailed in 
the Appendix of the paper (see Index$_{Ra}$, Index$_{Rb}$, ...Index$_{Rg}$.)

\subsubsection{case B) }

If case B) happens, then Picker asks the following two matching (see below), hence, using the symmetry,
it is enough to examine the following three cases.

\begin{figure}[ht]
\centering
\includegraphics[bb=0mm 0mm 208mm 296mm, width=116.9mm, height=63.3mm, viewport=3mm 4mm 205mm 292mm]{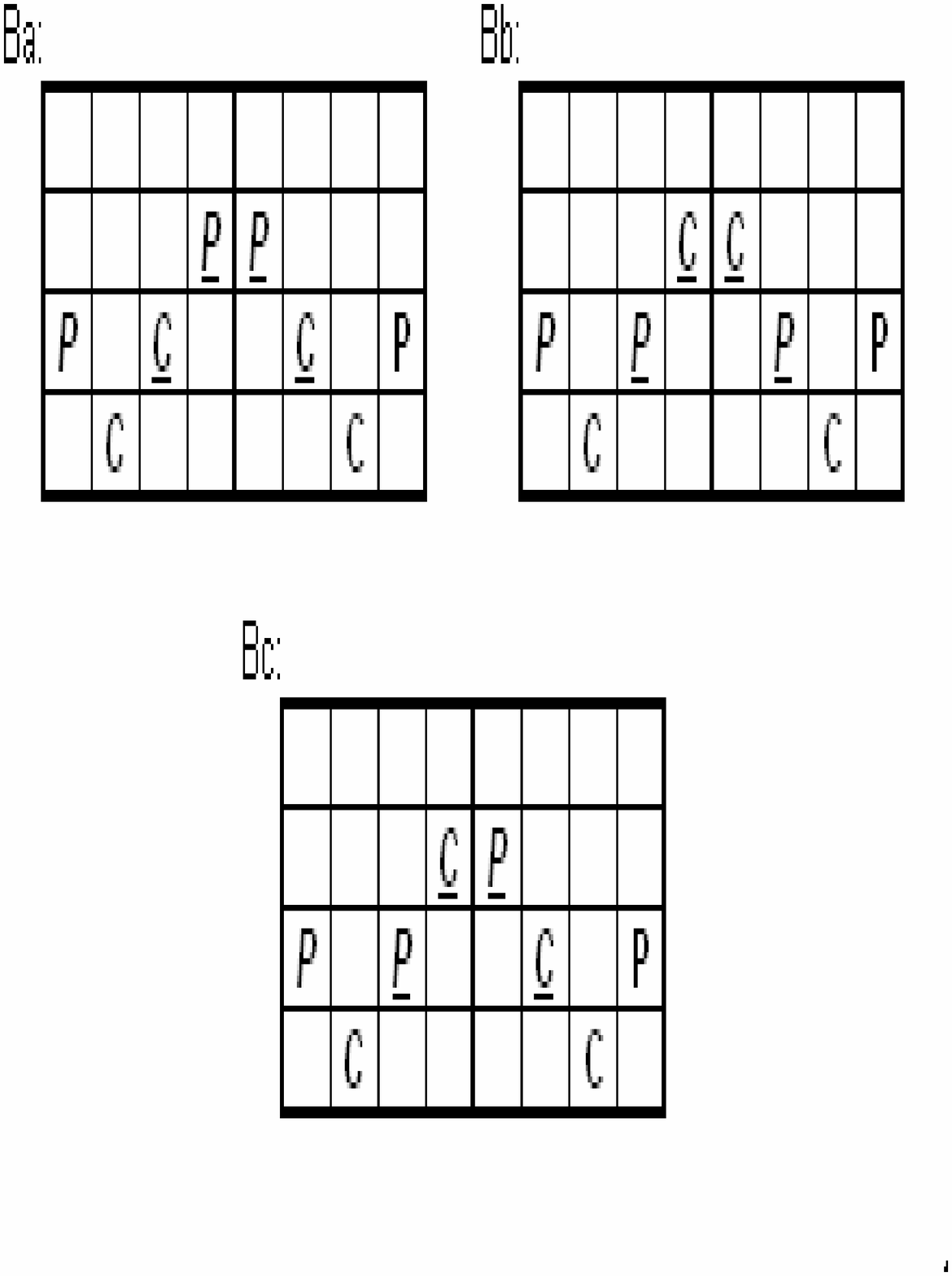}
\caption{If at the beginning Chooser takes the ``lower'' squares on both side, then Picker asks 
the two colored pair of squares in the middle. It gives rise to three cases.} 
\label{cp7_5} 
\end{figure}

The results of the three cases are also detailed at the end of the paper (see
\textbf{Index\_Ba, Index\_Bb, Index\_Bc}), that concludes the proof of Lemma~\ref{4x8} and
consequently Theorem~\ref{CP7}. \bizveg

\begin{remark}
We checked with brute force computer search the M-B game on the same auxiliary board, 
but it is a Maker win! So we cannot use the same table again, to prove that the weak 
version (=the Maker-Breaker version) of this game is a Breaker win. One is tempted to
look for other auxiliary games, which is not going to be easy. As a rule of thumb, it 
always good idea to check the C-P version of these games at first. 
\end{remark}

\end{document}


\begin{center}
\bf{\LARGE{Appendix} of the A. Csernenszky, 
The Chooser-Picker 7-in-a-row-game,  which appears in {\em Publicationes Mathematicae} (2010)}
\end{center}

\begin{center}

\includegraphics[bb=0mm 0mm 208mm 296mm, width=120mm, height=72mm, viewport=3mm 4mm 205mm 292mm]{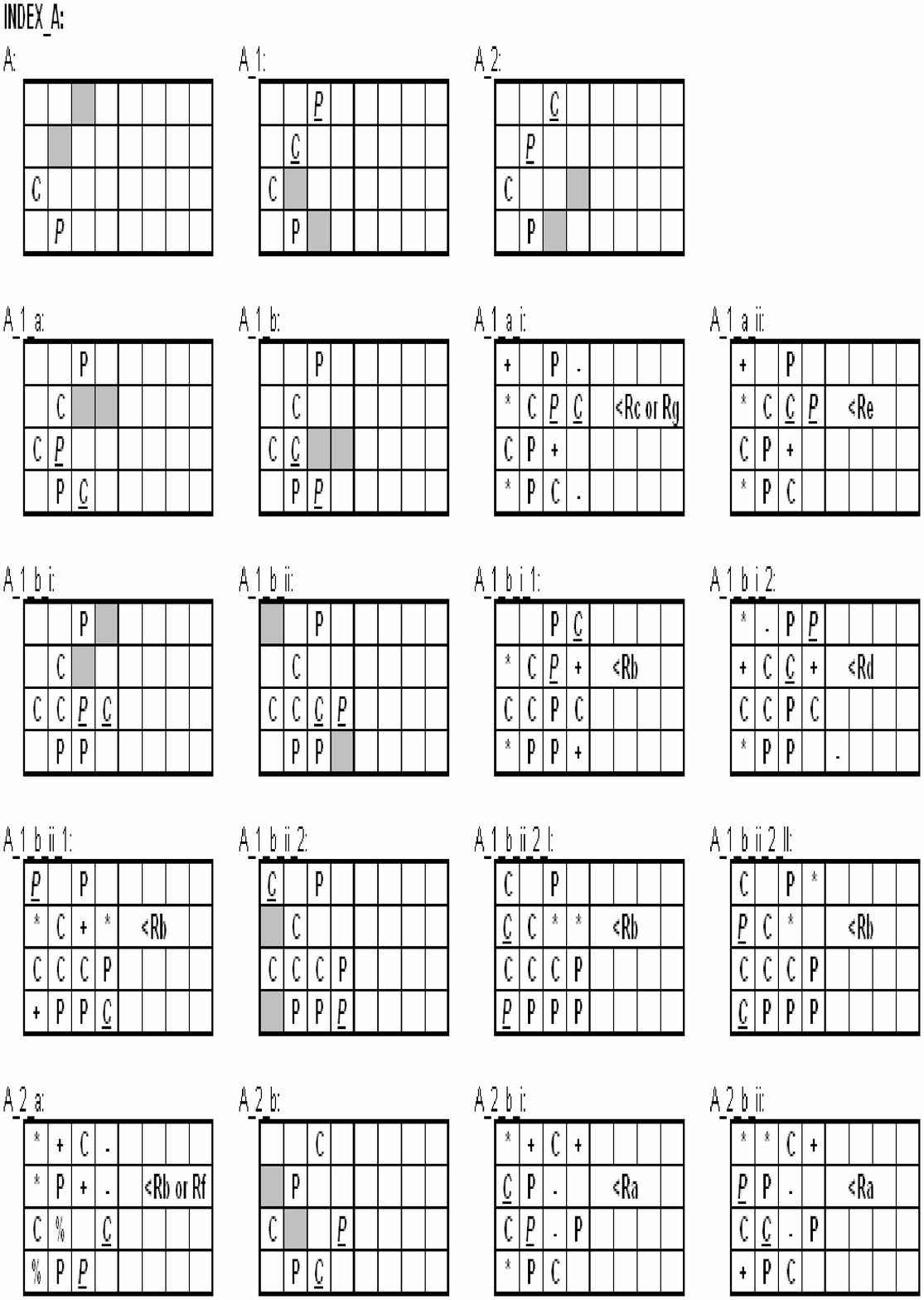}

\includegraphics[bb=0mm 0mm 208mm 296mm, width=120mm, height=22mm, viewport=3mm 4mm 205mm 292mm]{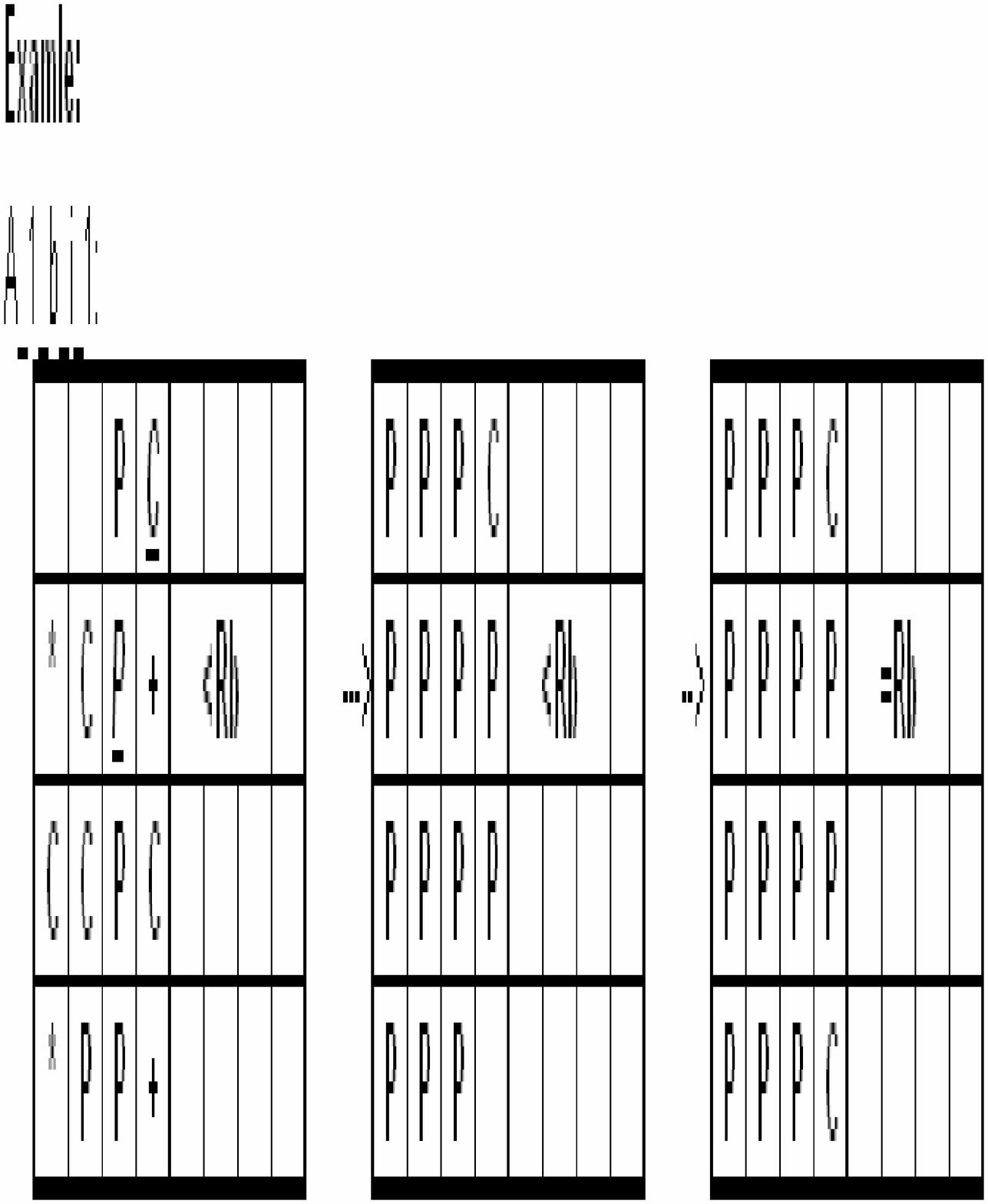}

\includegraphics[bb=0mm 0mm 208mm 296mm, width=120mm, height=20mm, viewport=3mm 4mm 205mm 292mm]{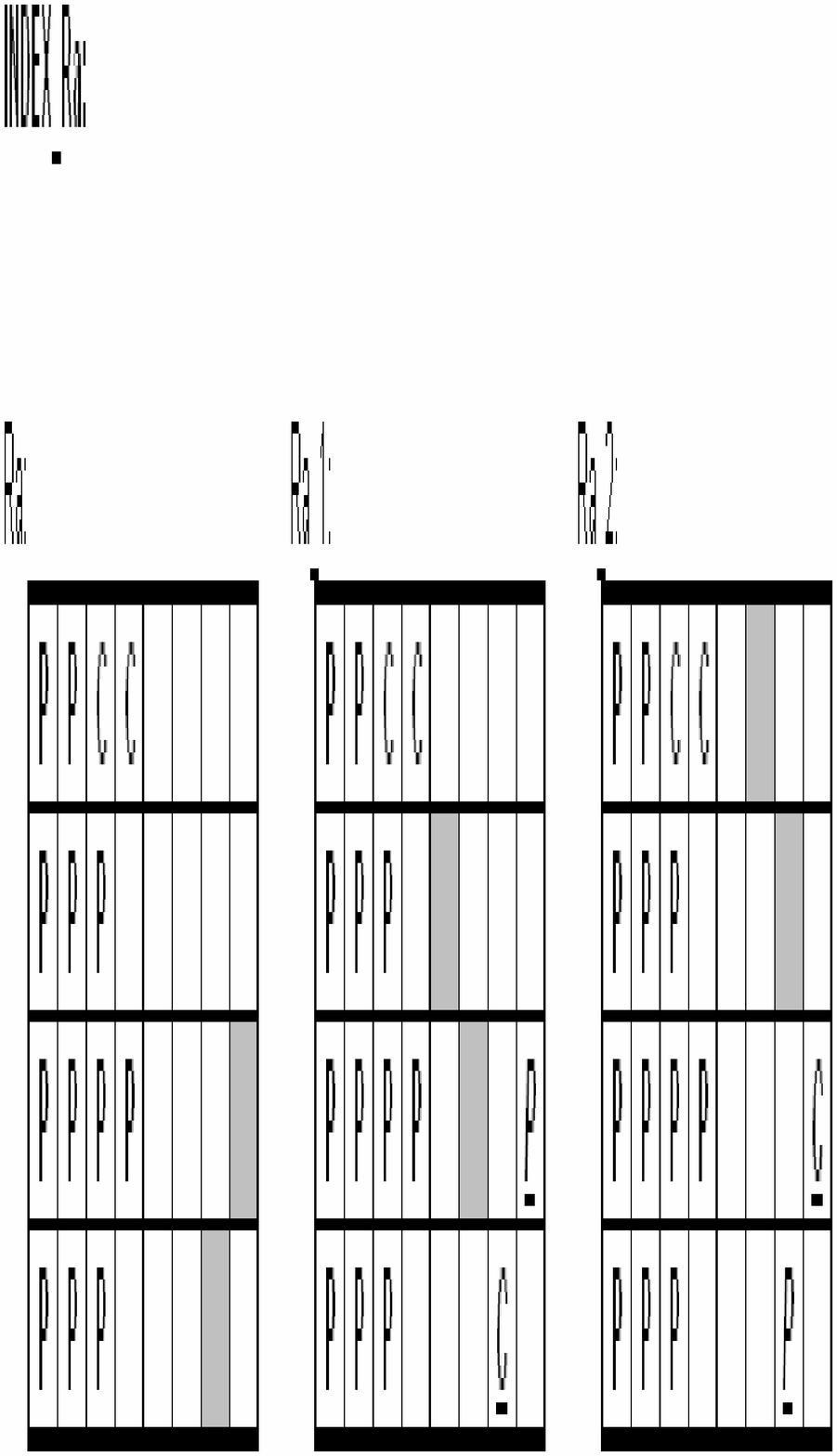}

\includegraphics[bb=0mm 0mm 208mm 296mm, width=120mm, height=16mm, viewport=3mm 4mm 205mm 292mm]{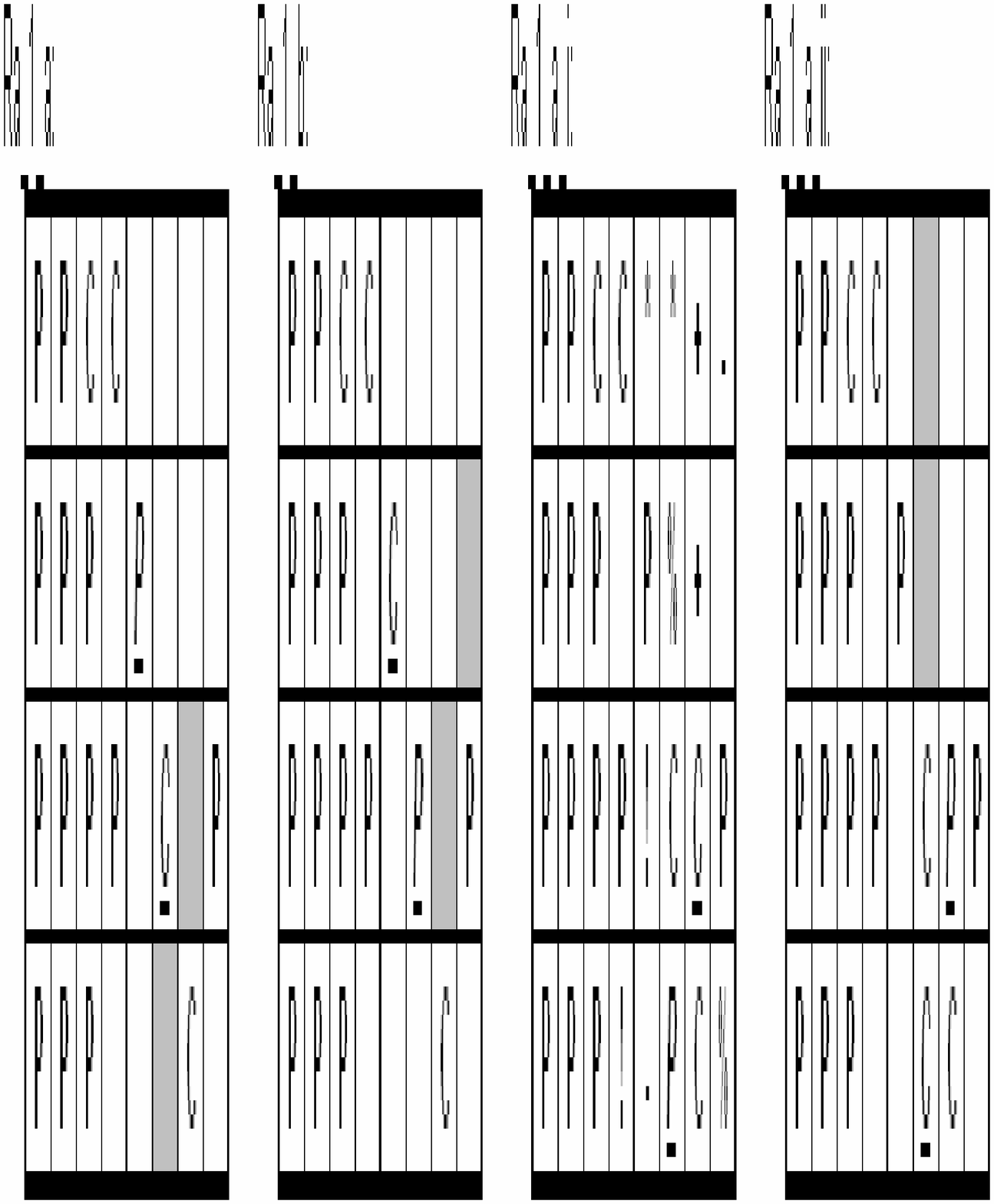}

\includegraphics[bb=0mm 0mm 208mm 296mm, width=120mm, height=16mm, viewport=3mm 4mm 205mm 292mm]{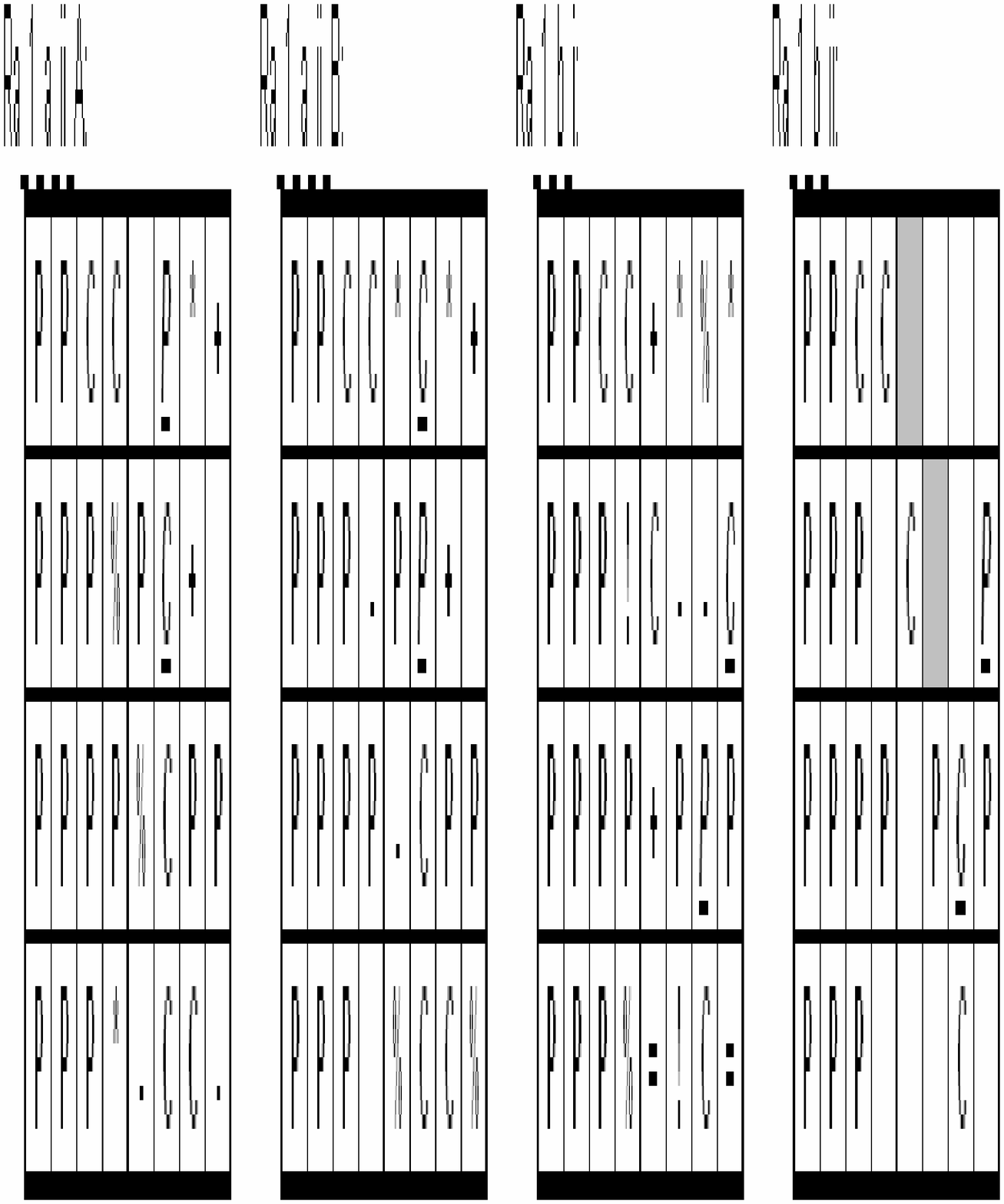}

\includegraphics[bb=0mm 0mm 208mm 296mm, width=120mm, height=16mm, viewport=3mm 4mm 205mm 292mm]{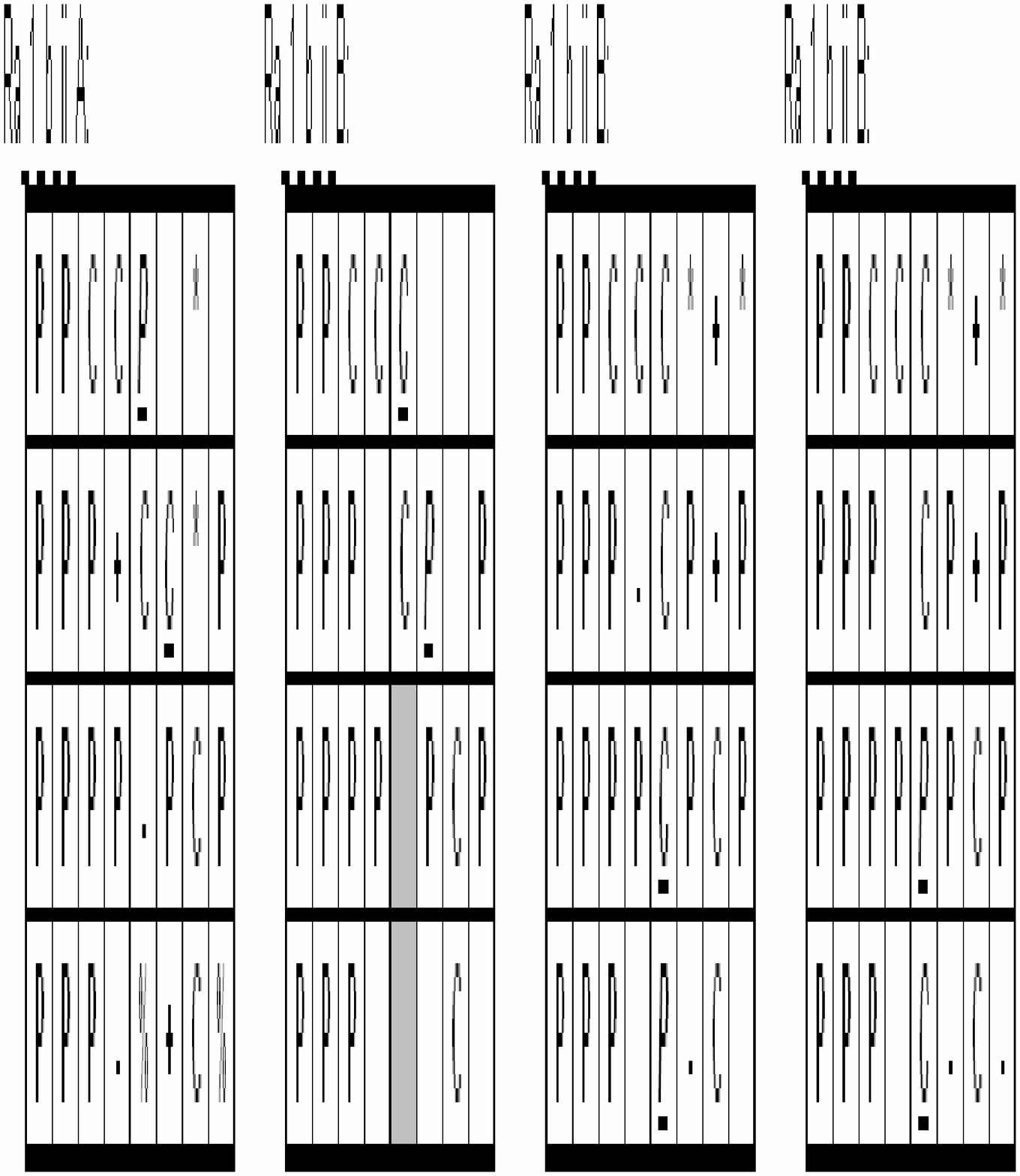}

\includegraphics[bb=0mm 0mm 208mm 296mm, width=120mm, height=16mm, viewport=3mm 4mm 205mm 292mm]{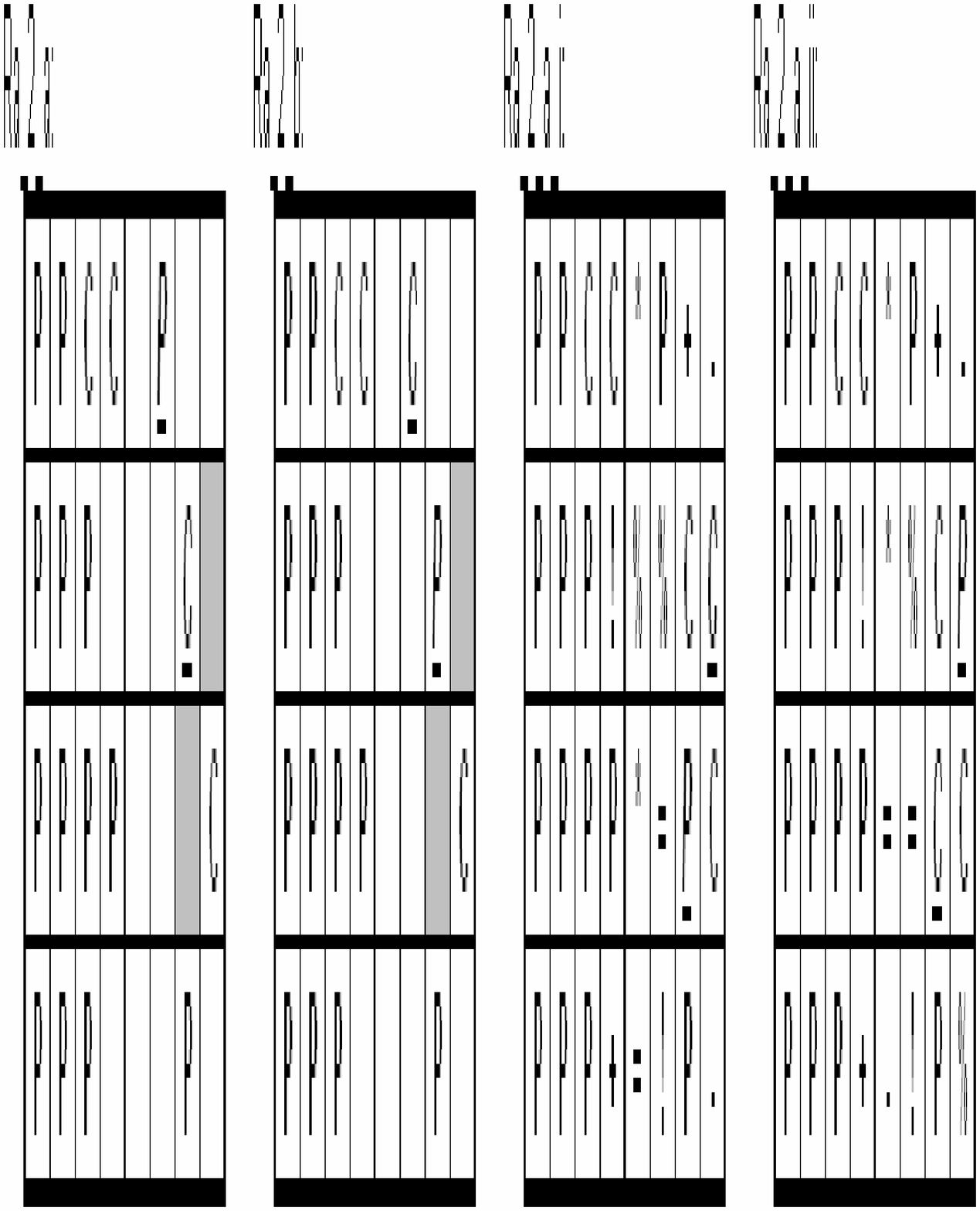}

\includegraphics[bb=0mm 0mm 208mm 296mm, width=120mm, height=16mm, viewport=3mm 4mm 205mm 292mm]{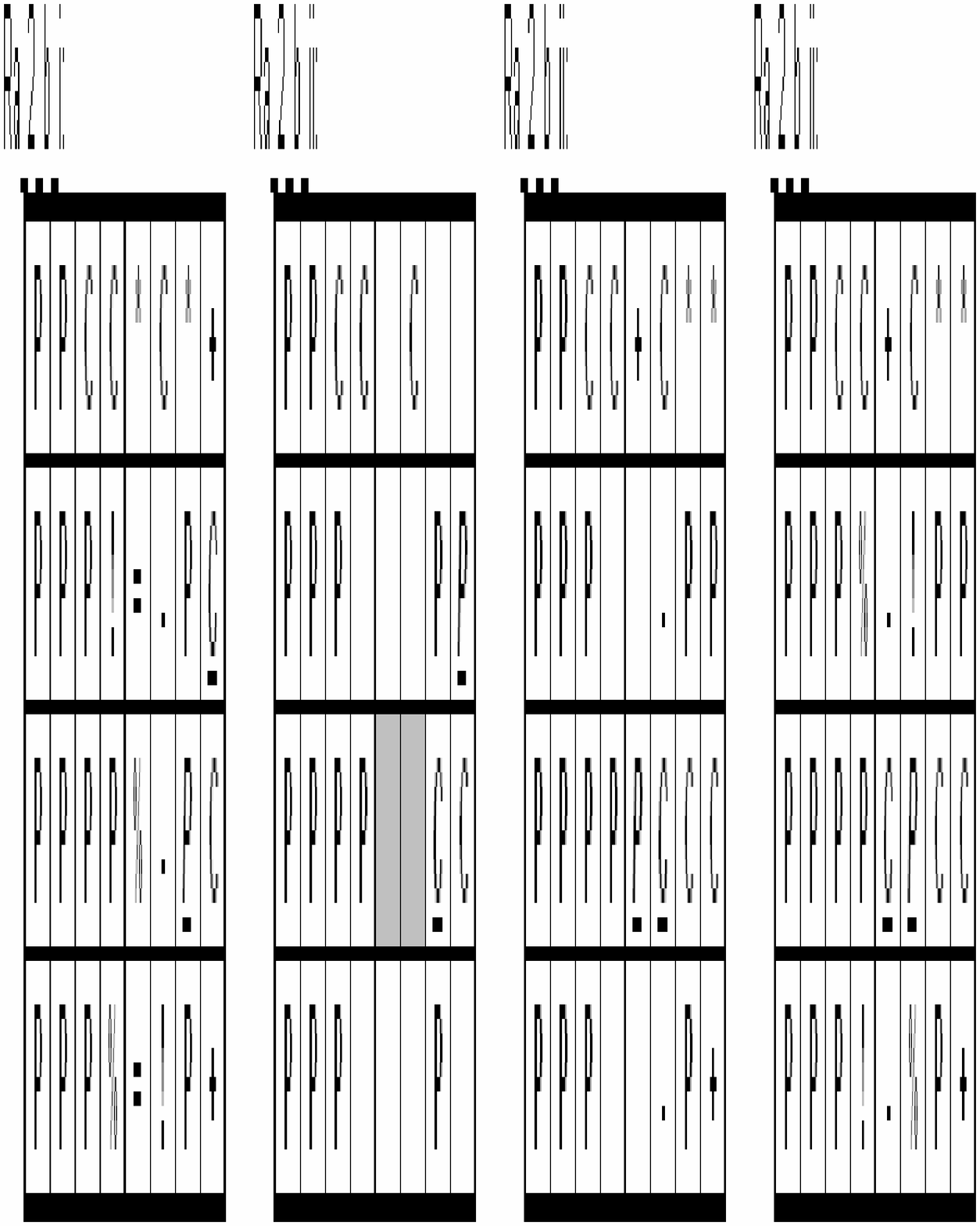}

\includegraphics[bb=0mm 0mm 208mm 296mm, width=120mm, height=20mm, viewport=3mm 4mm 205mm 292mm]{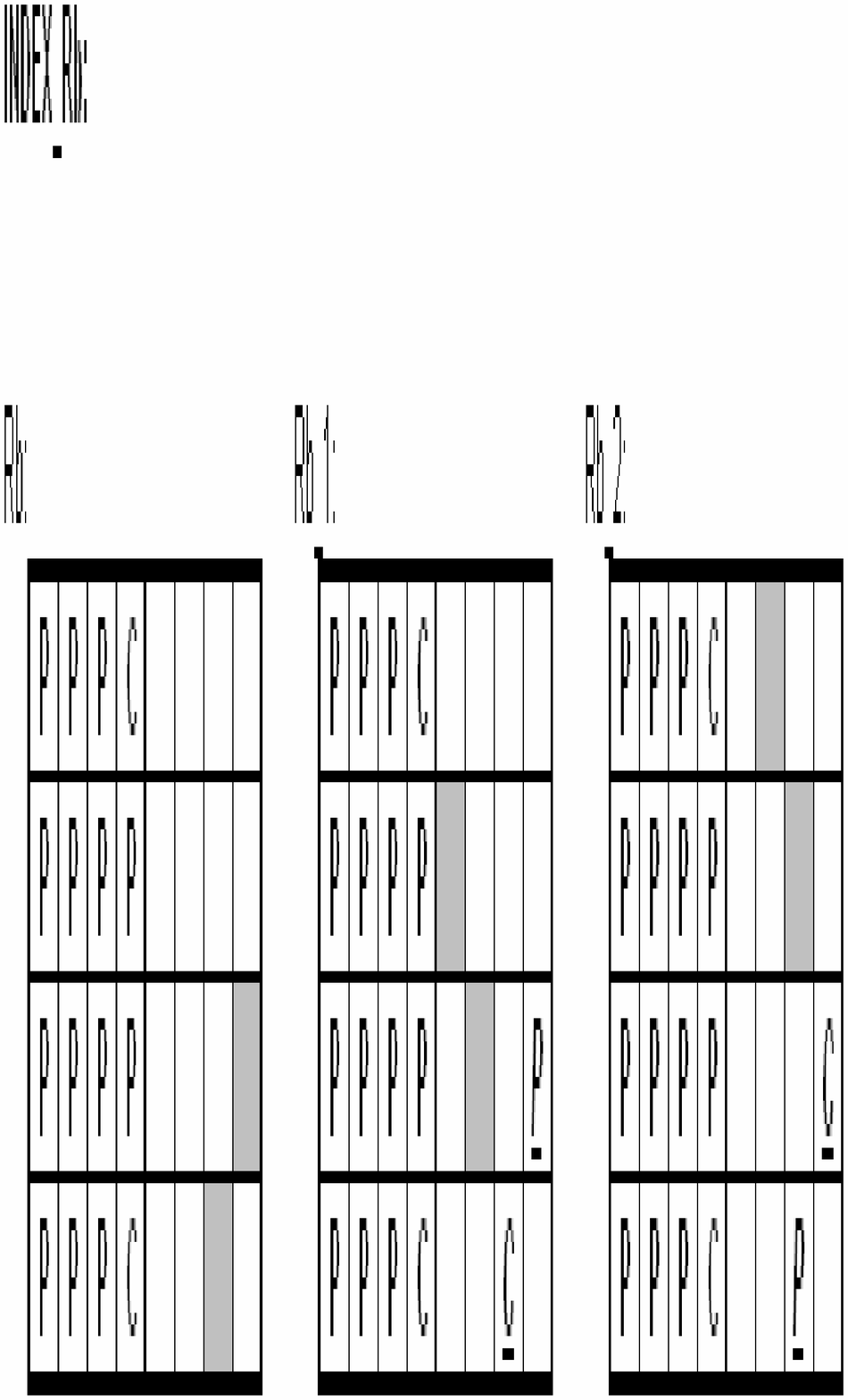}

\includegraphics[bb=0mm 0mm 208mm 296mm, width=120mm, height=15mm, viewport=3mm 4mm 205mm 292mm]{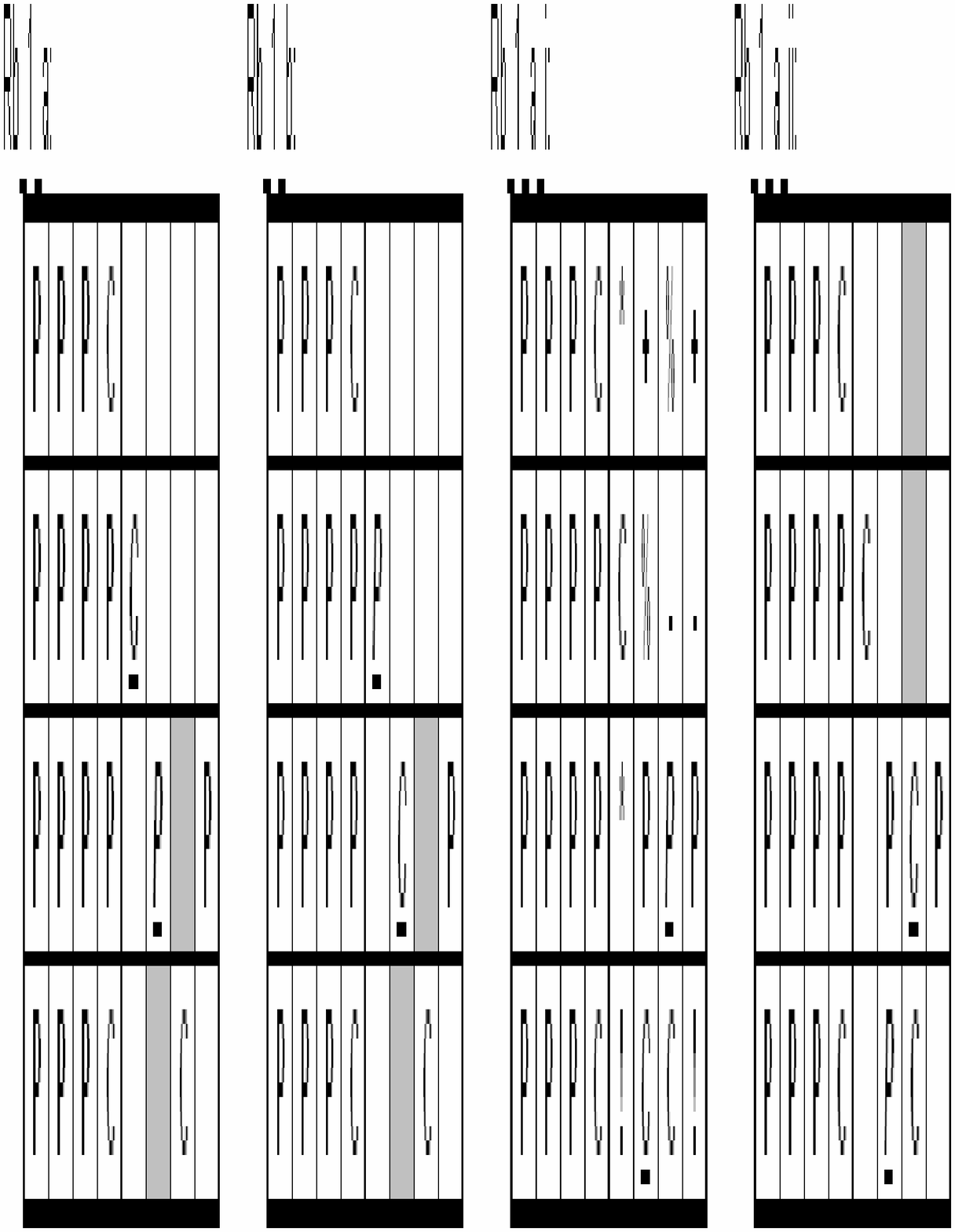}

\includegraphics[bb=0mm 0mm 208mm 296mm, width=120mm, height=16mm, viewport=3mm 4mm 205mm 292mm]{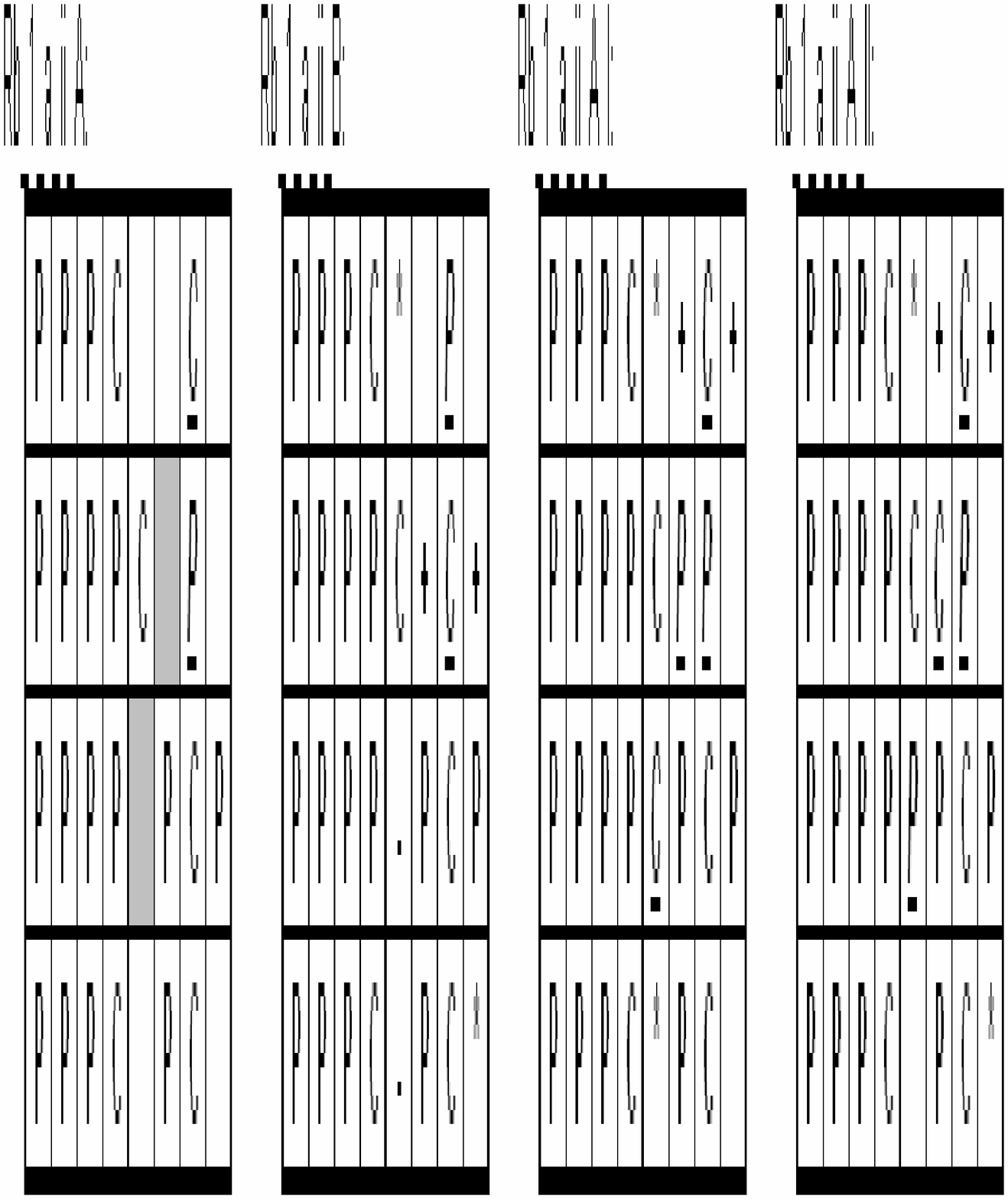}

\includegraphics[bb=0mm 0mm 208mm 296mm, width=120mm, height=16mm, viewport=3mm 4mm 205mm 292mm]{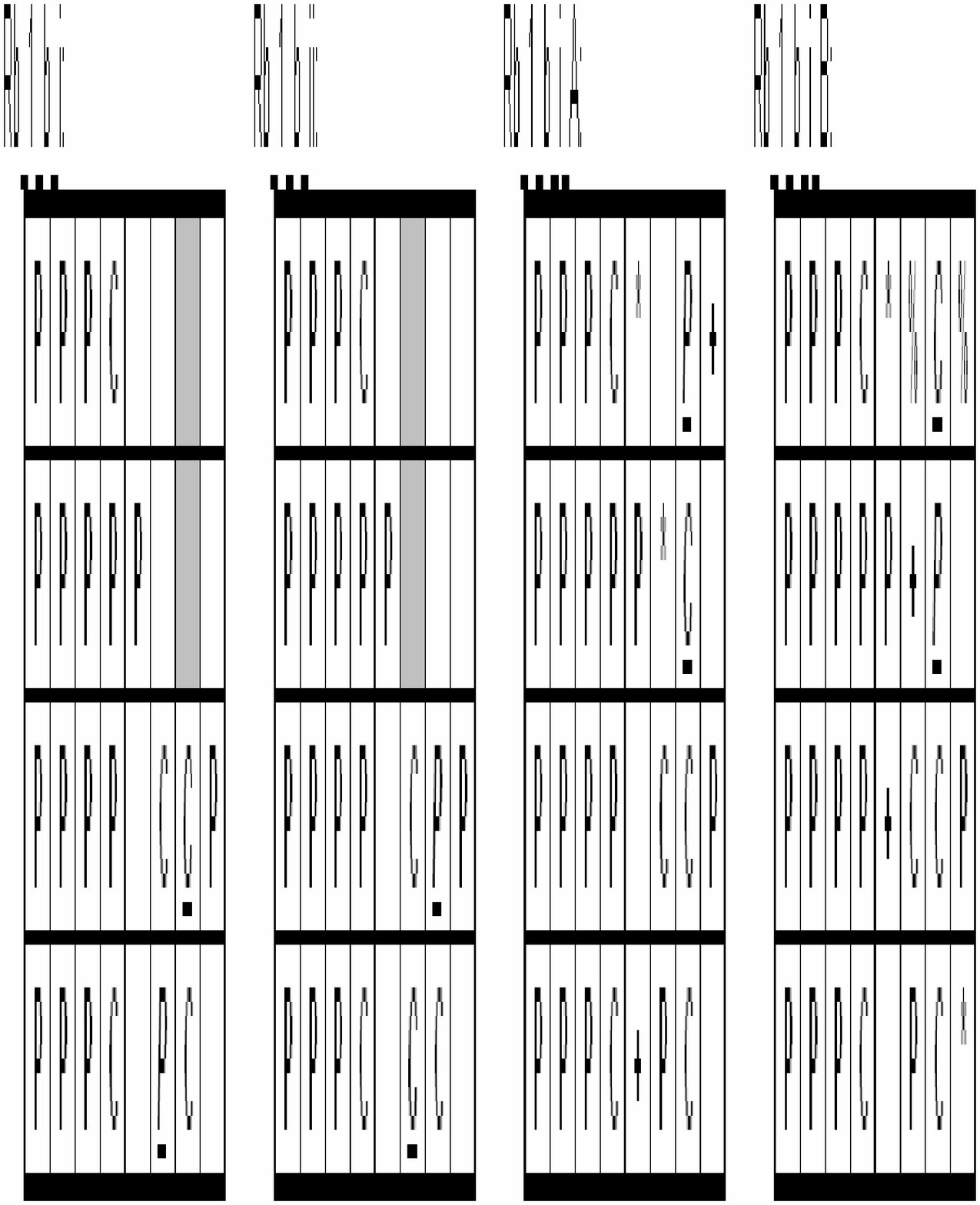}

\includegraphics[bb=0mm 0mm 208mm 296mm, width=120mm, height=16mm, viewport=3mm 4mm 205mm 292mm]{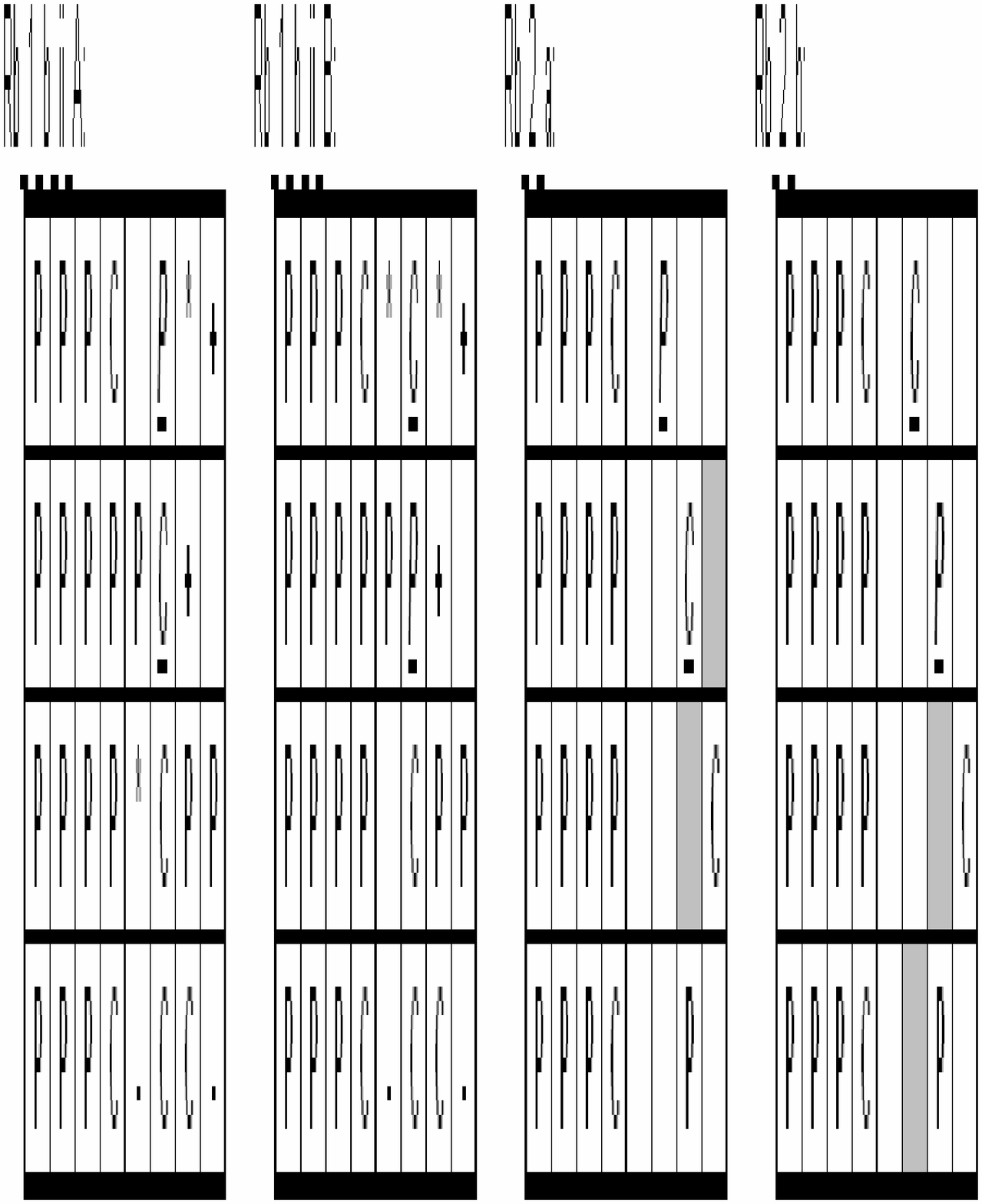}

\includegraphics[bb=0mm 0mm 208mm 296mm, width=120mm, height=16mm, viewport=3mm 4mm 205mm 292mm]{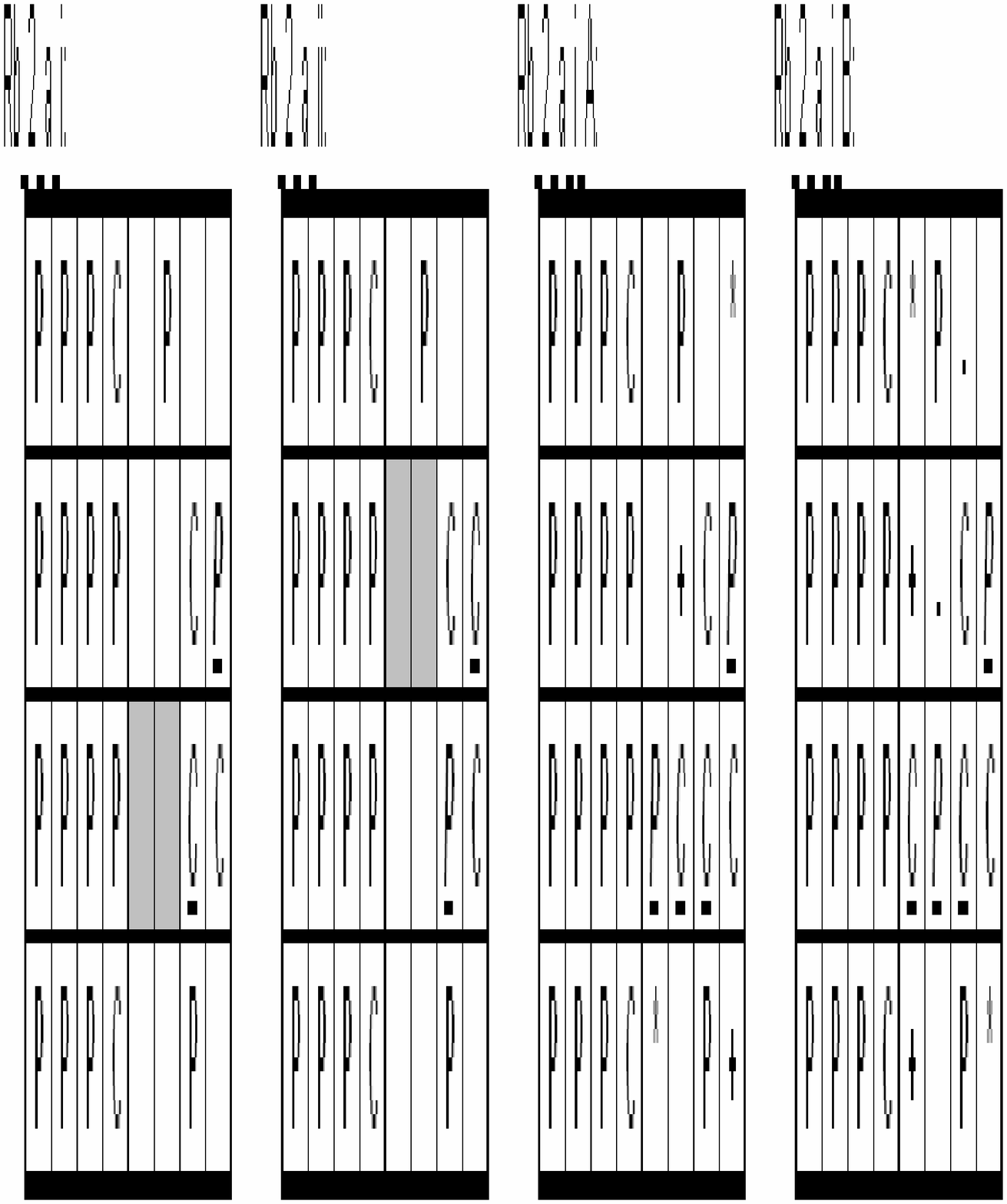}

\includegraphics[bb=0mm 0mm 208mm 296mm, width=120mm, height=15.5mm, viewport=3mm 4mm 205mm 292mm]{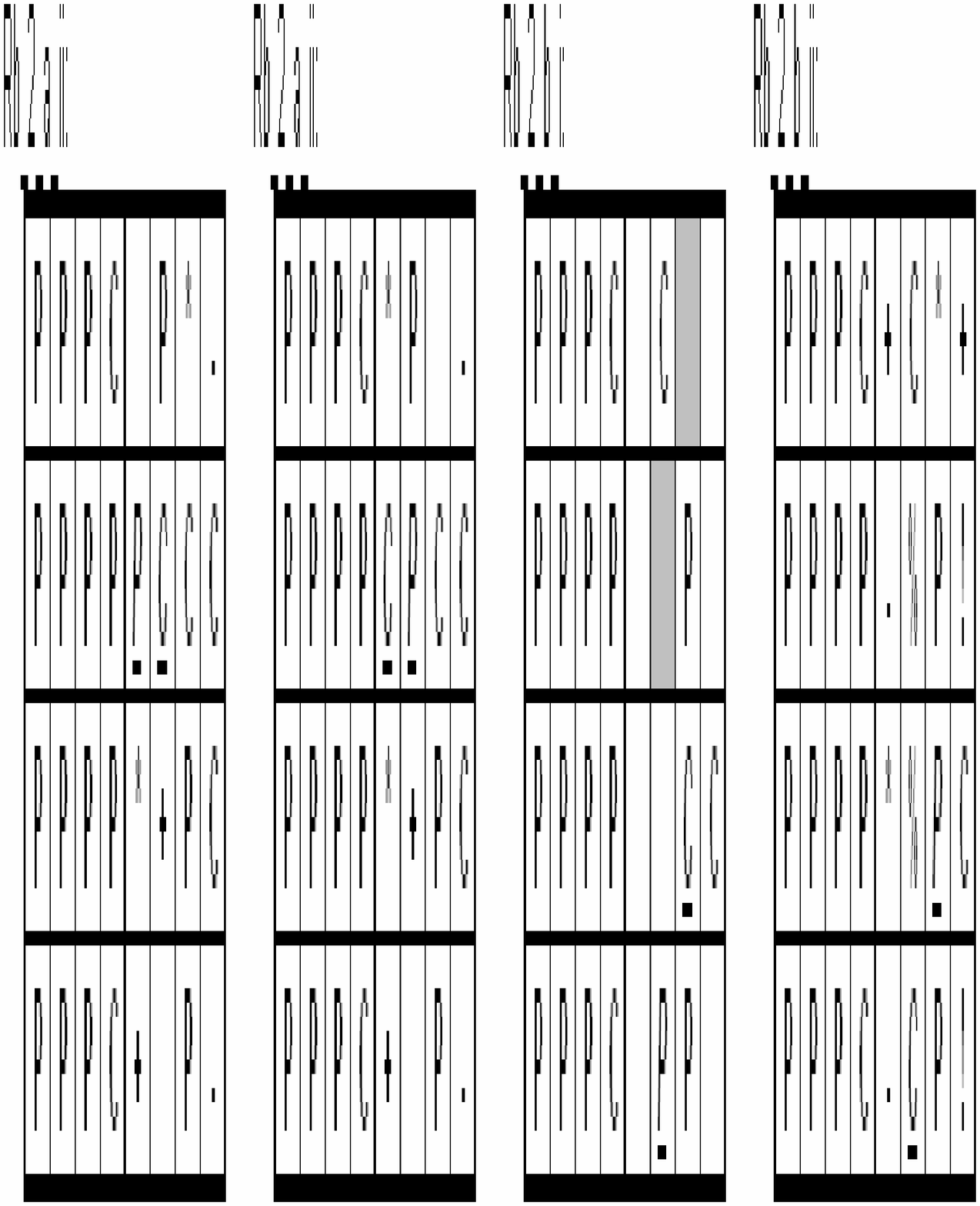}

\includegraphics[bb=0mm 0mm 208mm 296mm, width=120mm, height=16mm, viewport=3mm 4mm 205mm 292mm]{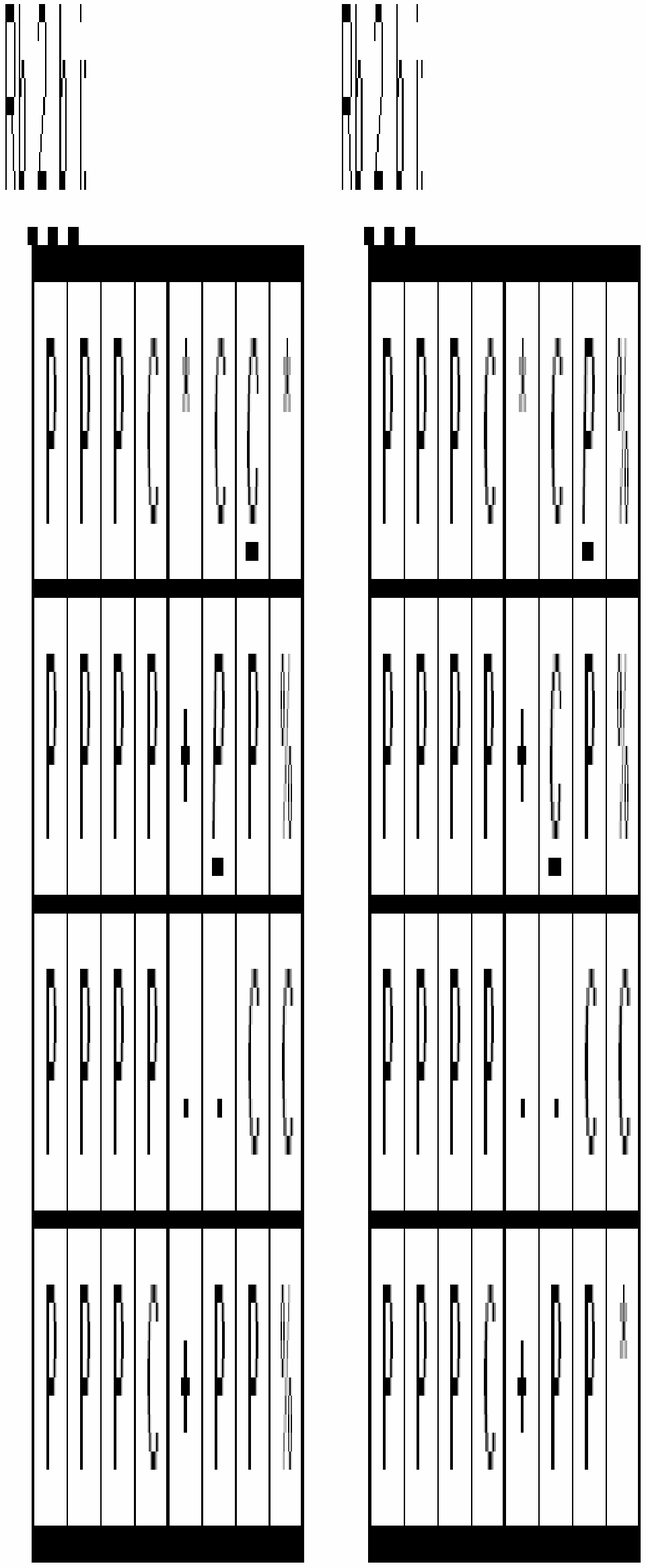}

\includegraphics[bb=0mm 0mm 208mm 296mm, width=120mm, height=19mm, viewport=3mm 4mm 205mm 292mm]{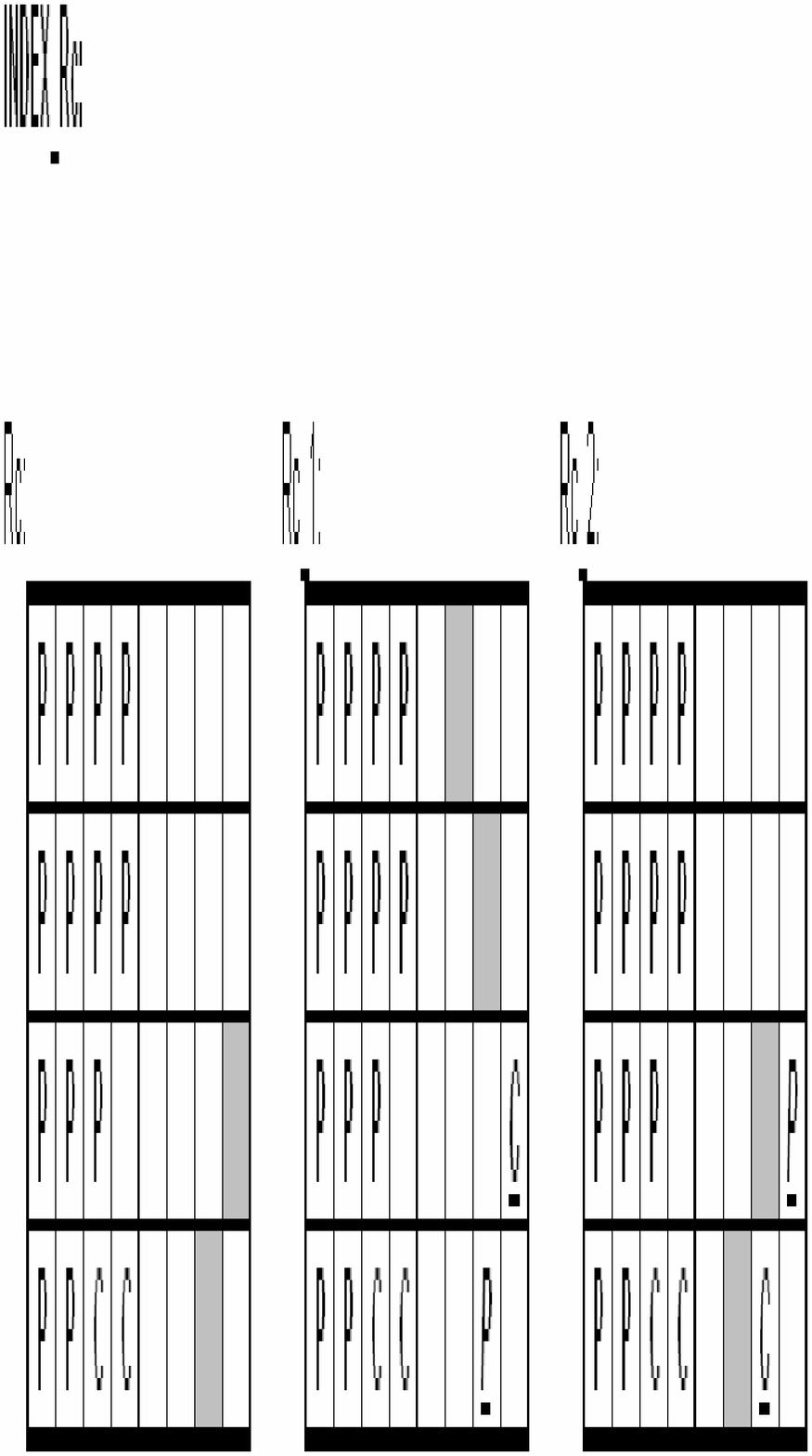}

\includegraphics[bb=0mm 0mm 208mm 296mm, width=120mm, height=16mm, viewport=3mm 4mm 205mm 292mm]{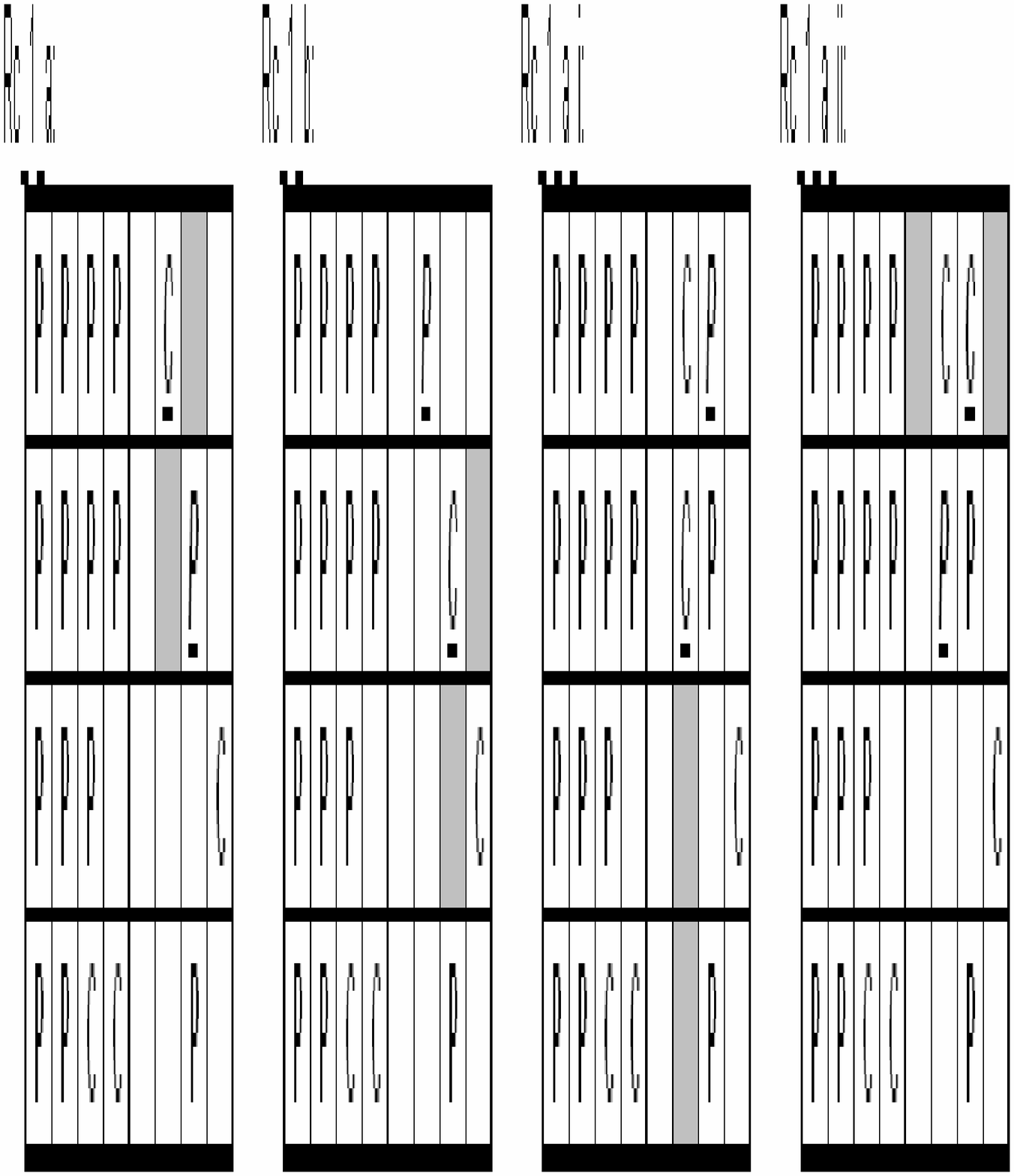}

\includegraphics[bb=0mm 0mm 208mm 296mm, width=120mm, height=15.5mm, viewport=3mm 4mm 205mm 292mm]{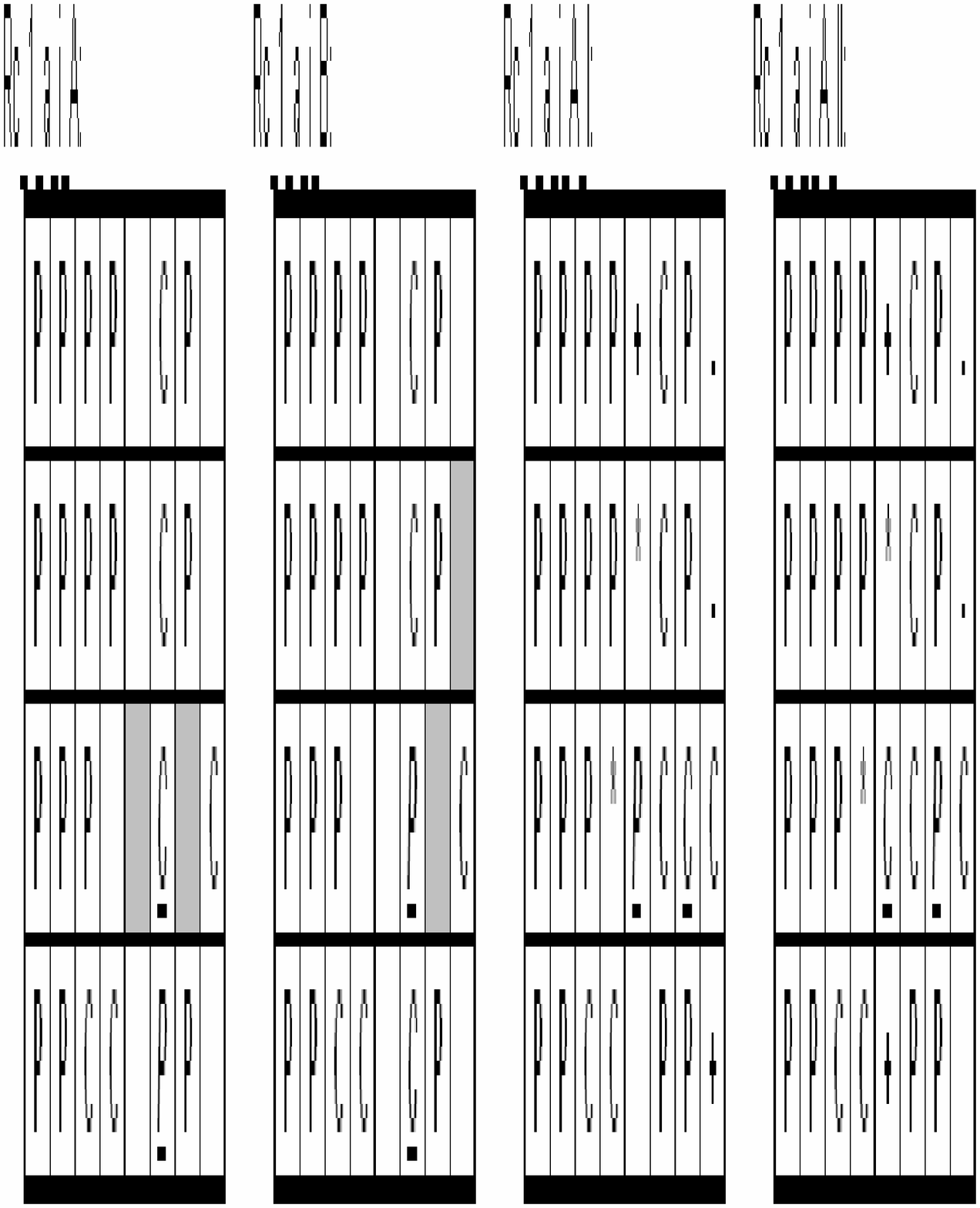}

\includegraphics[bb=0mm 0mm 208mm 296mm, width=120mm, height=14.8mm, viewport=3mm 4mm 205mm 292mm]{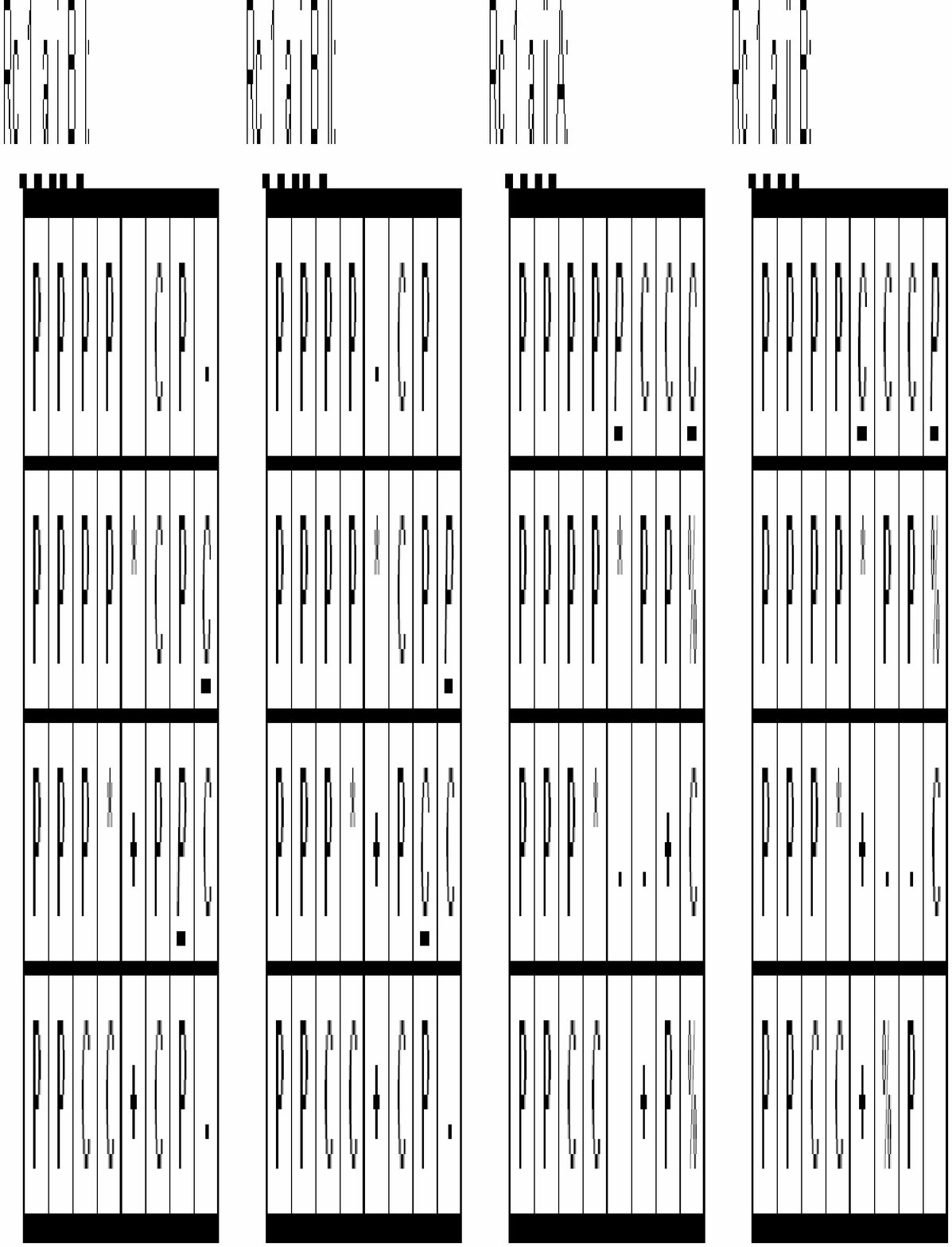}

\includegraphics[bb=0mm 0mm 208mm 296mm, width=120mm, height=16mm, viewport=3mm 4mm 205mm 292mm]{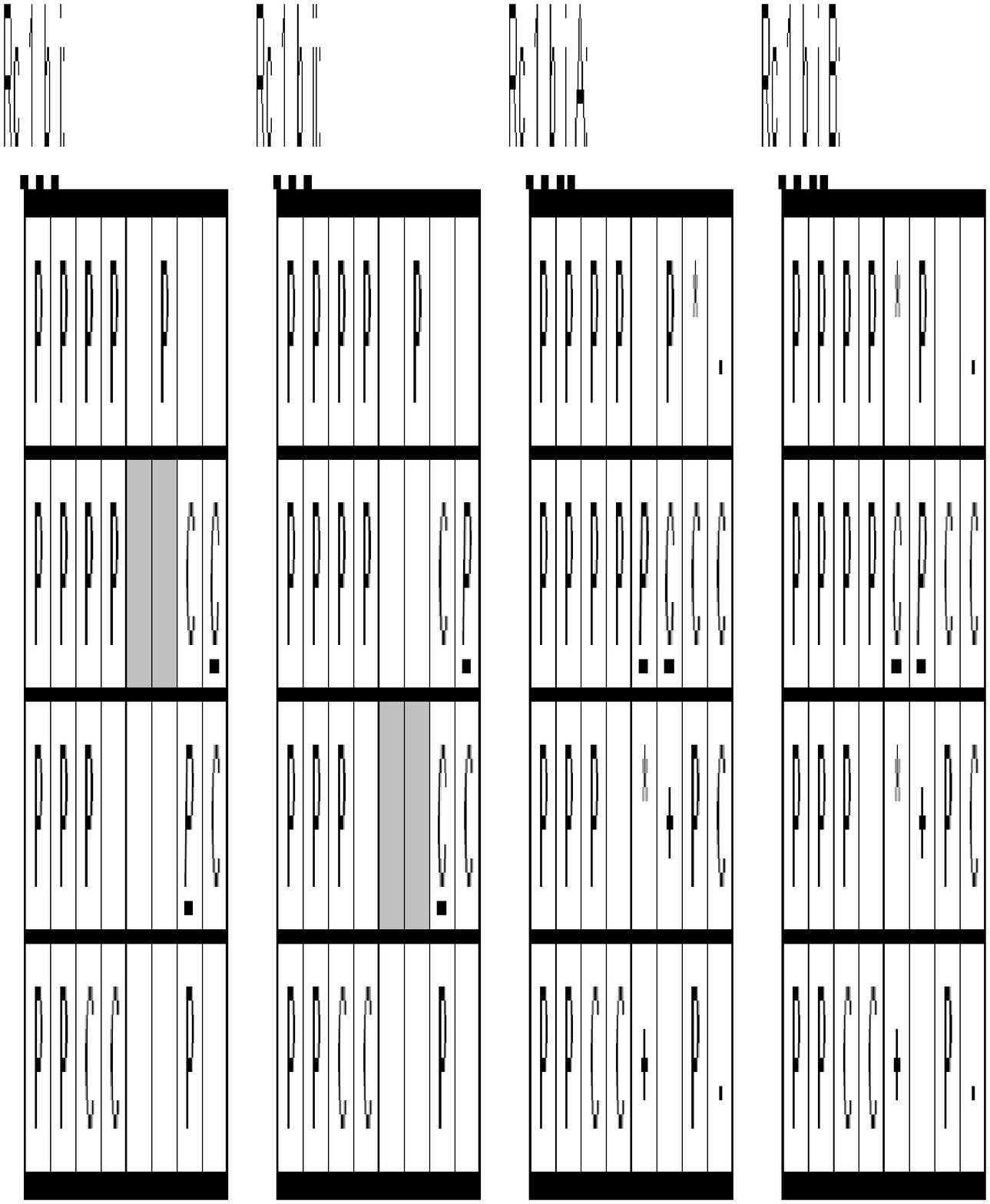}

\includegraphics[bb=0mm 0mm 208mm 296mm, width=120mm, height=15.5mm, viewport=3mm 4mm 205mm 292mm]{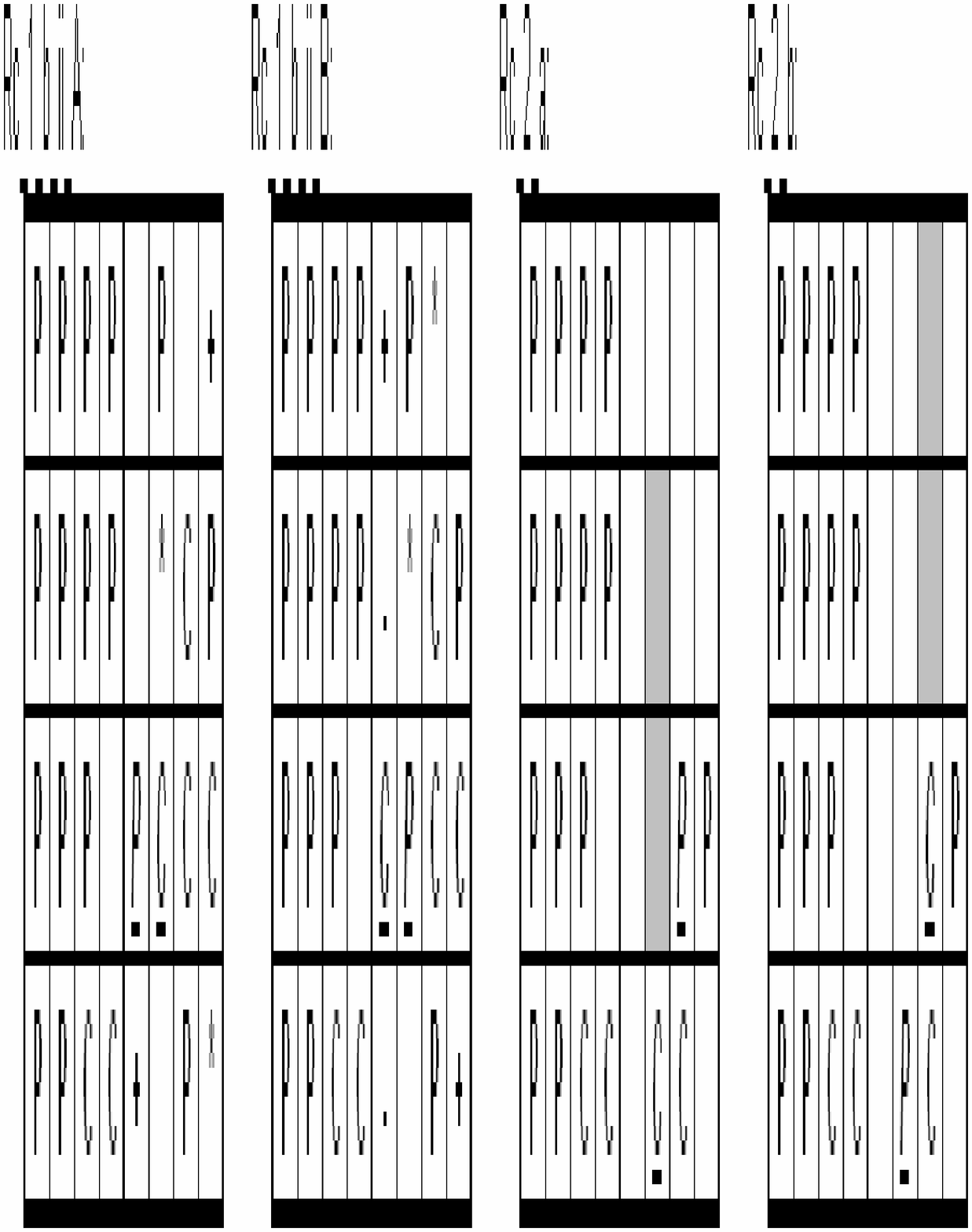}

\includegraphics[bb=0mm 0mm 208mm 296mm, width=120mm, height=16mm, viewport=3mm 4mm 205mm 292mm]{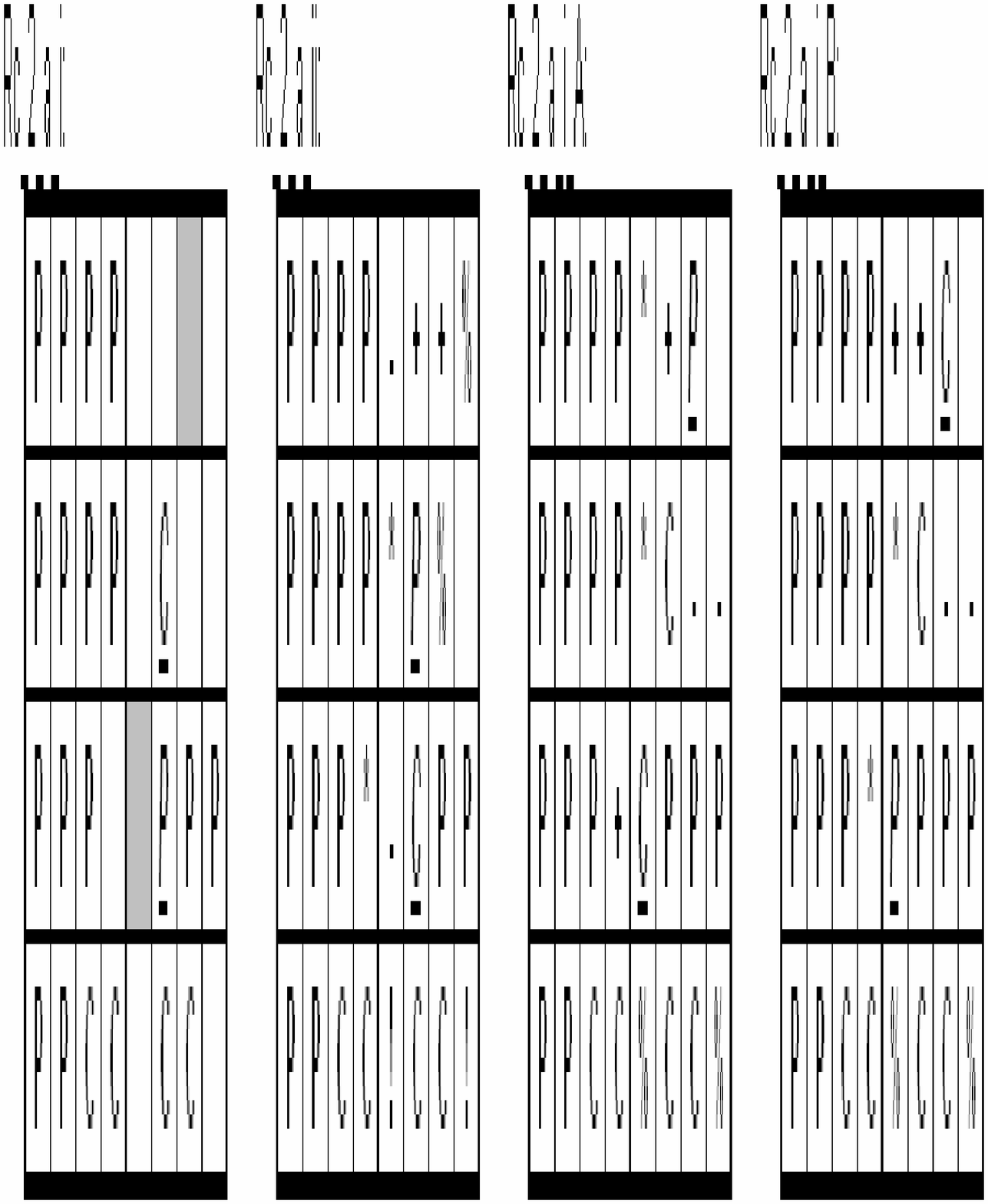}

\includegraphics[bb=0mm 0mm 208mm 296mm, width=120mm, height=15.5mm, viewport=3mm 4mm 205mm 292mm]{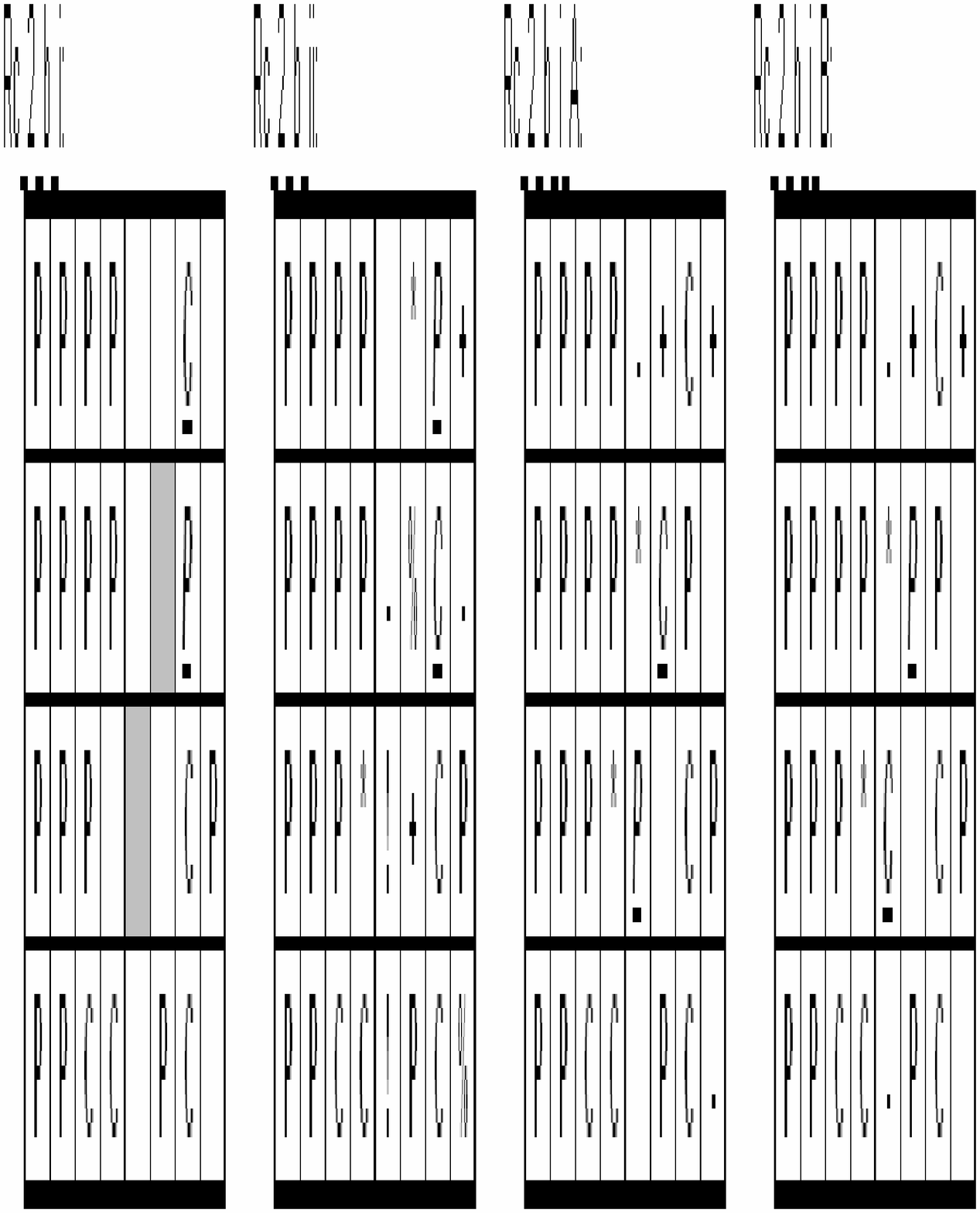}

\includegraphics[bb=0mm 0mm 208mm 296mm, width=120mm, height=19mm, viewport=3mm 4mm 205mm 292mm]{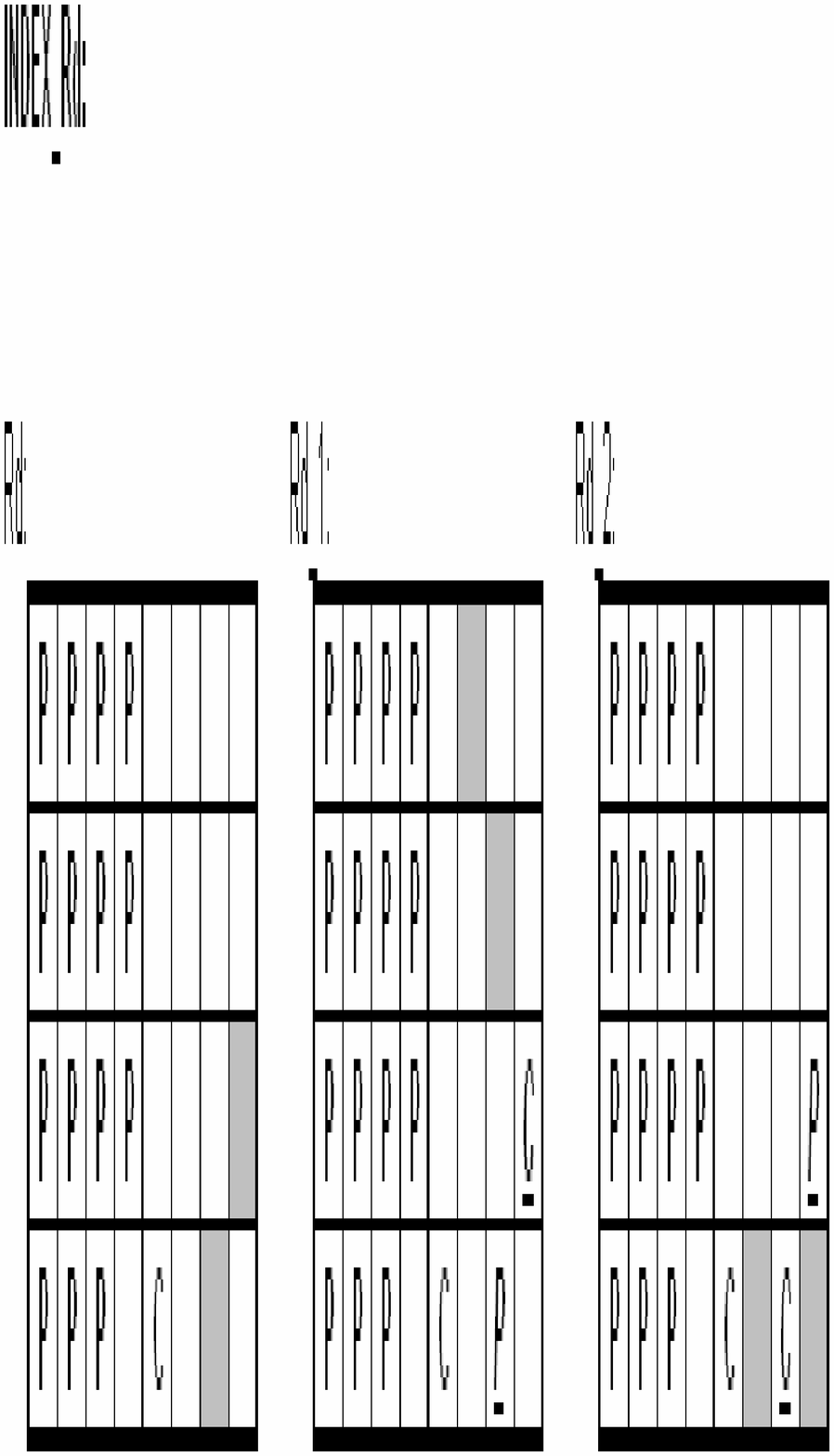}

\includegraphics[bb=0mm 0mm 208mm 296mm, width=120mm, height=16mm, viewport=3mm 4mm 205mm 292mm]{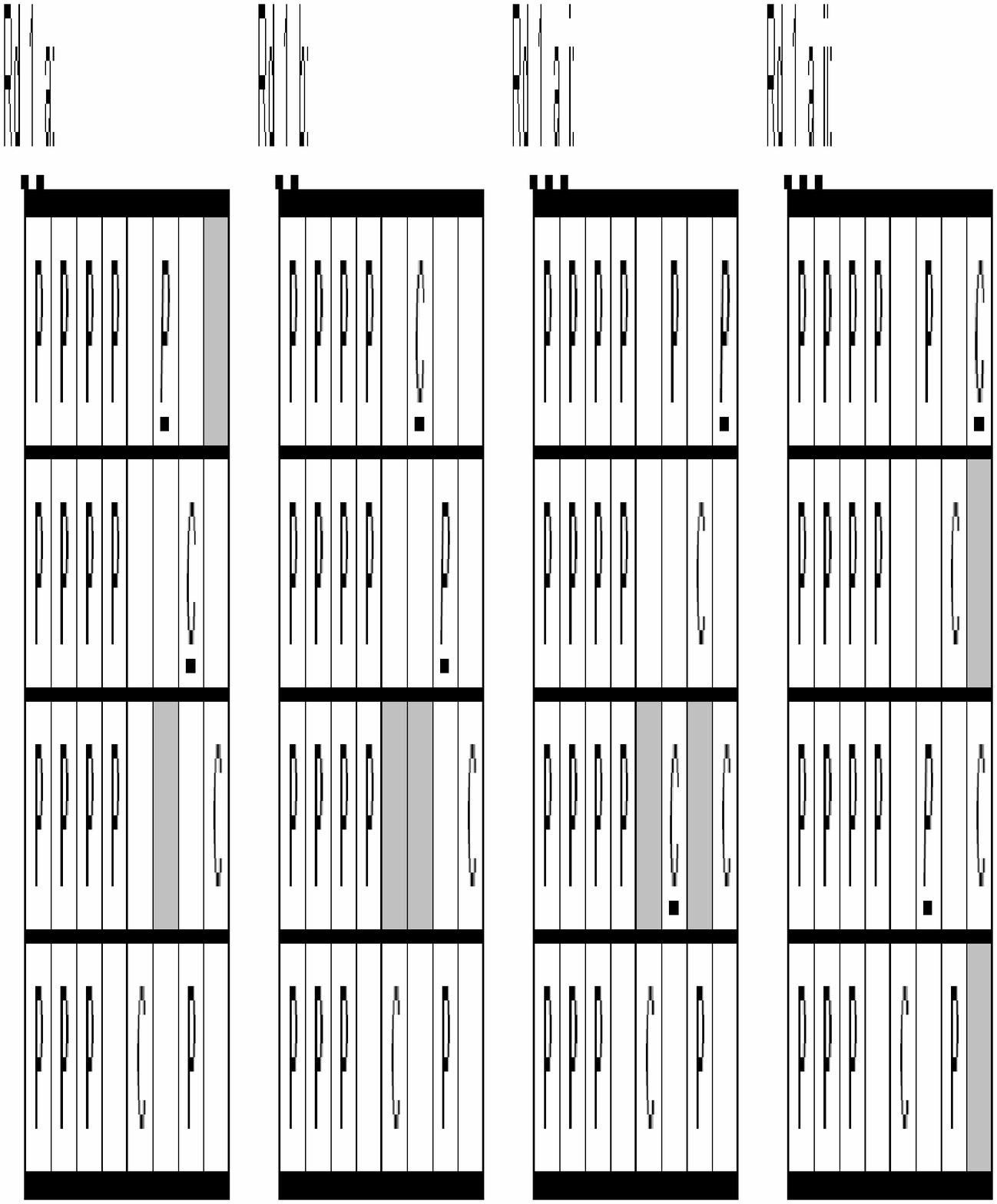}

\includegraphics[bb=0mm 0mm 208mm 296mm, width=120mm, height=15.5mm, viewport=3mm 4mm 205mm 292mm]{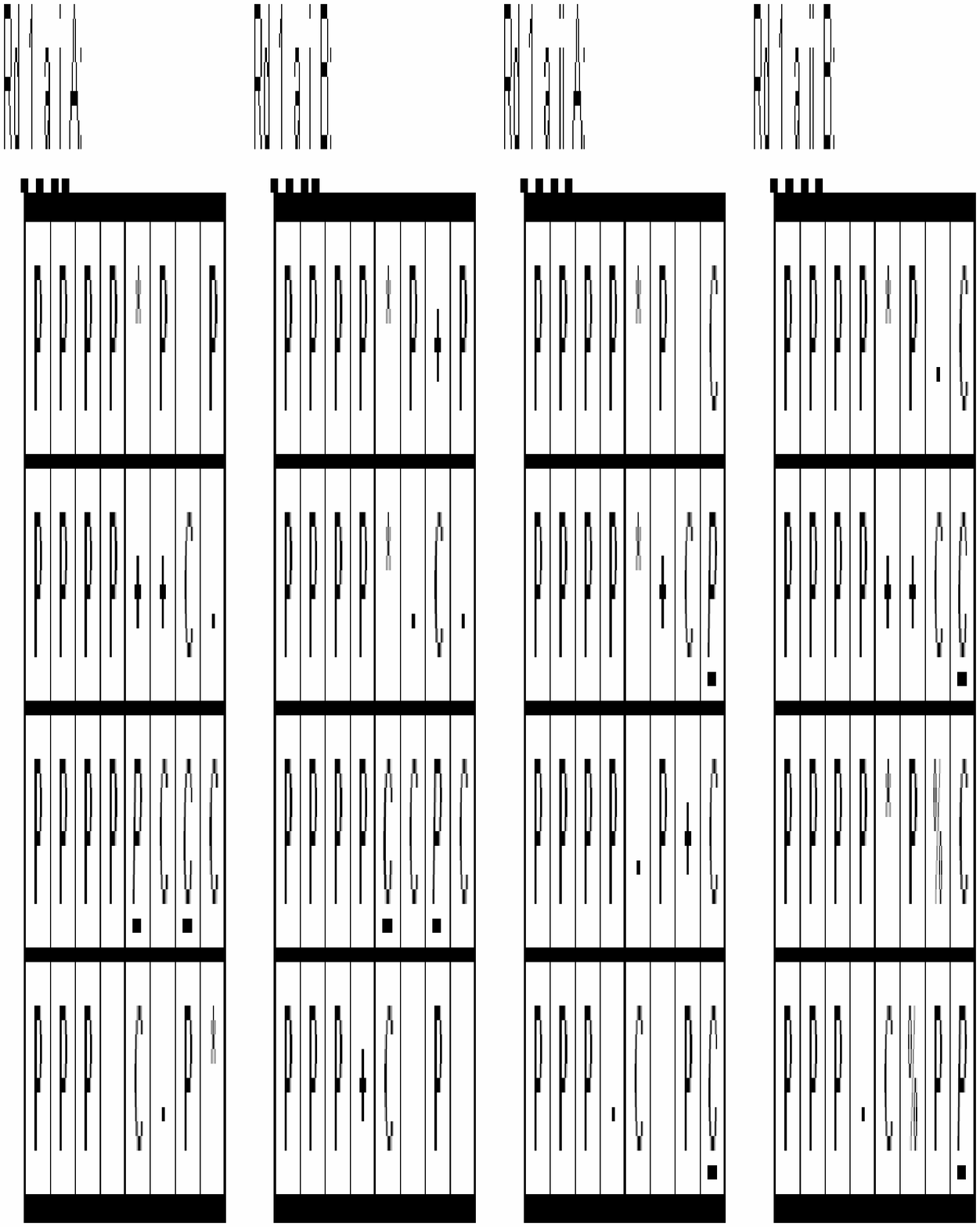}

\includegraphics[bb=0mm 0mm 208mm 296mm, width=120mm, height=15mm, viewport=3mm 4mm 205mm 292mm]{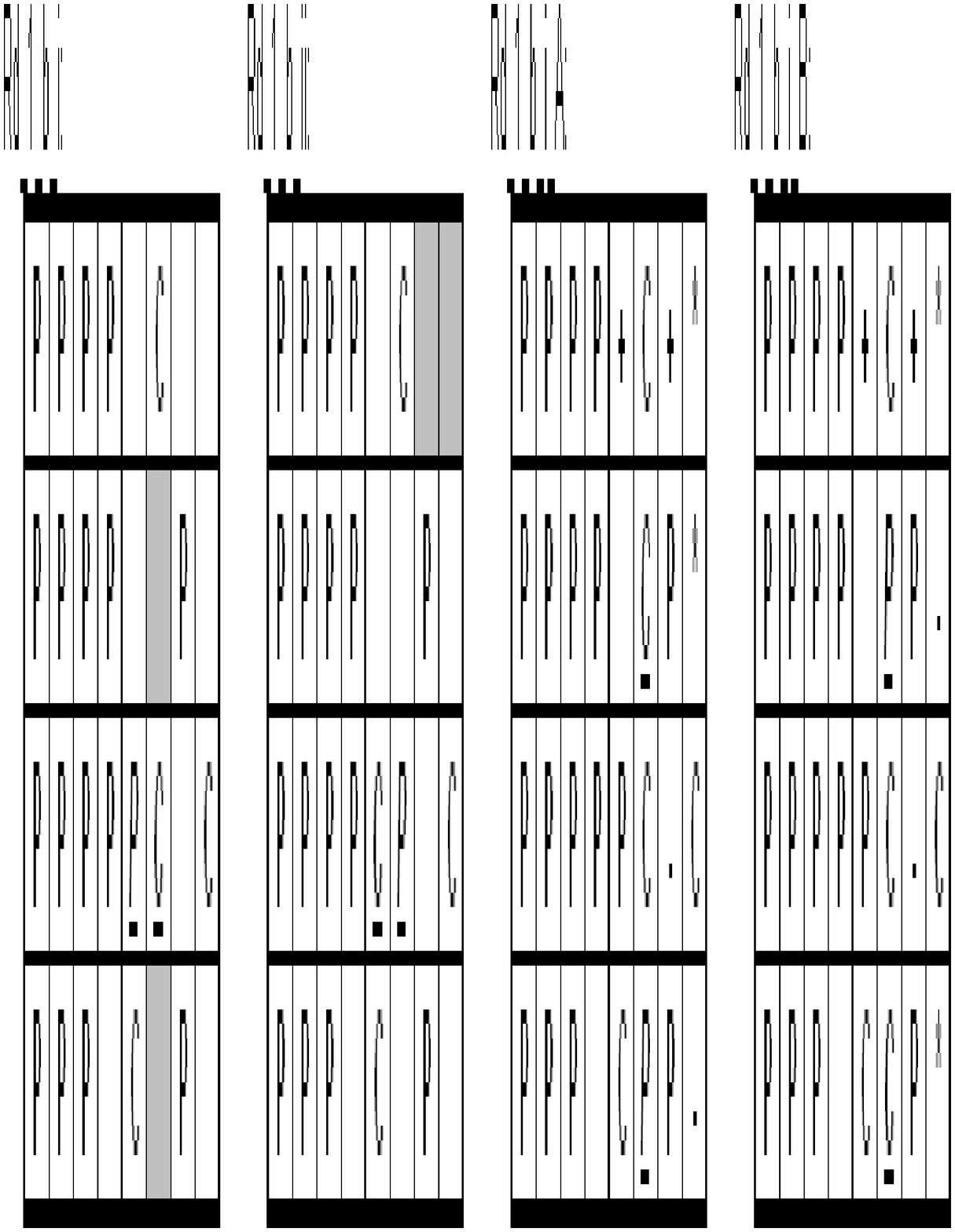}

\includegraphics[bb=0mm 0mm 208mm 296mm, width=120mm, height=16mm, viewport=3mm 4mm 205mm 292mm]{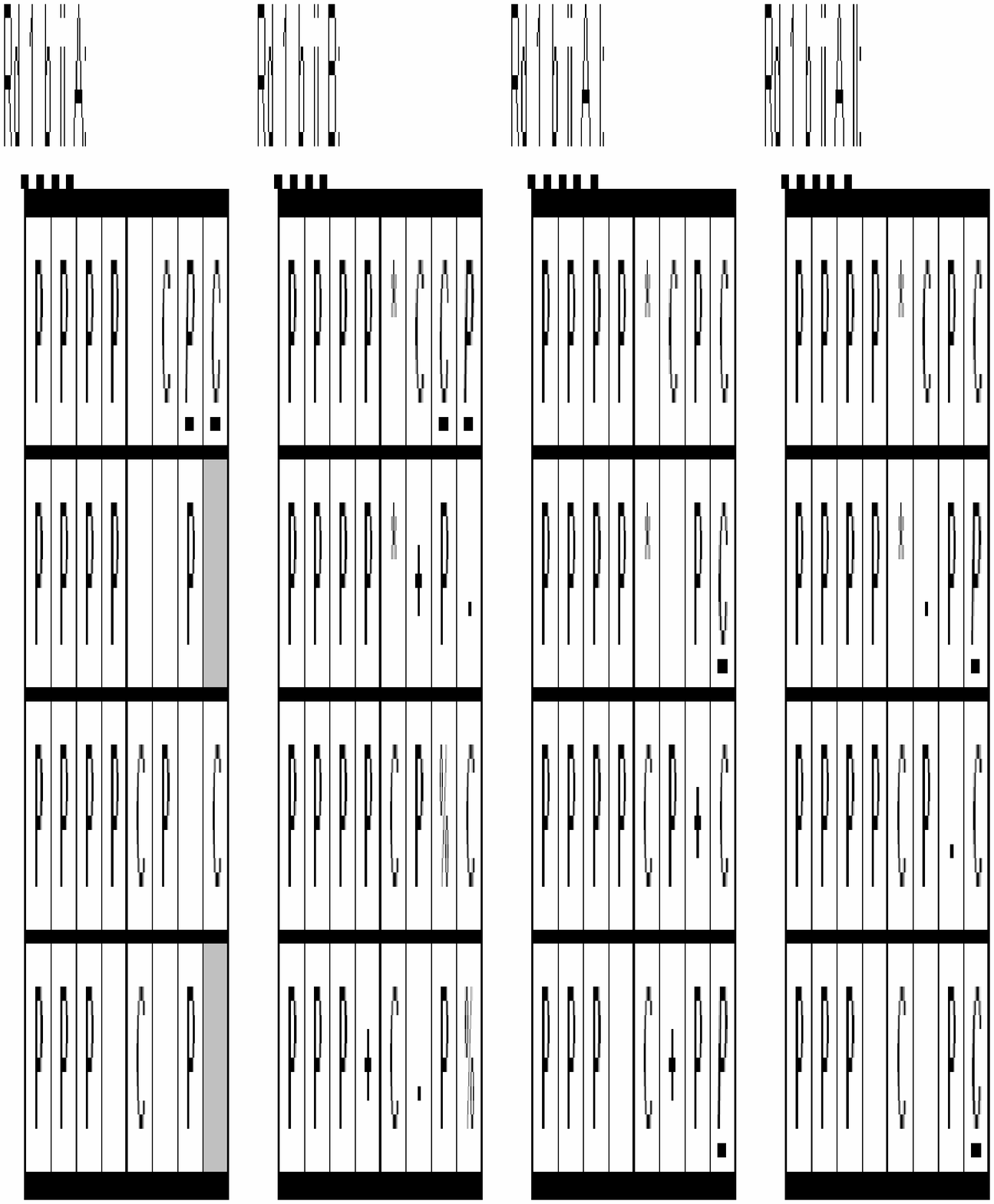}

\includegraphics[bb=0mm 0mm 208mm 296mm, width=120mm, height=15.5mm, viewport=3mm 4mm 205mm 292mm]{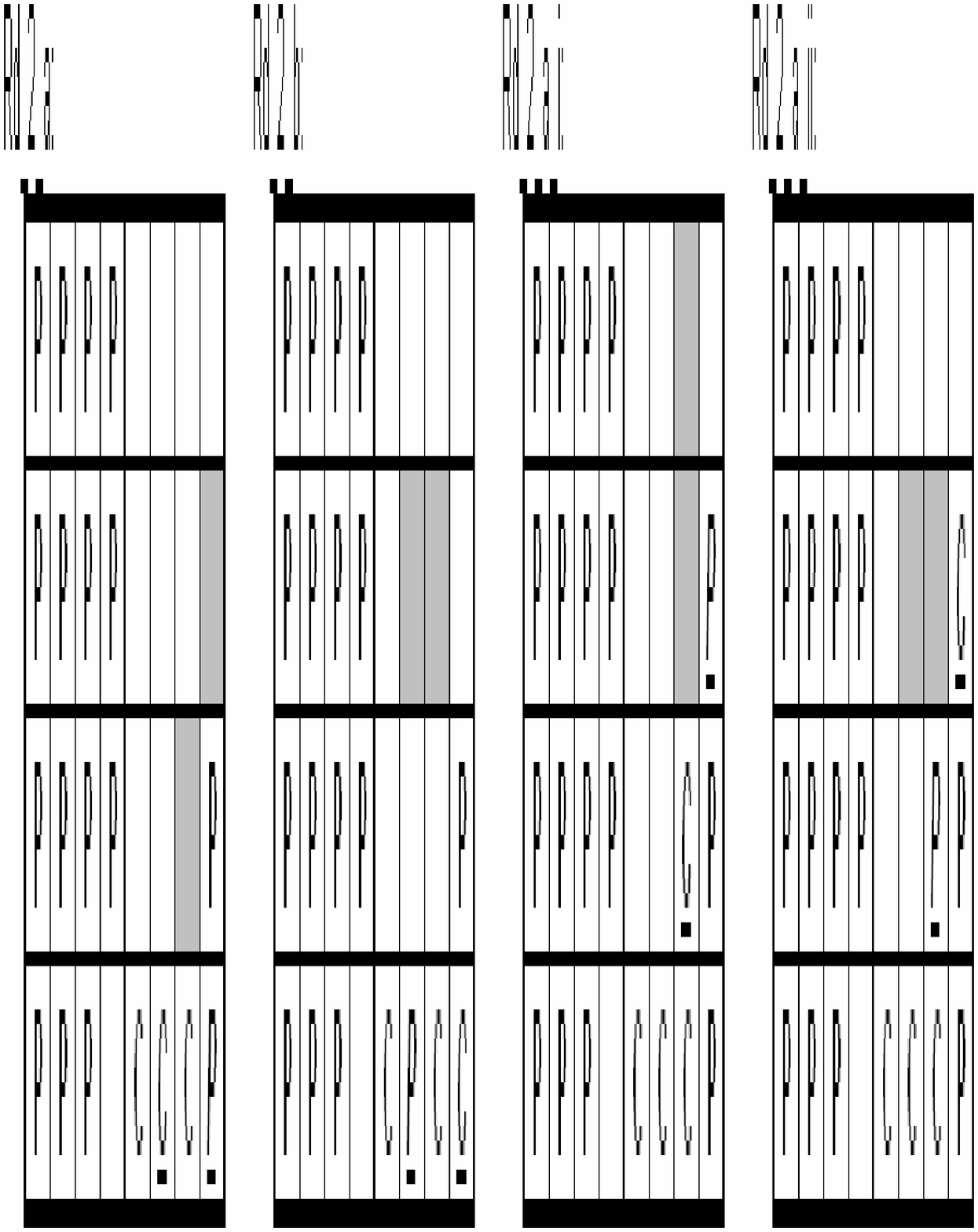}

\includegraphics[bb=0mm 0mm 208mm 296mm, width=120mm, height=15.5mm, viewport=3mm 4mm 205mm 292mm]{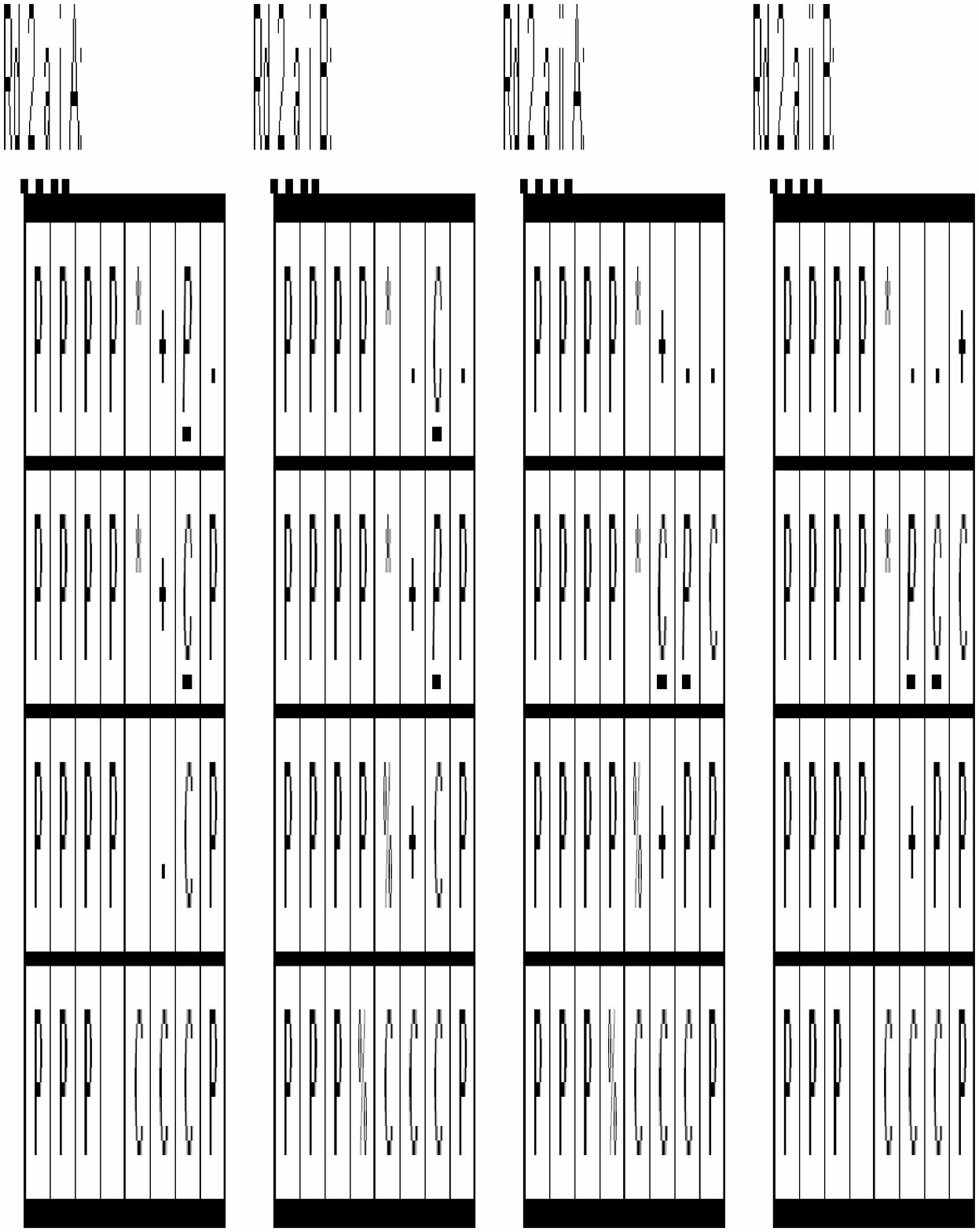}

\includegraphics[bb=0mm 0mm 208mm 296mm, width=120mm, height=16.4mm, viewport=3mm 4mm 205mm 292mm]{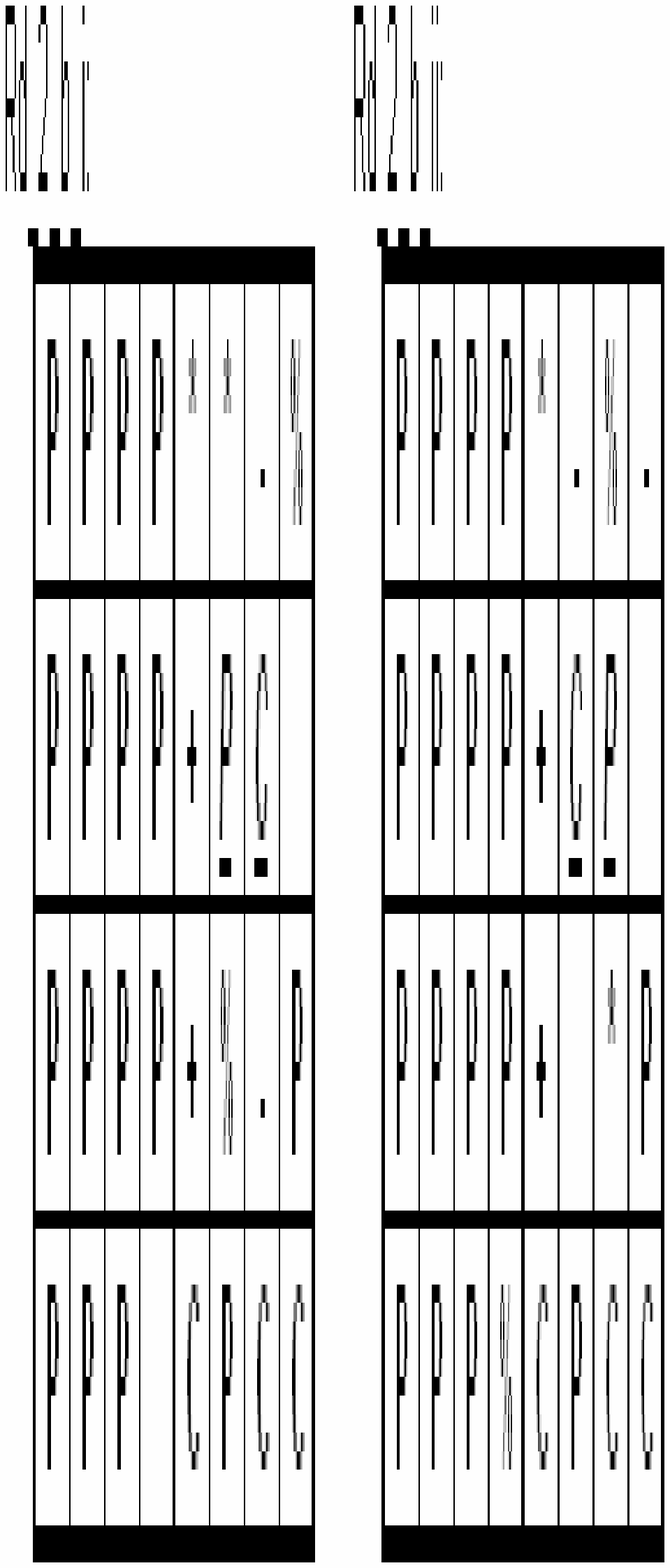}

\includegraphics[bb=0mm 0mm 208mm 296mm, width=120mm, height=18.5mm, viewport=3mm 4mm 205mm 292mm]{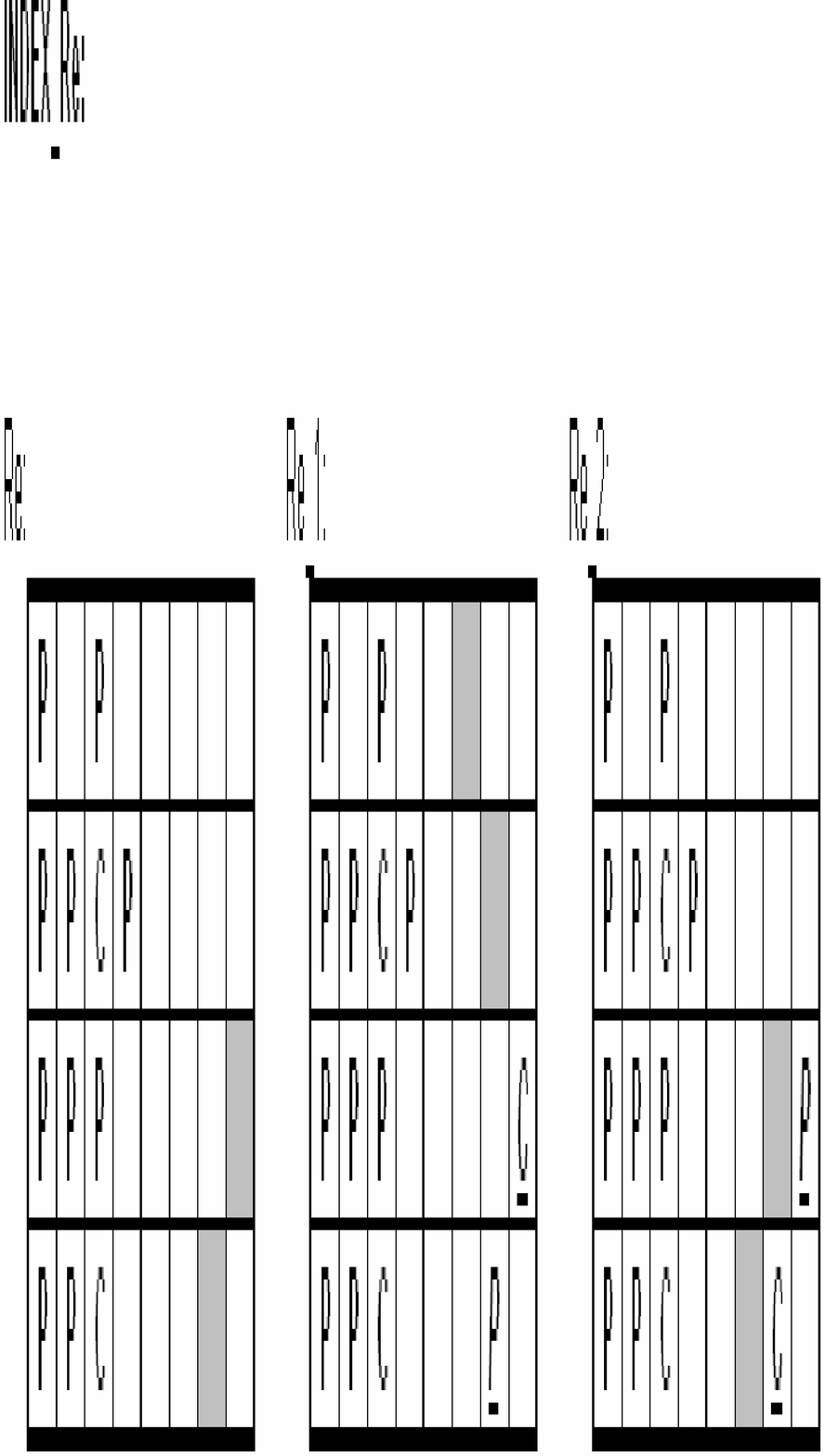}

\includegraphics[bb=0mm 0mm 208mm 296mm, width=120mm, height=16mm, viewport=3mm 4mm 205mm 292mm]{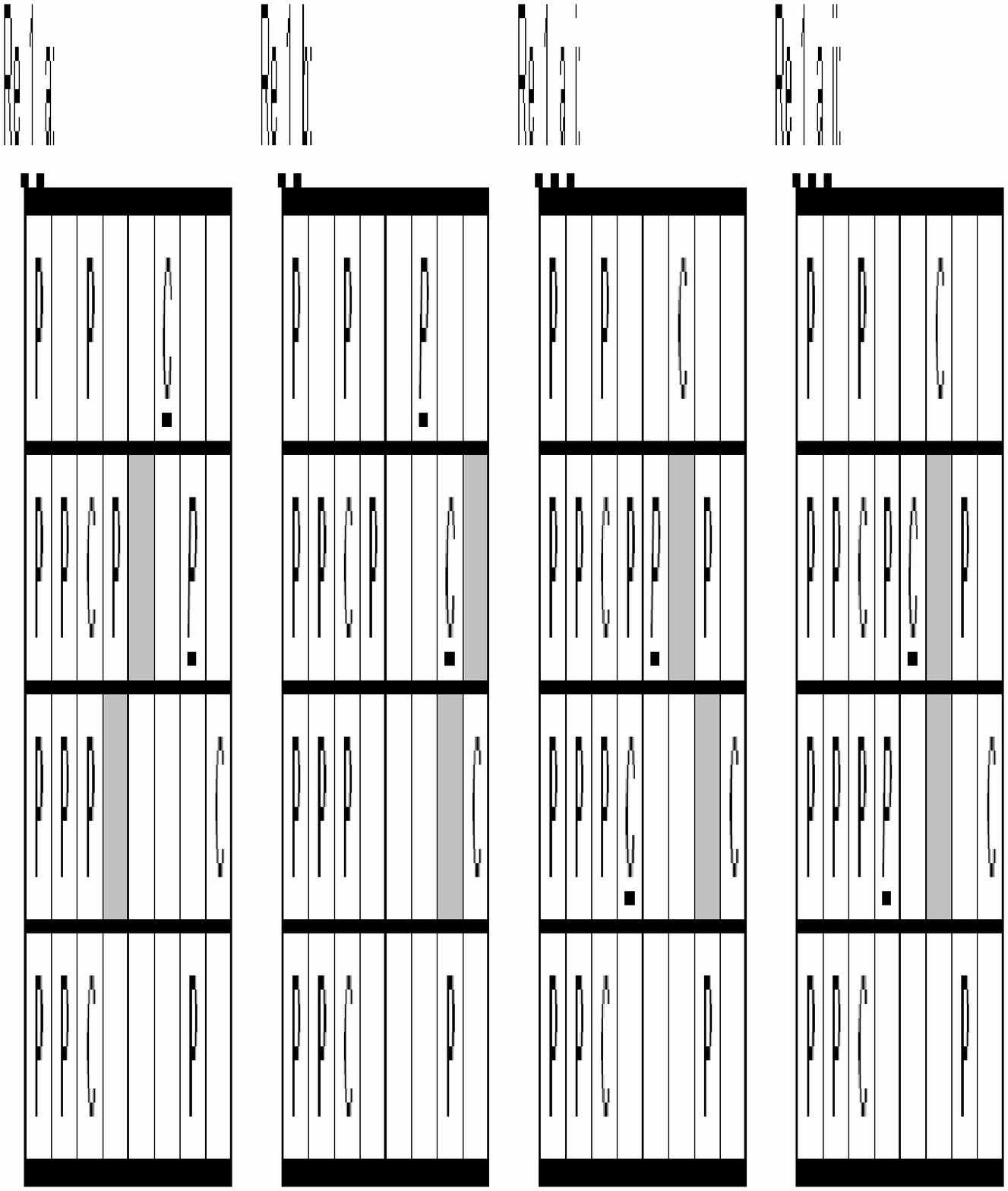}

\includegraphics[bb=0mm 0mm 208mm 296mm, width=120mm, height=15.5mm, viewport=3mm 4mm 205mm 292mm]{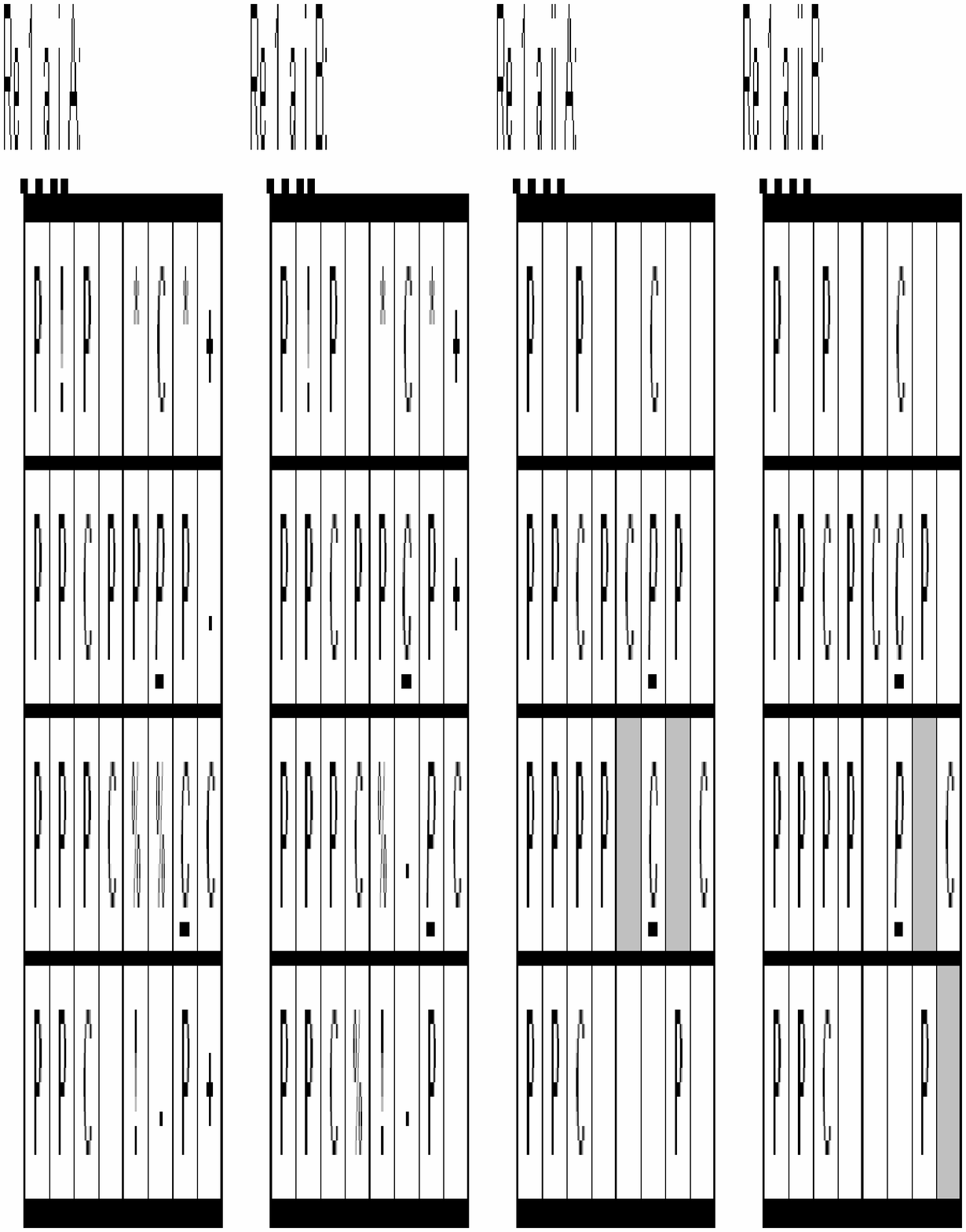}

\includegraphics[bb=0mm 0mm 208mm 296mm, width=120mm, height=16mm, viewport=3mm 4mm 205mm 292mm]{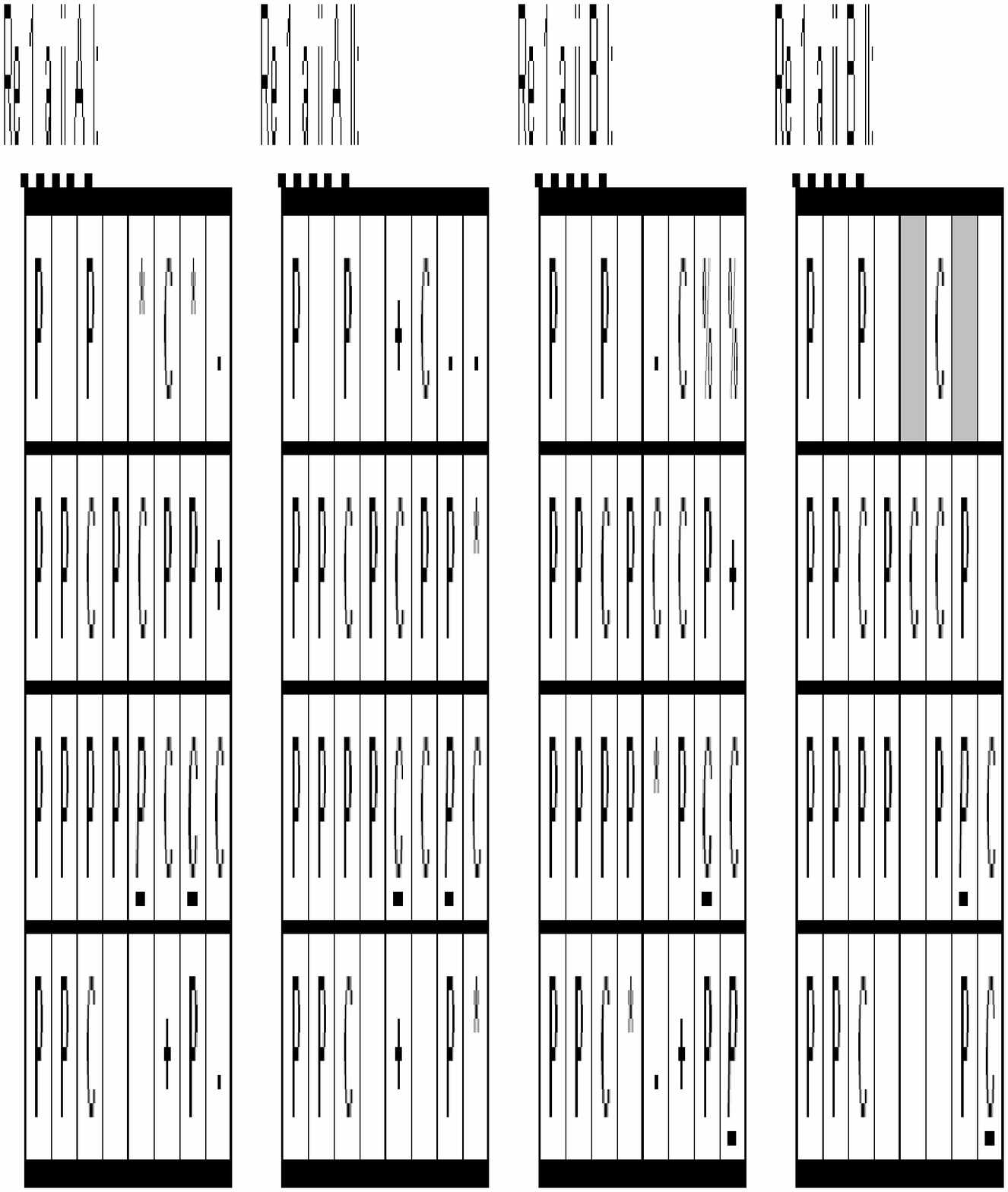}

\includegraphics[bb=0mm 0mm 208mm 296mm, width=120mm, height=15.5mm, viewport=3mm 4mm 205mm 292mm]{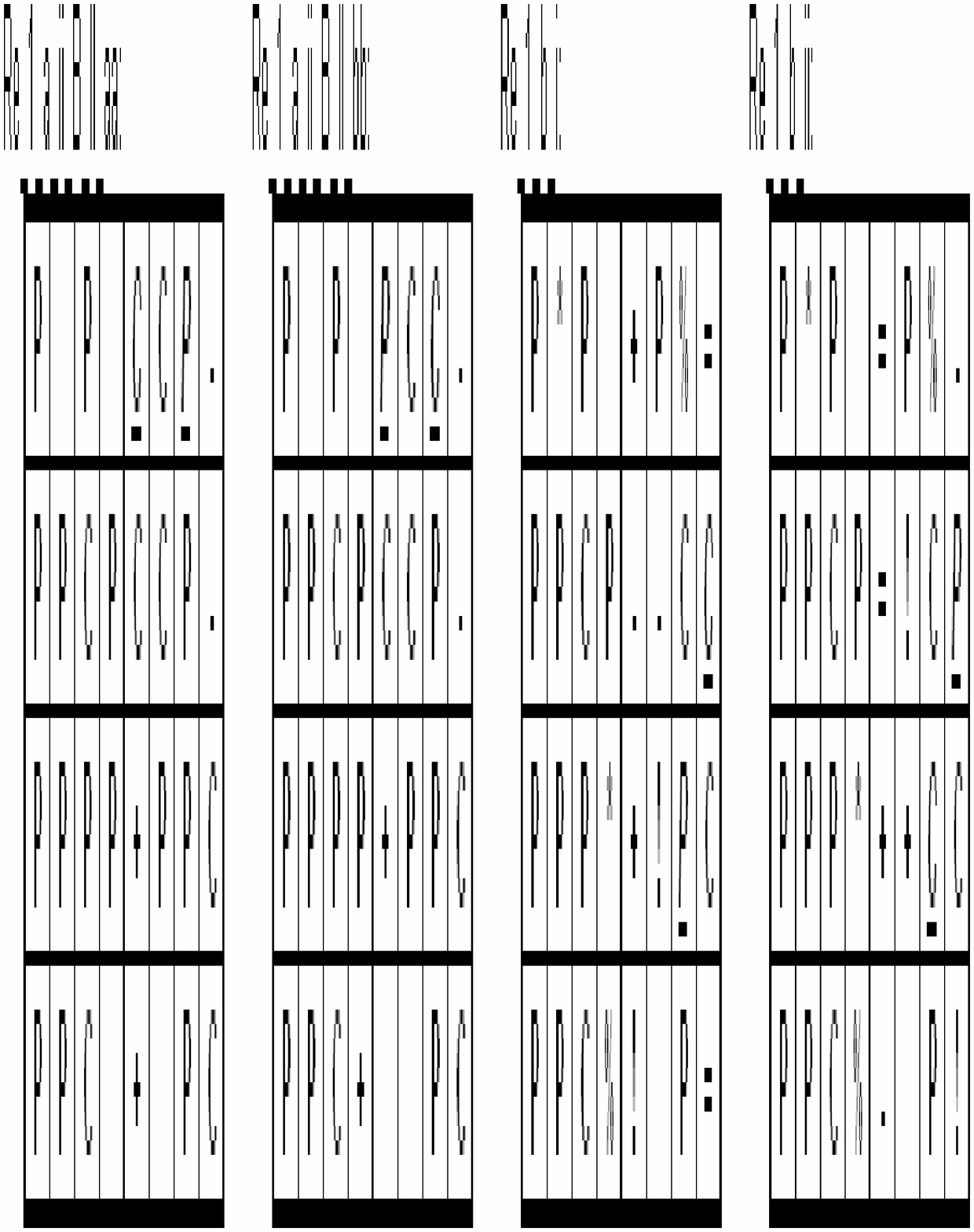}

\includegraphics[bb=0mm 0mm 208mm 296mm, width=120mm, height=15.5mm, viewport=3mm 4mm 205mm 292mm]{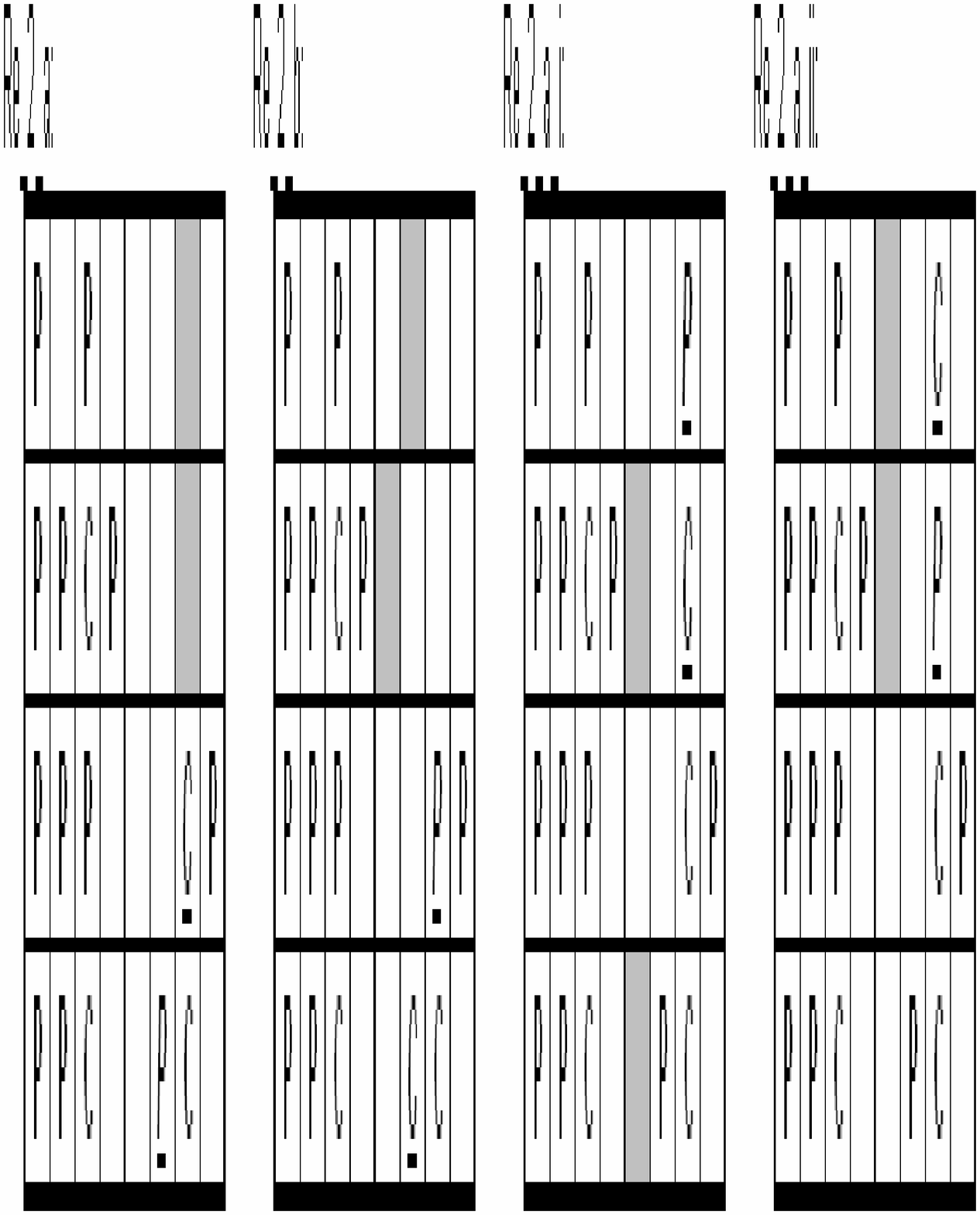}

\includegraphics[bb=0mm 0mm 208mm 296mm, width=120mm, height=16mm, viewport=3mm 4mm 205mm 292mm]{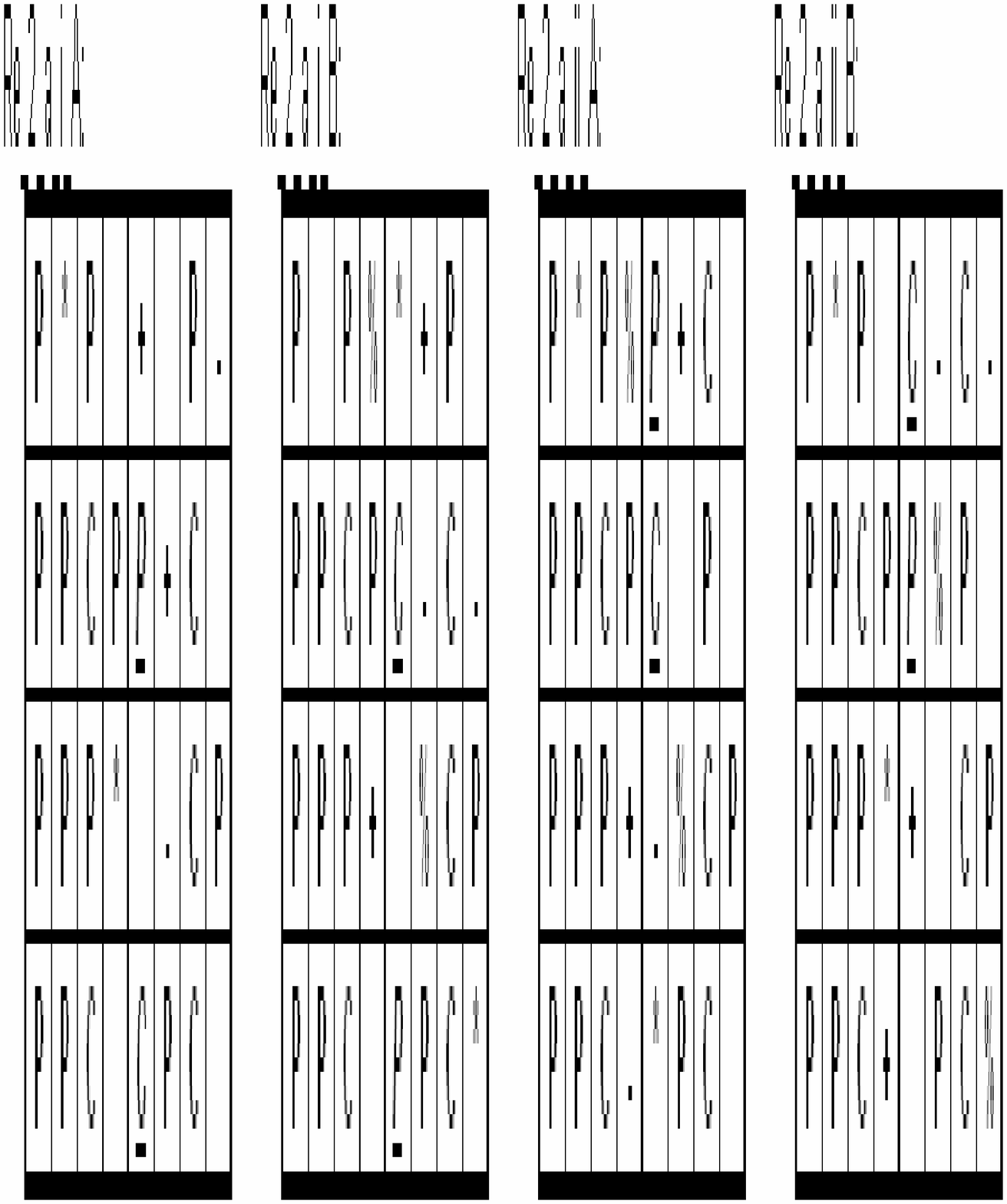}

\includegraphics[bb=0mm 0mm 208mm 296mm, width=120mm, height=16mm, viewport=3mm 4mm 205mm 292mm]{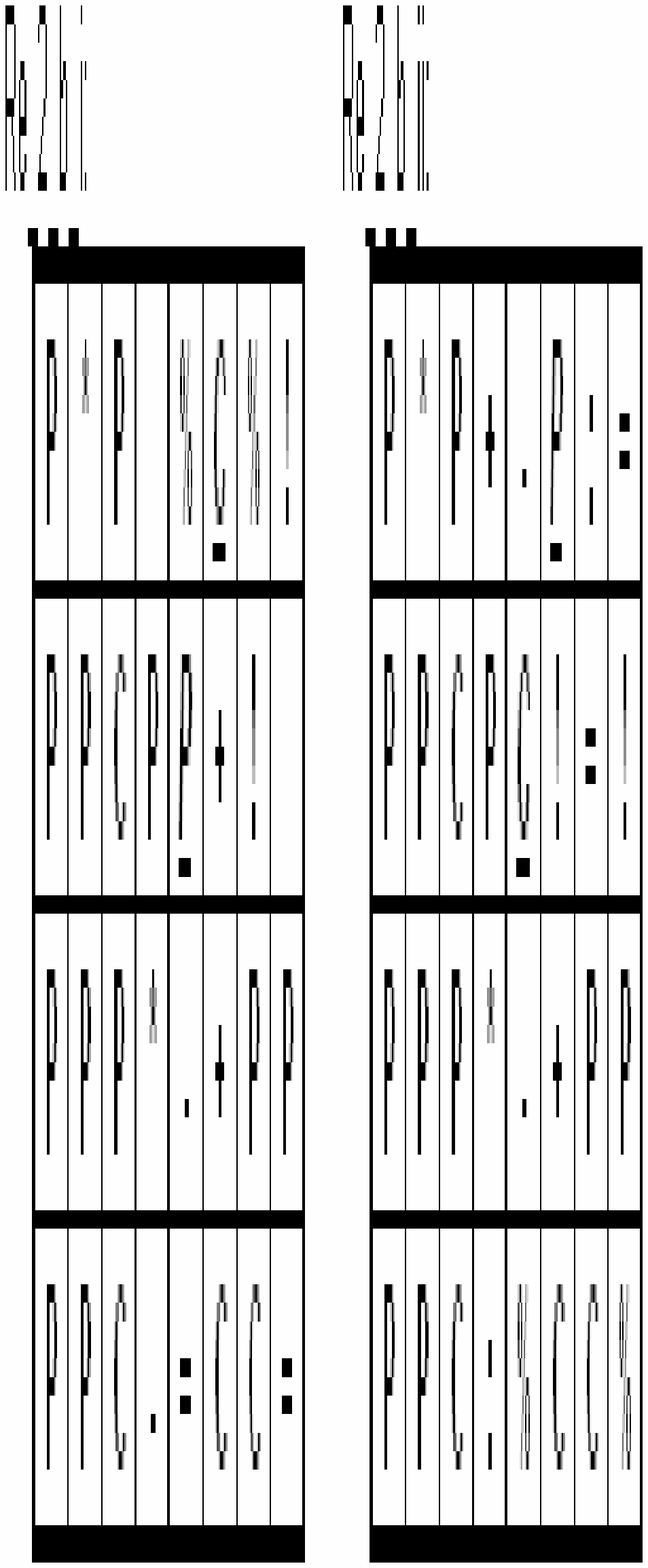}

\includegraphics[bb=0mm 0mm 208mm 296mm, width=120mm, height=20mm, viewport=3mm 4mm 205mm 292mm]{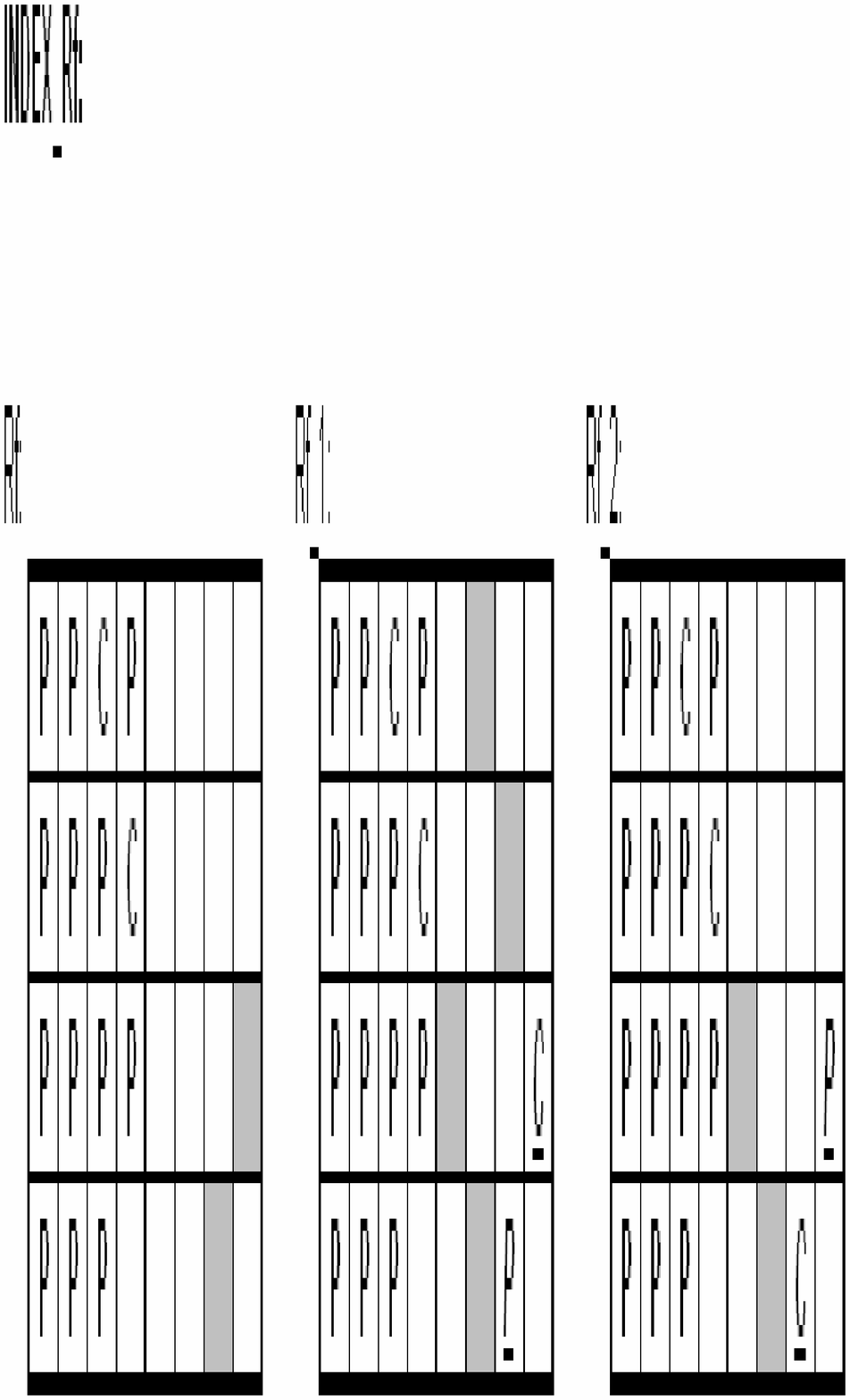}

\includegraphics[bb=0mm 0mm 208mm 296mm, width=120mm, height=15.5mm, viewport=3mm 4mm 205mm 292mm]{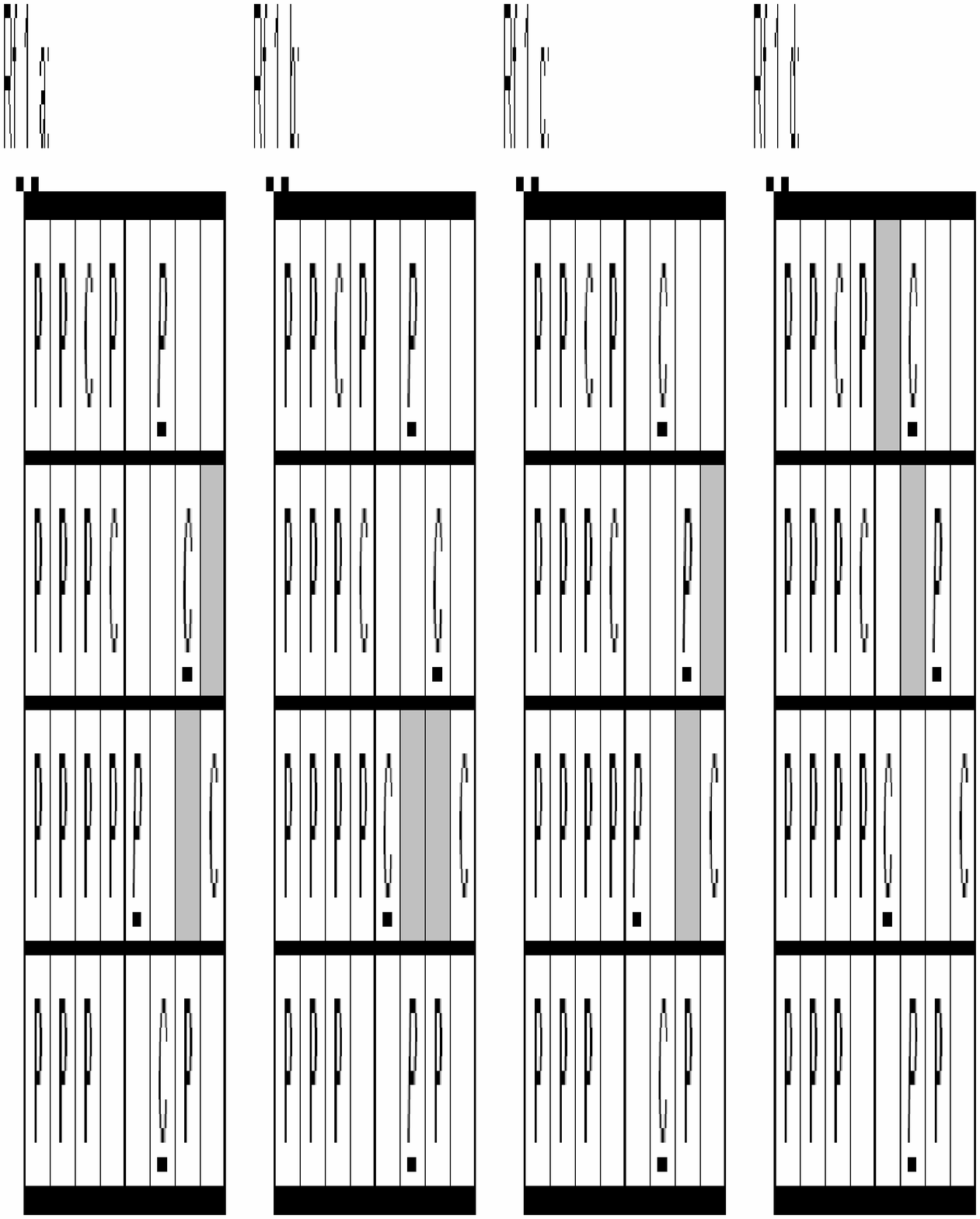}

\includegraphics[bb=0mm 0mm 208mm 296mm, width=120mm, height=15.5mm, viewport=3mm 4mm 205mm 292mm]{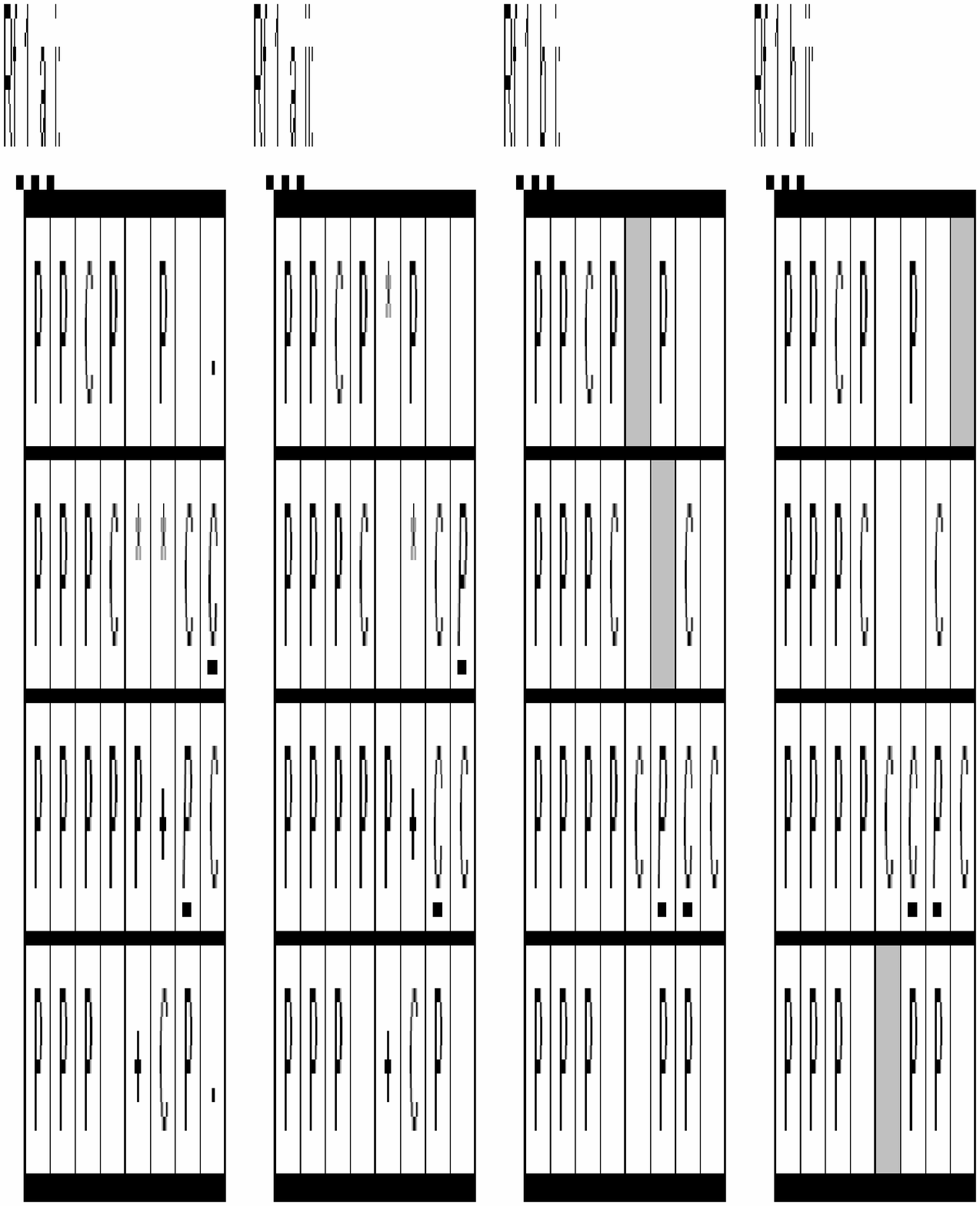}

\includegraphics[bb=0mm 0mm 208mm 296mm, width=120mm, height=15mm, viewport=3mm 4mm 205mm 292mm]{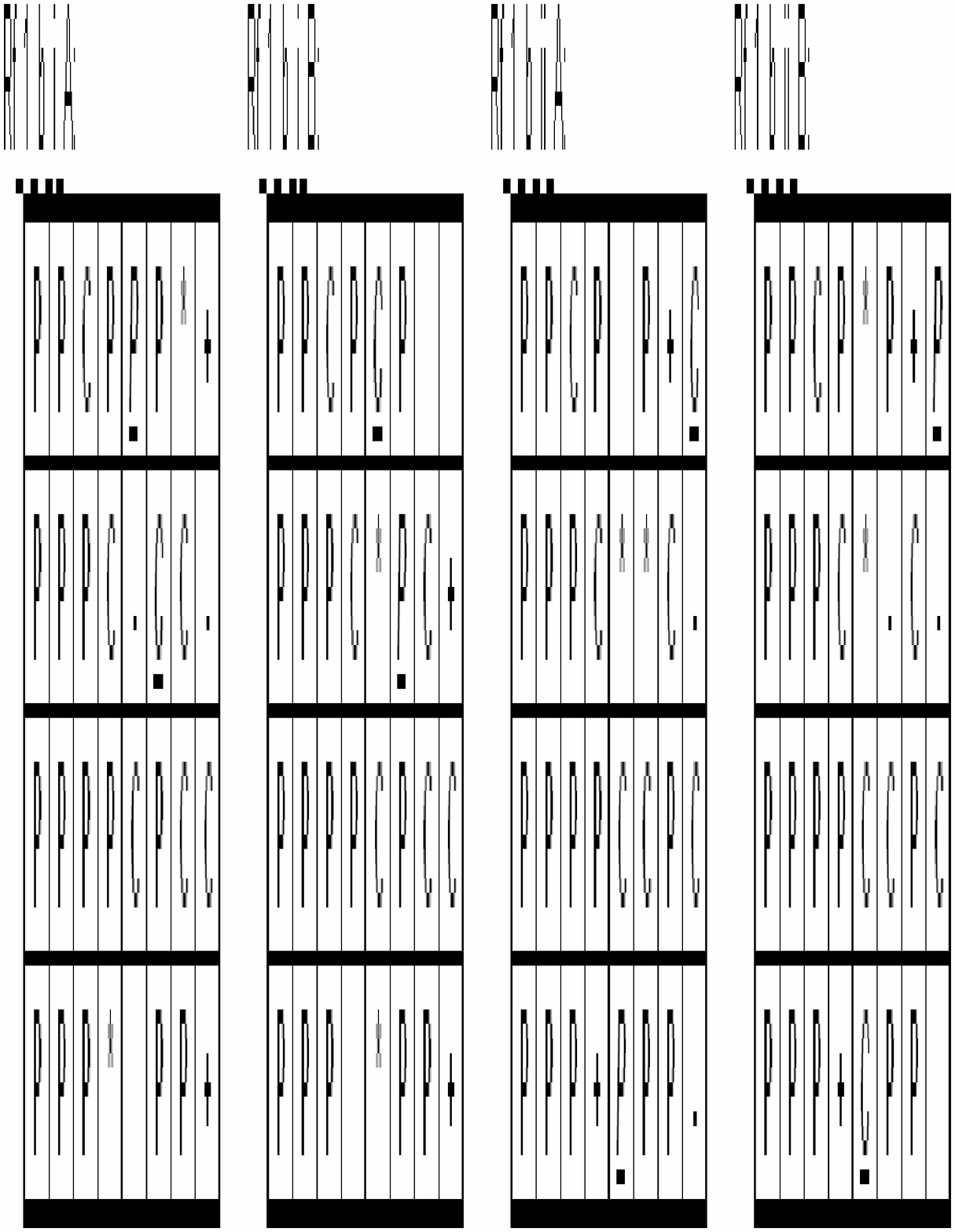}

\includegraphics[bb=0mm 0mm 208mm 296mm, width=120mm, height=16mm, viewport=3mm 4mm 205mm 292mm]{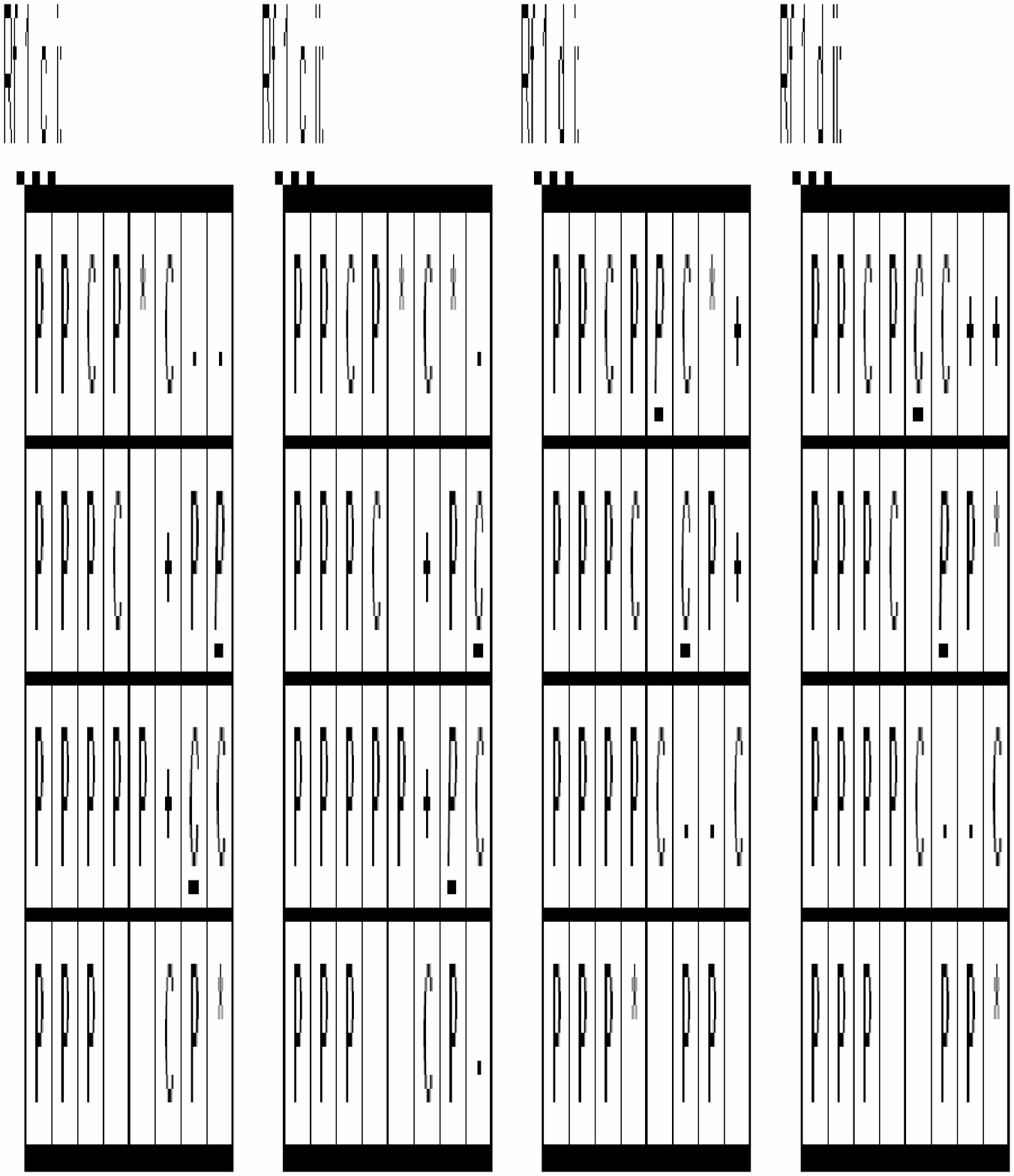}

\includegraphics[bb=0mm 0mm 208mm 296mm, width=120mm, height=15.5mm, viewport=3mm 4mm 205mm 292mm]{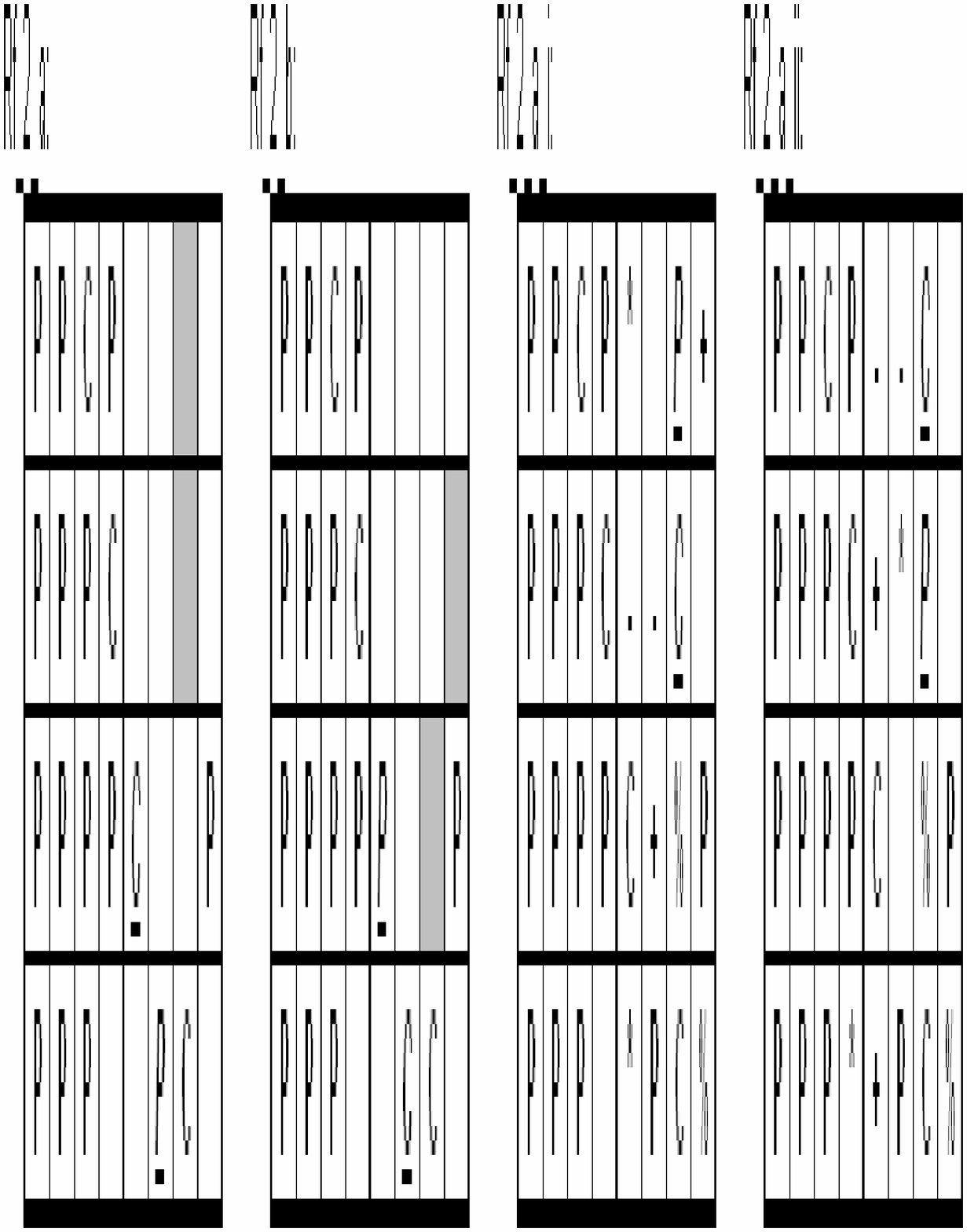}

\includegraphics[bb=0mm 0mm 208mm 296mm, width=120mm, height=15.5mm, viewport=3mm 4mm 205mm 292mm]{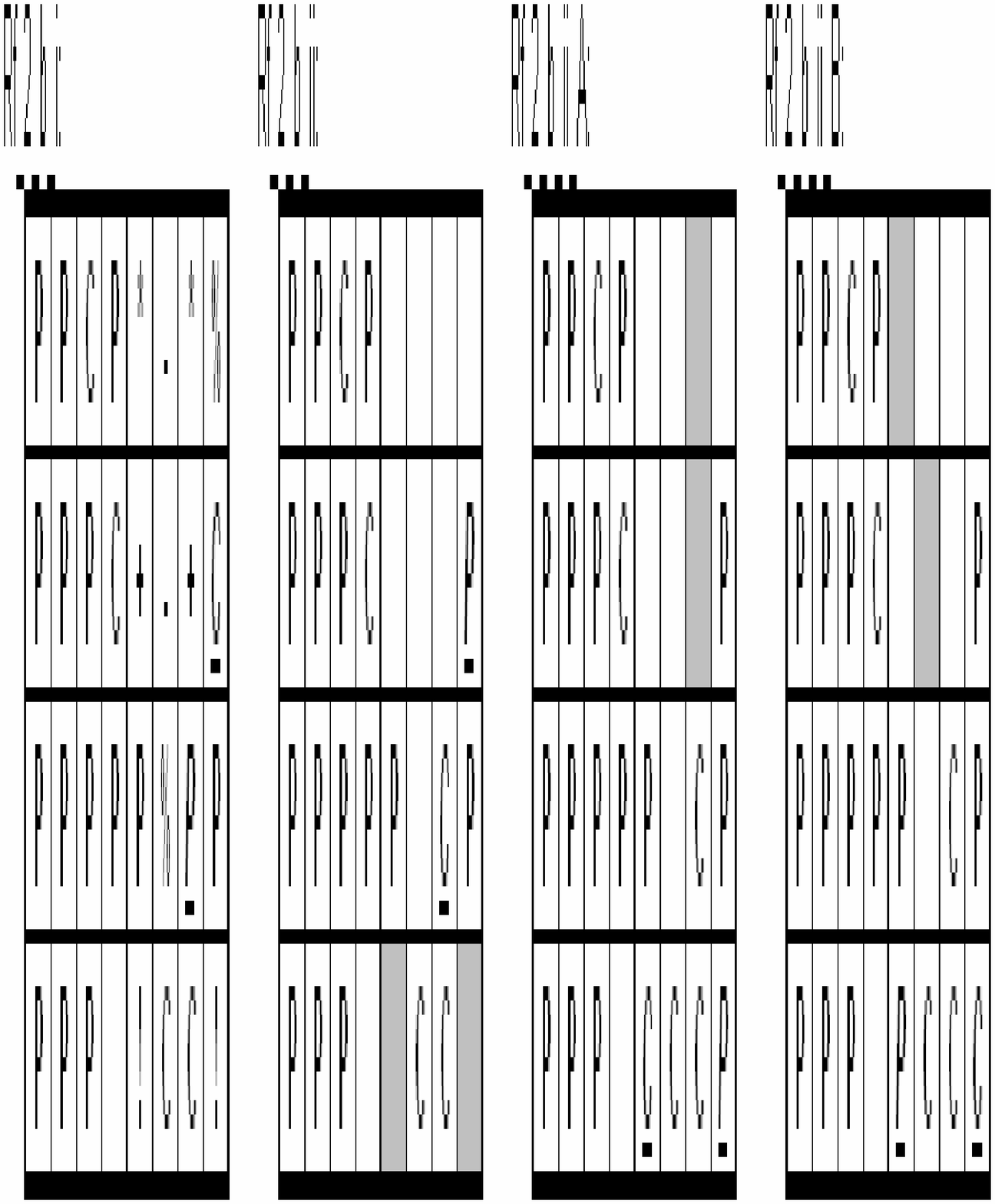}

\includegraphics[bb=0mm 0mm 208mm 296mm, width=120mm, height=15.5mm, viewport=3mm 4mm 205mm 292mm]{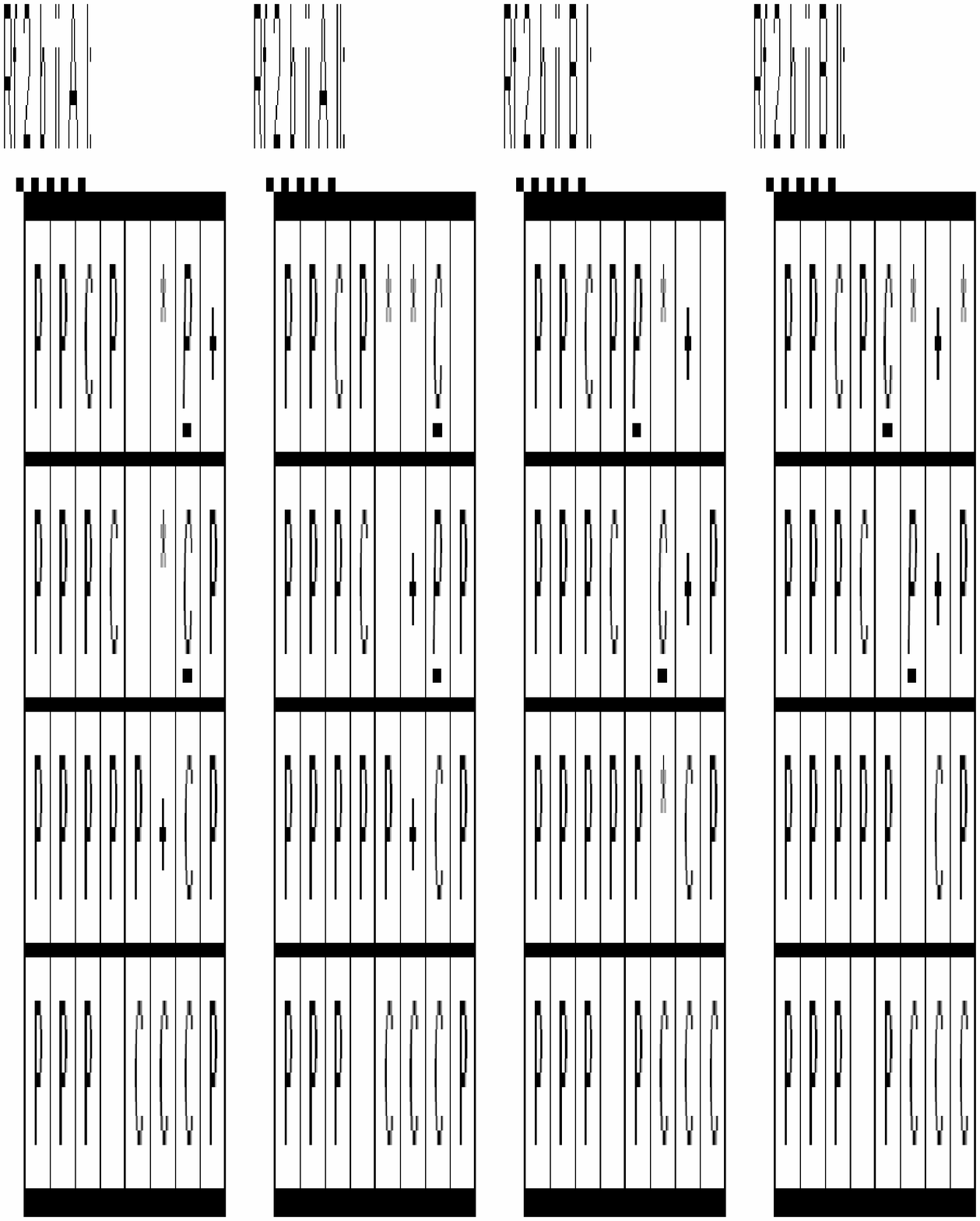}

\includegraphics[bb=0mm 0mm 208mm 296mm, width=120mm, height=19.5mm, viewport=3mm 4mm 205mm 292mm]{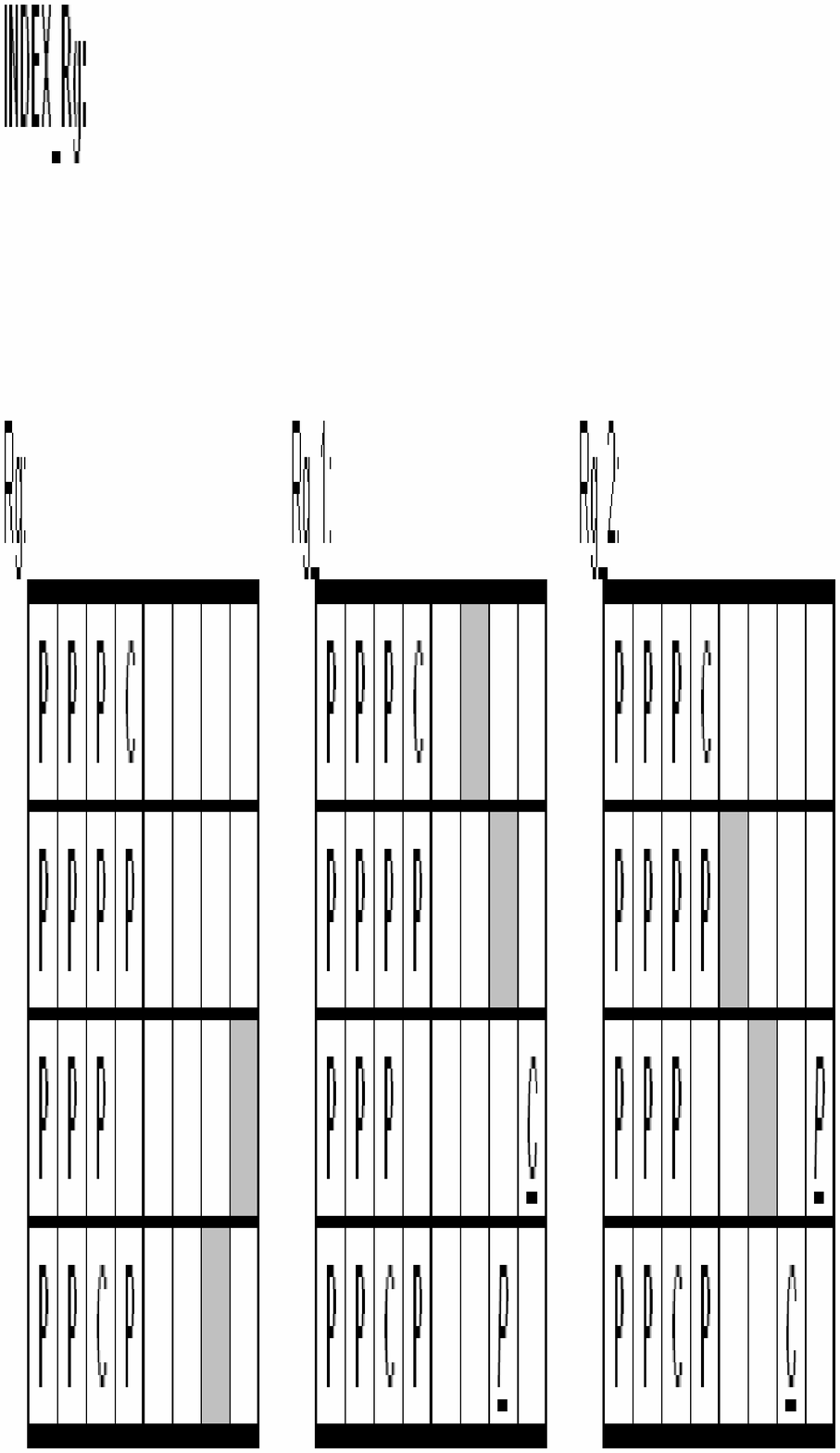}

\includegraphics[bb=0mm 0mm 208mm 296mm, width=120mm, height=15.5mm, viewport=3mm 4mm 205mm 292mm]{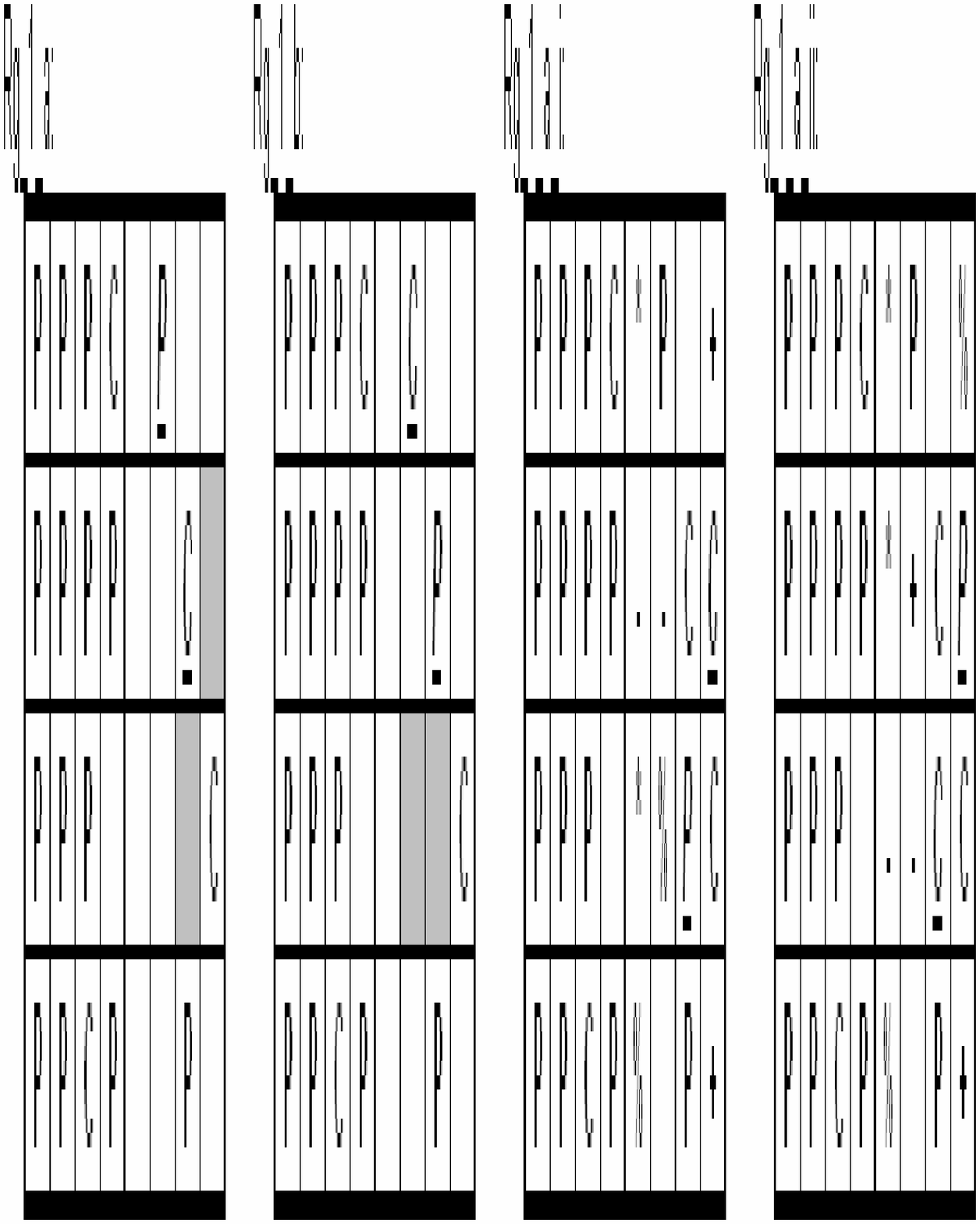}

\includegraphics[bb=0mm 0mm 208mm 296mm, width=120mm, height=15.5mm, viewport=3mm 4mm 205mm 292mm]{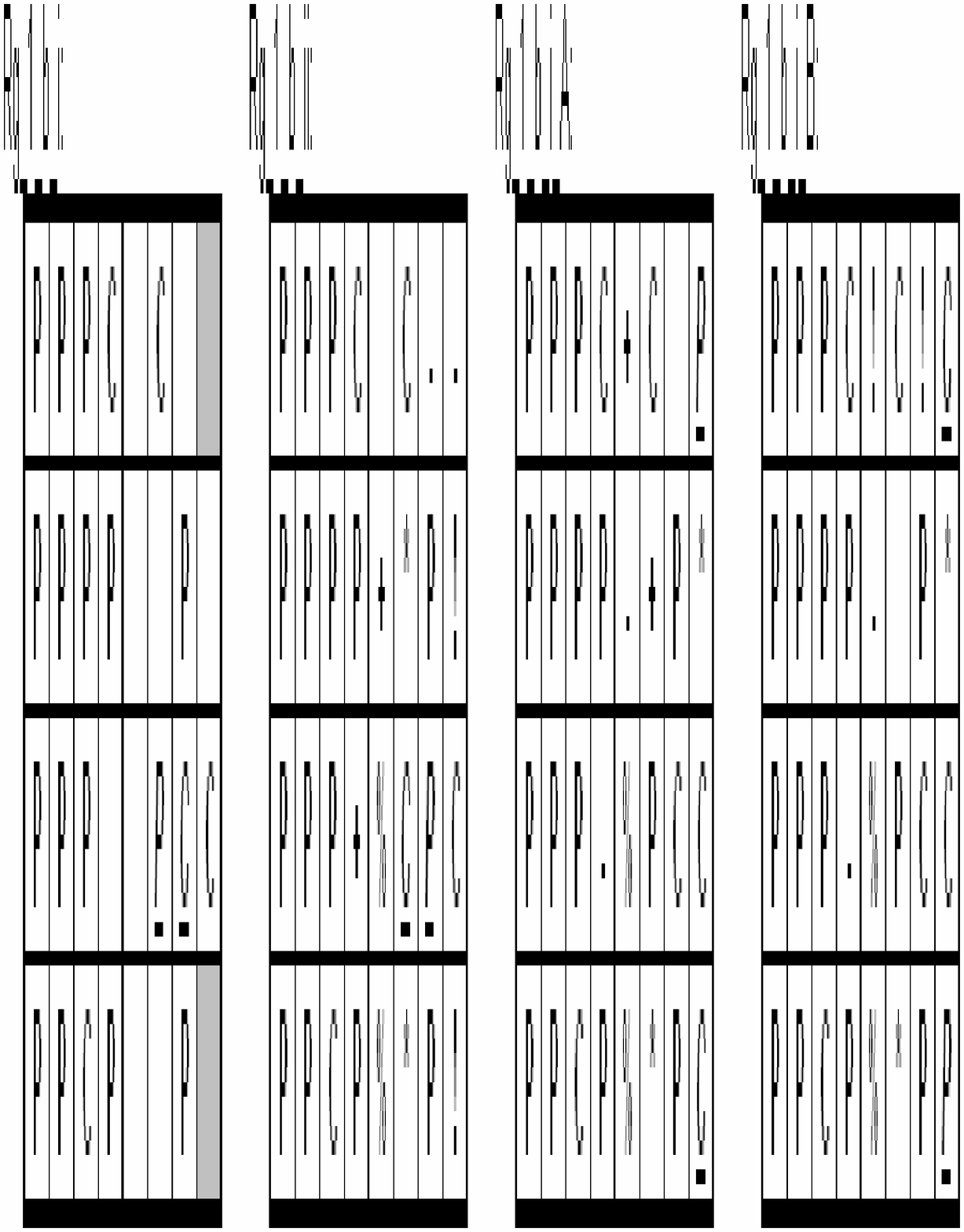}

\includegraphics[bb=0mm 0mm 208mm 296mm, width=120mm, height=15.5mm, viewport=3mm 4mm 205mm 292mm]{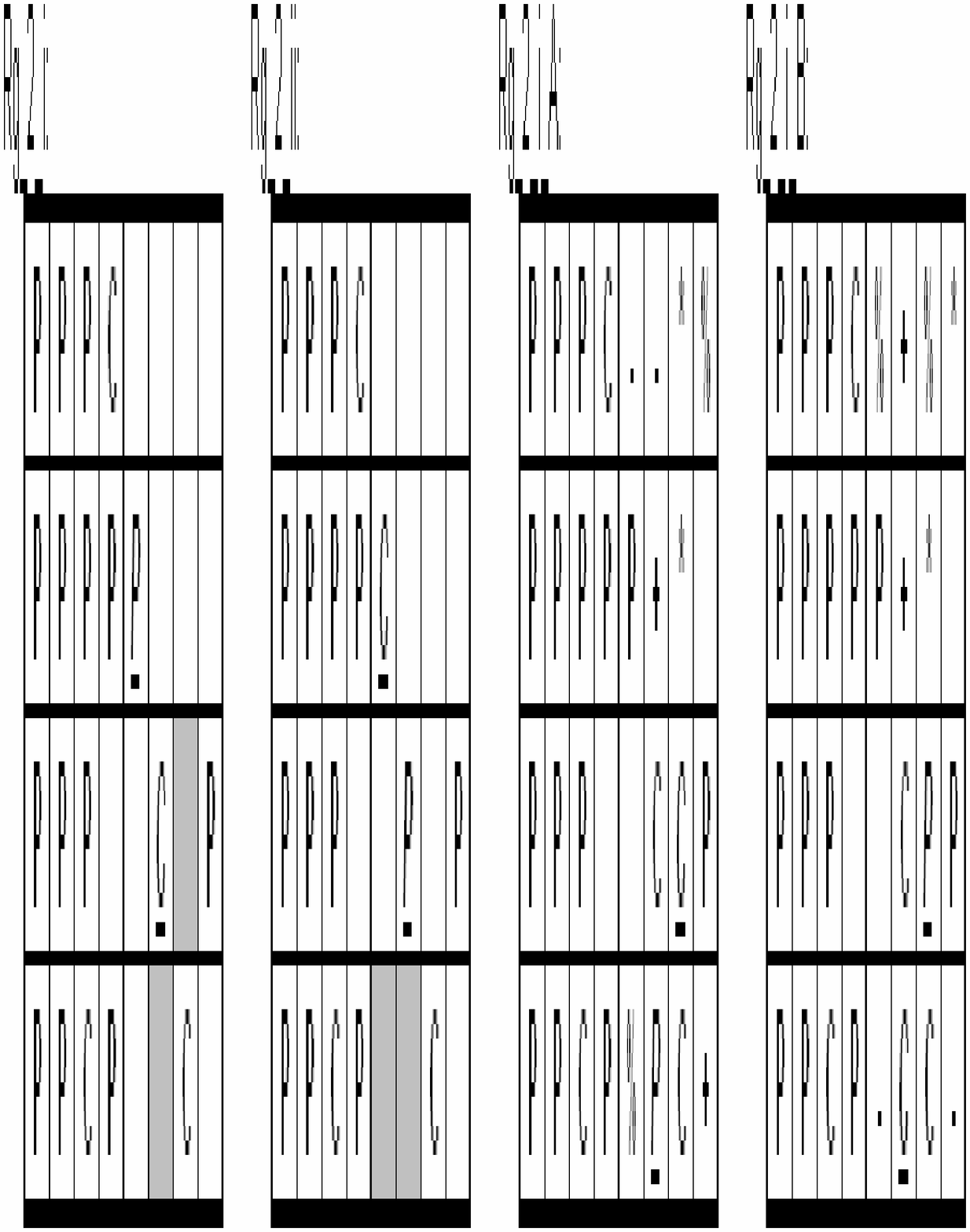}

\includegraphics[bb=0mm 0mm 208mm 296mm, width=120mm, height=15.5mm, viewport=3mm 4mm 205mm 292mm]{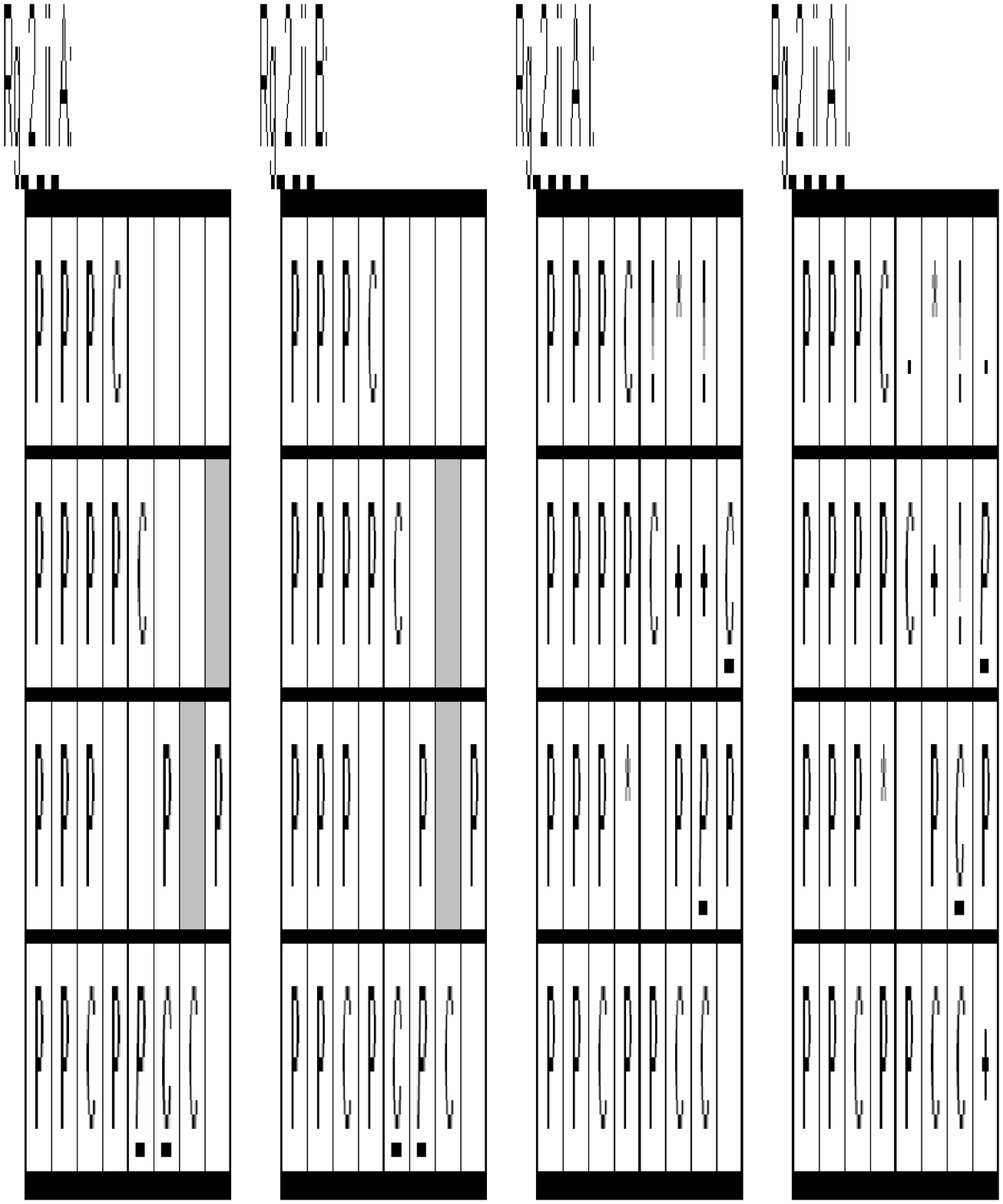}

\includegraphics[bb=0mm 0mm 208mm 296mm, width=120mm, height=15.5mm, viewport=3mm 4mm 205mm 292mm]{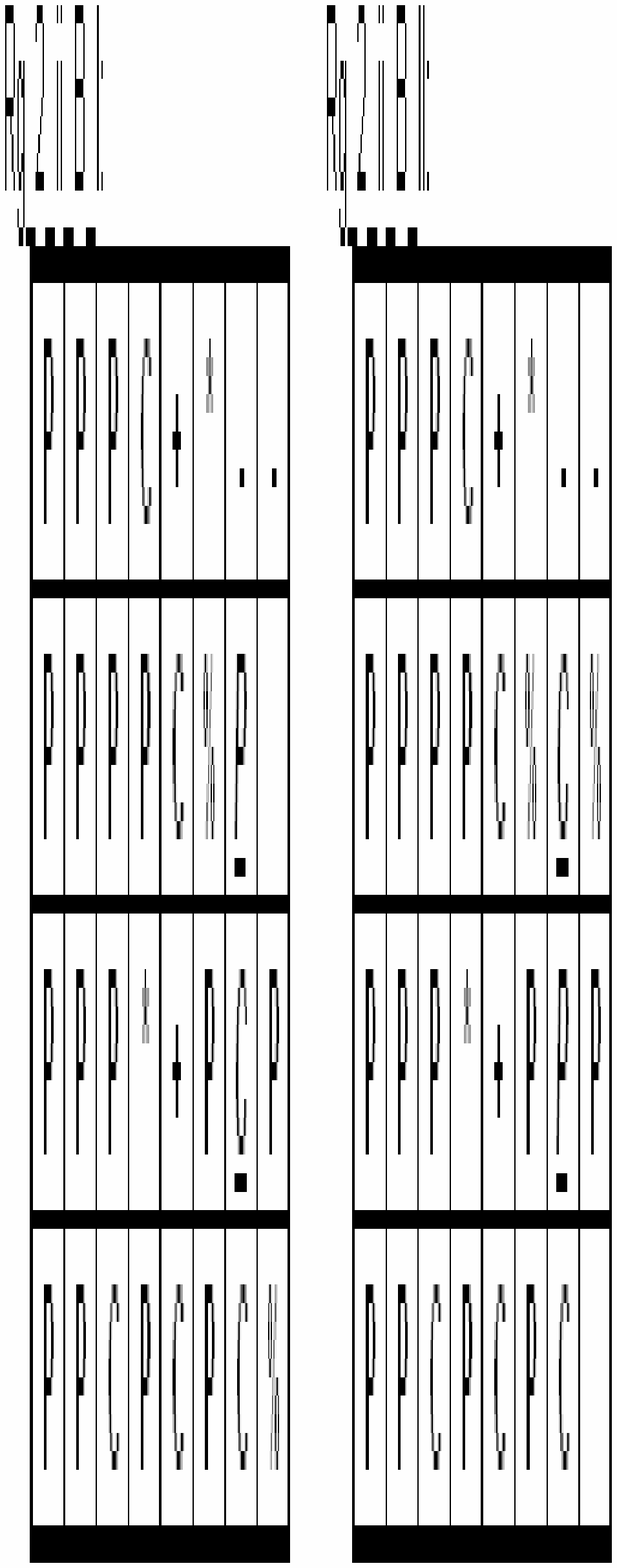}

\includegraphics[bb=0mm 0mm 208mm 296mm, width=120mm, height=16mm, viewport=3mm 4mm 205mm 292mm]{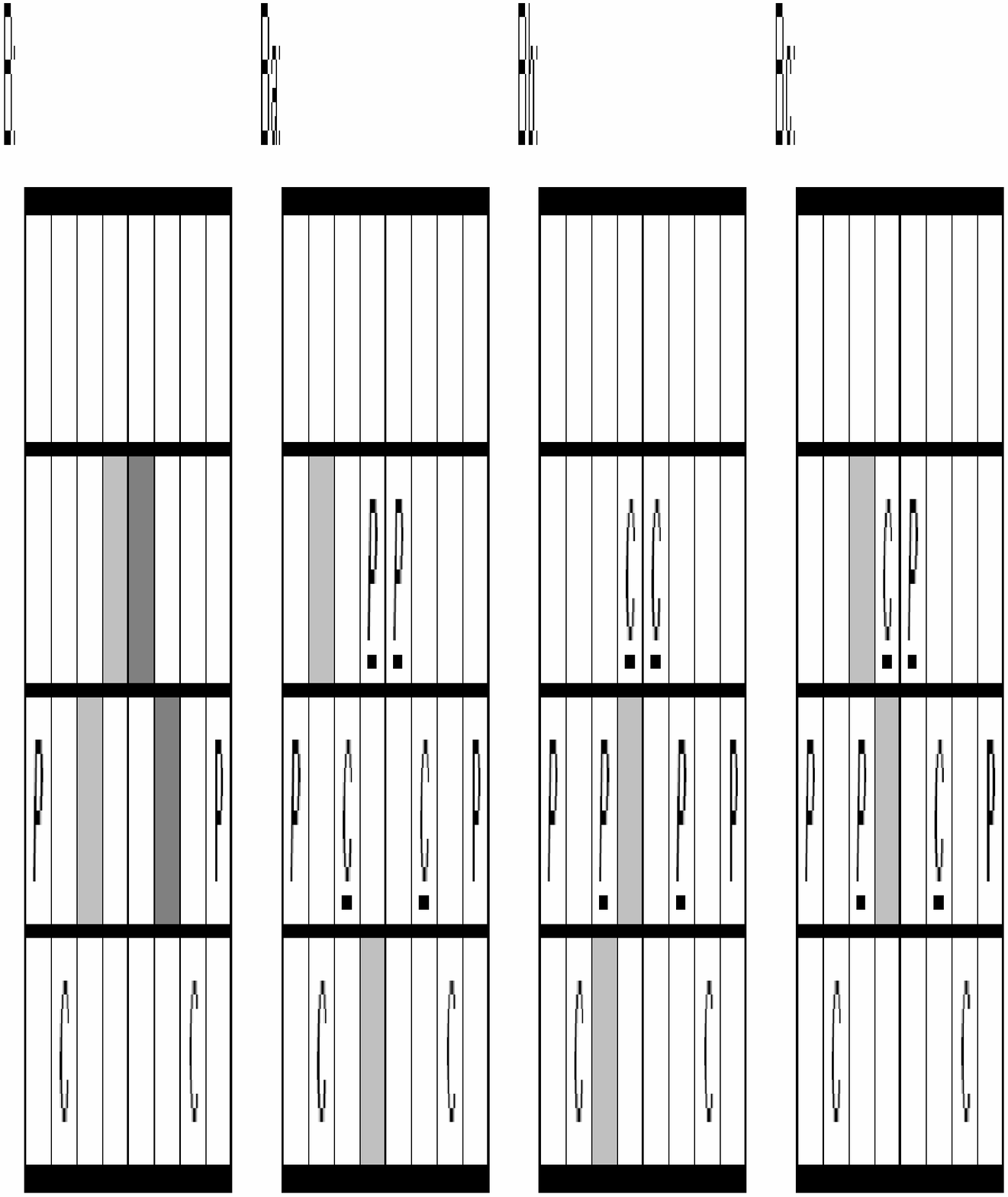}

\includegraphics[bb=0mm 0mm 208mm 296mm, width=120mm, height=16mm, viewport=3mm 4mm 205mm 292mm]{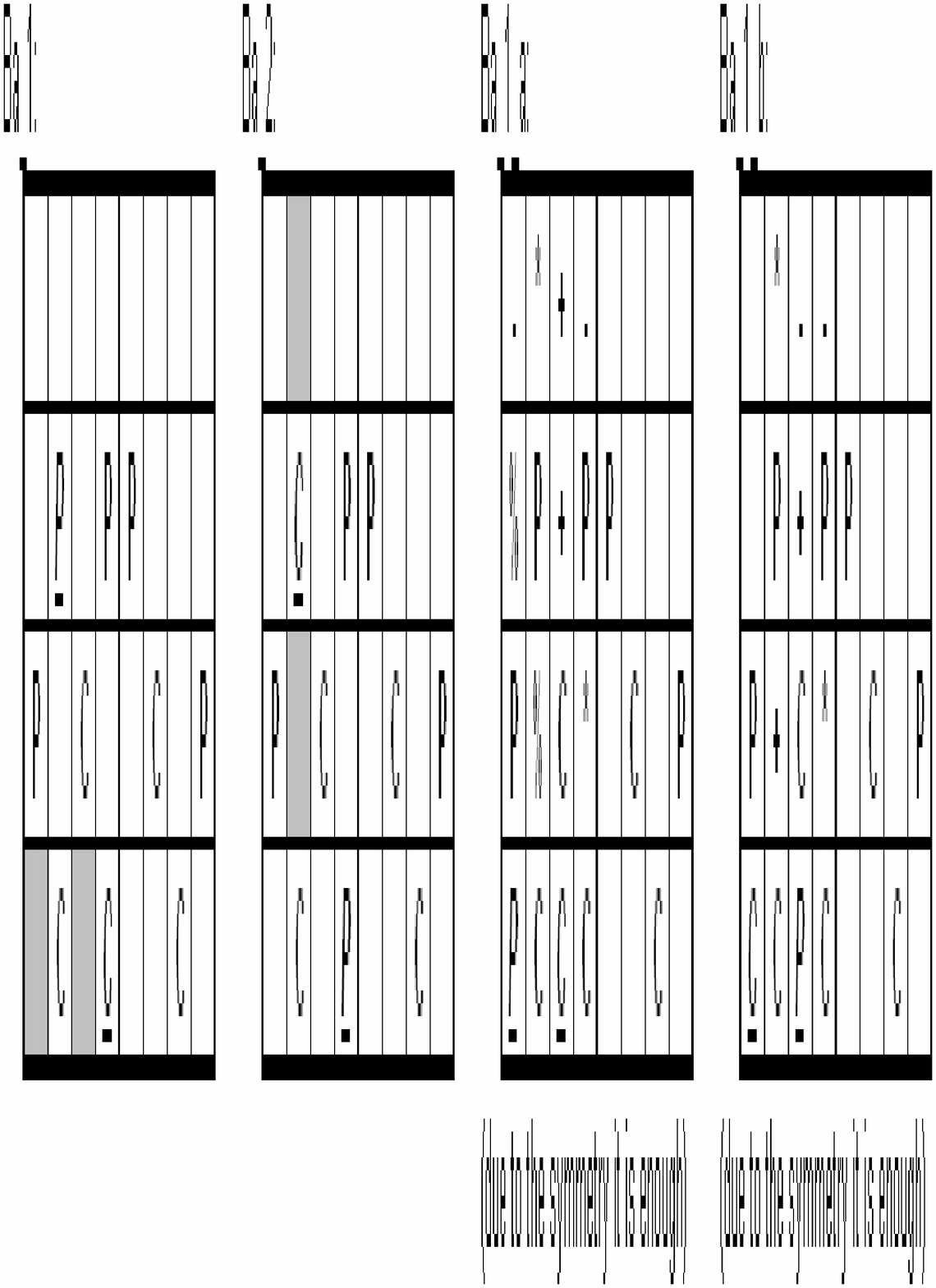}

\includegraphics[bb=0mm 0mm 208mm 296mm, width=120mm, height=15.6mm, viewport=3mm 4mm 205mm 292mm]{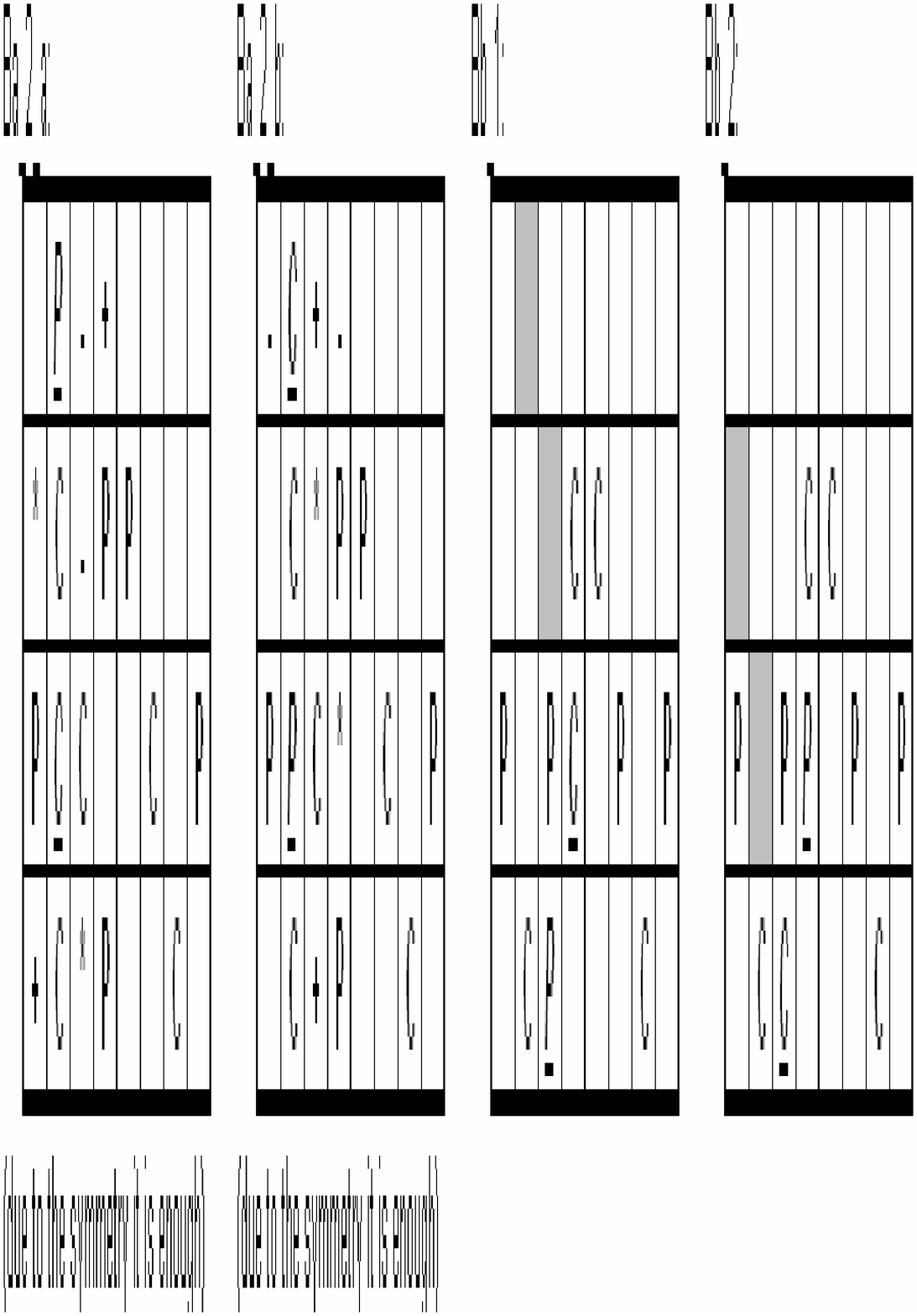}

\includegraphics[bb=0mm 0mm 208mm 296mm, width=120mm, height=16mm, viewport=3mm 4mm 205mm 292mm]{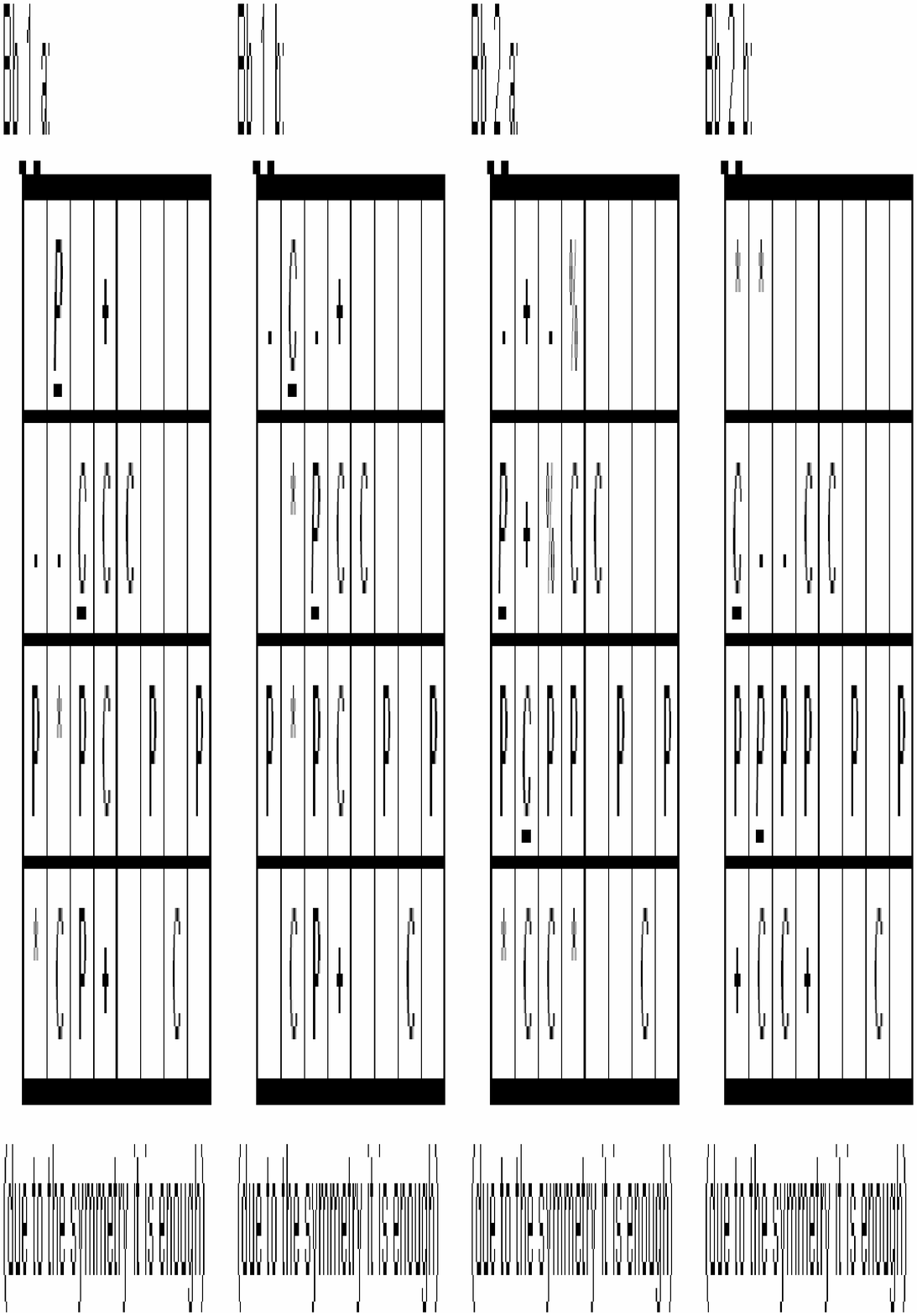}

\includegraphics[bb=0mm 0mm 208mm 296mm, width=120mm, height=16mm, viewport=3mm 4mm 205mm 292mm]{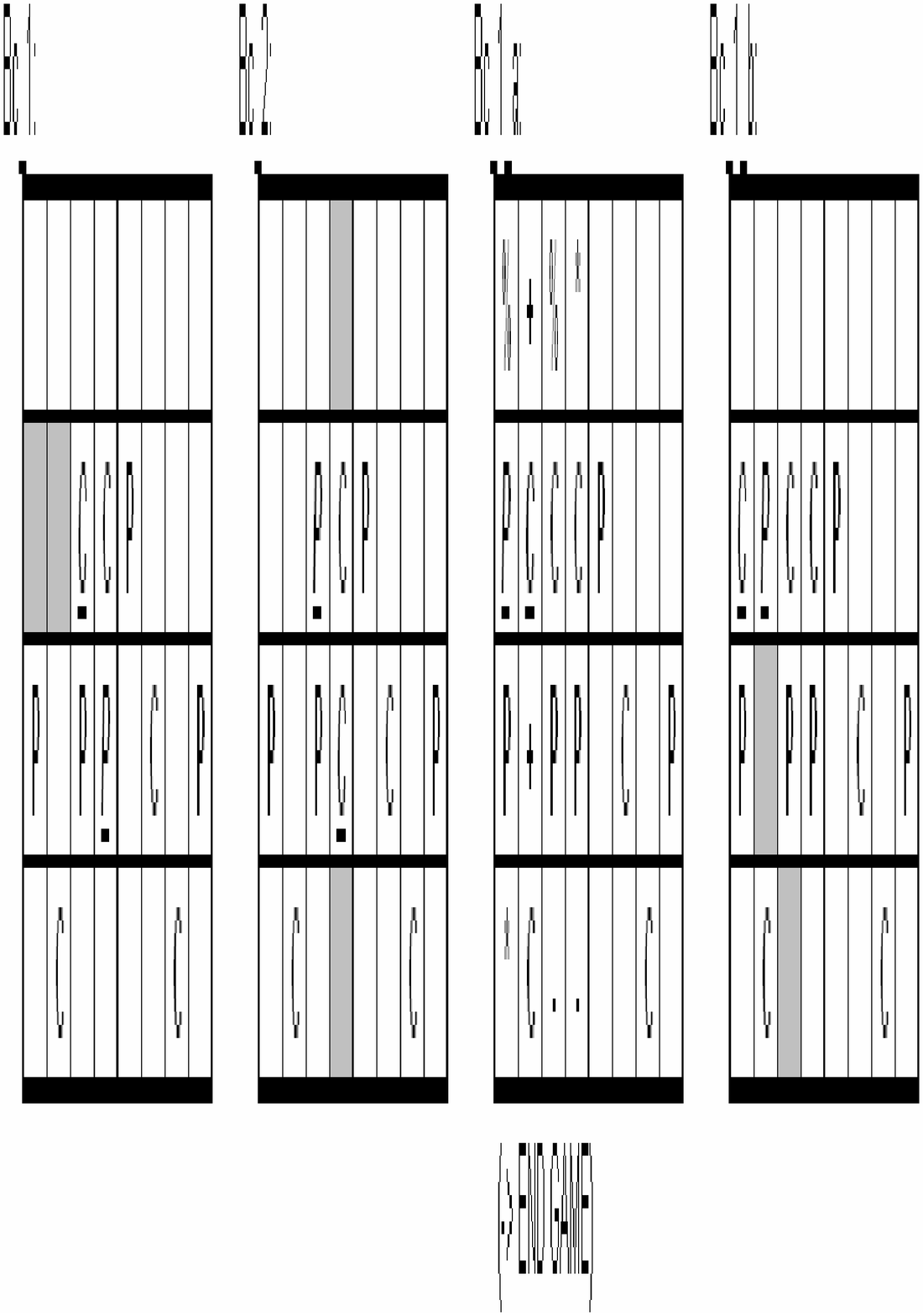}

\includegraphics[bb=0mm 0mm 208mm 296mm, width=120mm, height=16.3mm, viewport=3mm 4mm 205mm 292mm]{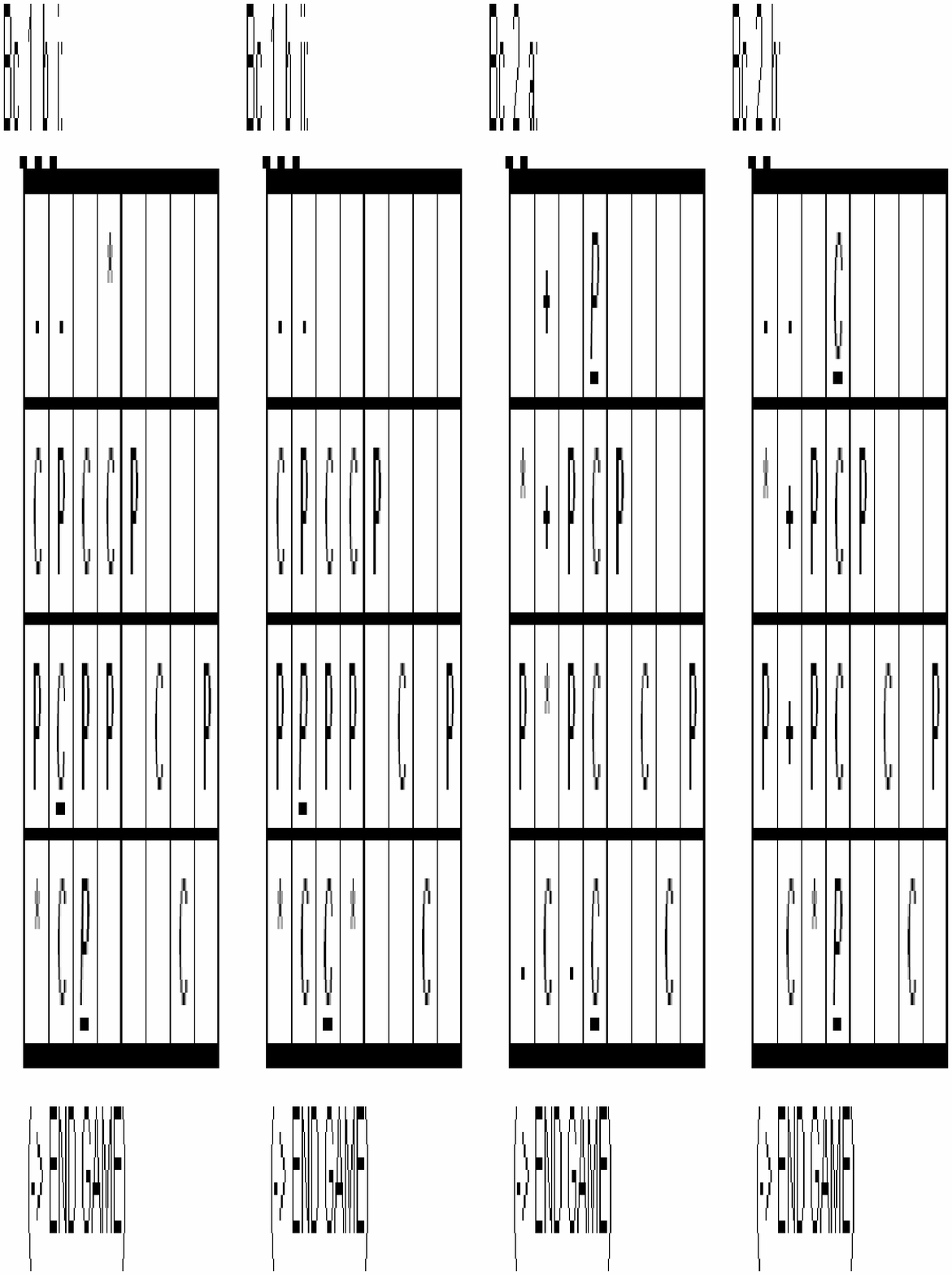}

\includegraphics[bb=0mm 0mm 208mm 296mm, width=120mm, height=17.5mm, viewport=3mm 4mm 205mm 292mm]{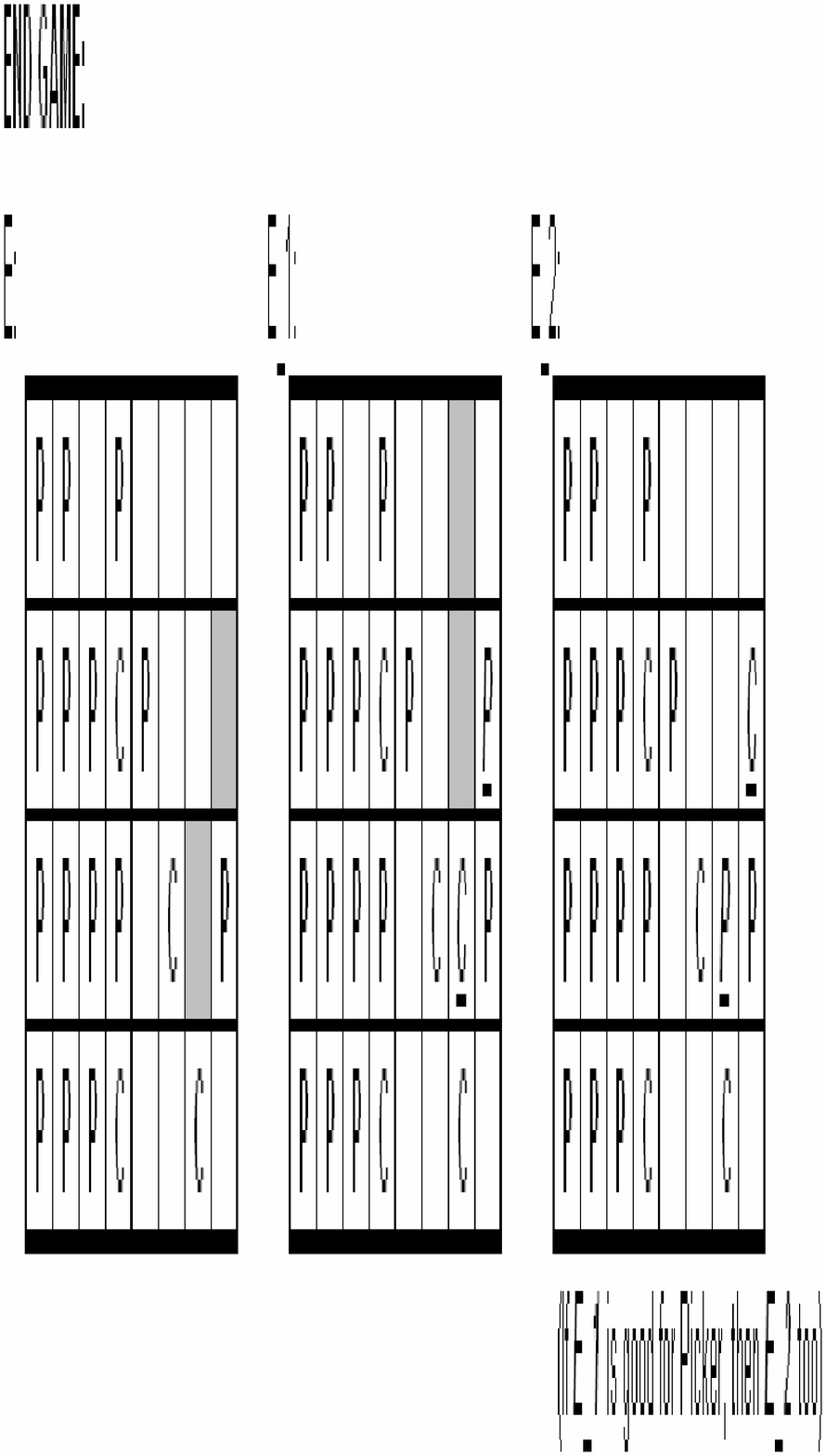}

\includegraphics[bb=0mm 0mm 208mm 296mm, width=120mm, height=16mm, viewport=3mm 4mm 205mm 292mm]{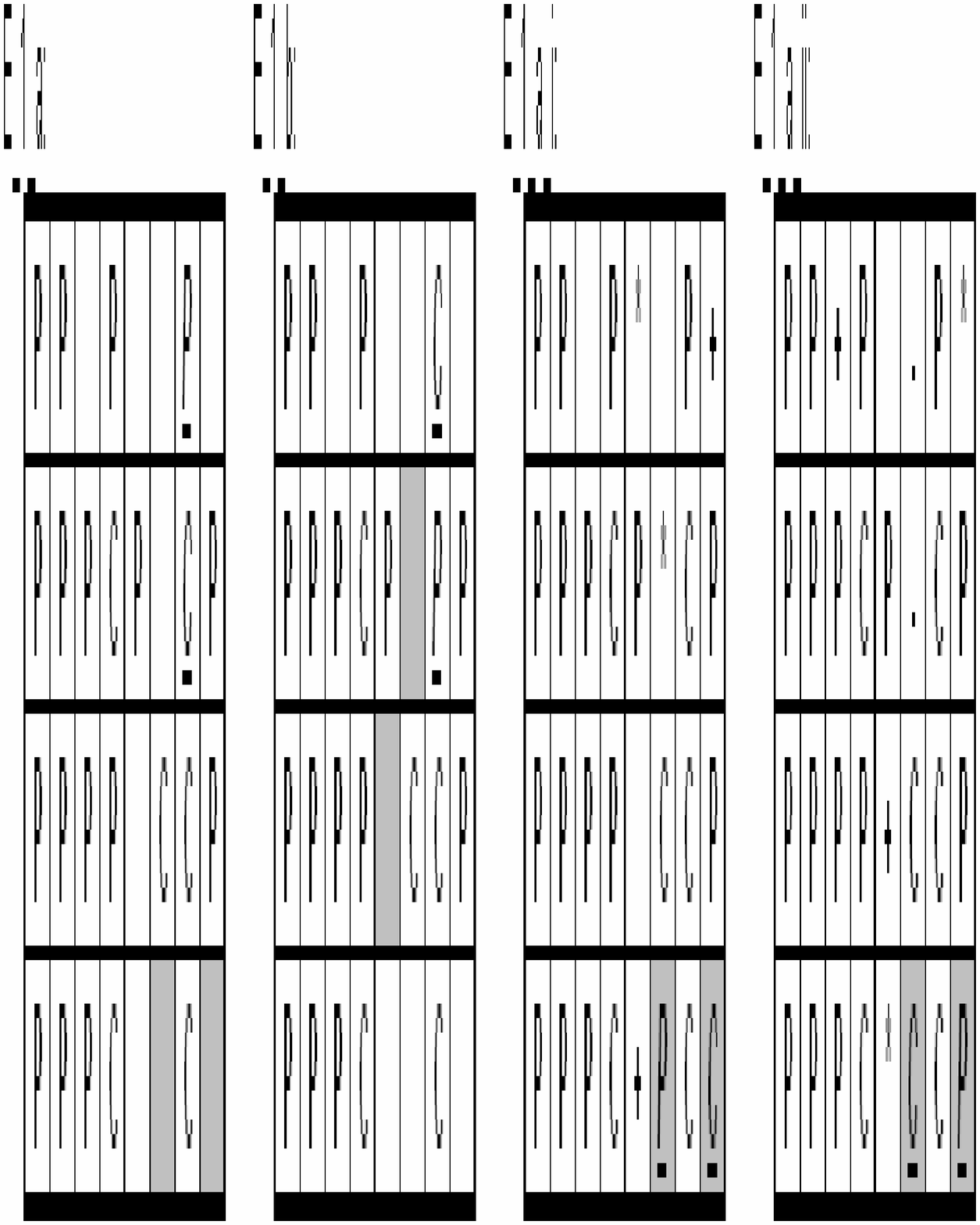}

\includegraphics[bb=0mm 0mm 208mm 296mm, width=120mm, height=16mm, viewport=3mm 4mm 205mm 292mm]{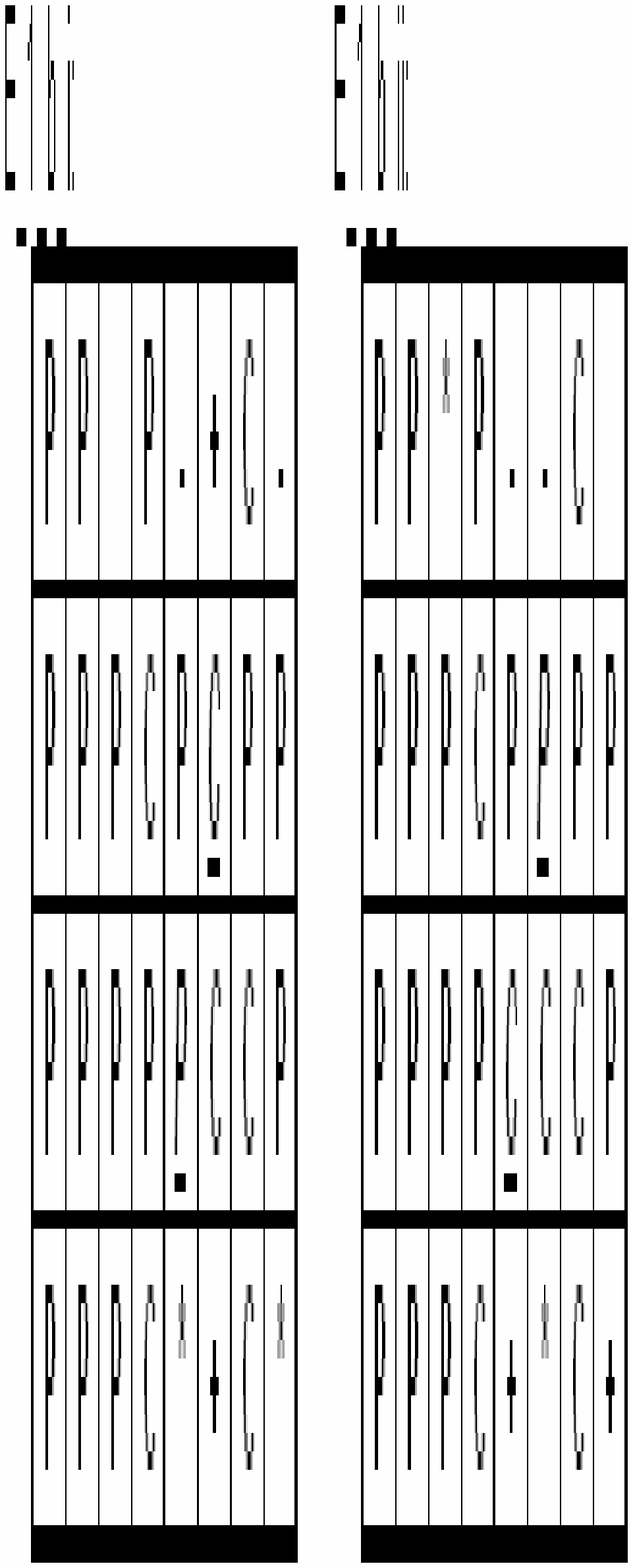}

\end{center}